\documentclass[10pt,twoside]{article}
\usepackage[utf8]{inputenc}

\title{\uppercase{\textbf{\large{Motivic cohomology of mixed characteristic schemes}}}}

\author{\textsc{tess bouis}}

\date{}
\usepackage[left=3cm, right=3cm, top=1in, bottom=1in, headheight=2cm]{geometry}

\usepackage{fancyhdr}
\usepackage{stmaryrd}

\usepackage{mathdots}

\usepackage{amsmath, amscd, amsfonts, amssymb, amsthm,todonotes,eurosym,enumitem,relsize}

\usepackage[utf8]{inputenc}
\usepackage{dirtytalk}
\usepackage[T1]{fontenc}
\usepackage{mathrsfs}

\usepackage[colorlinks]{hyperref}
\usepackage[figure,table]{hypcap}
\definecolor{imperialred}{RGB}{237, 41, 57}
\definecolor{royalblue}{RGB}{64, 106, 212}
\definecolor{link}{RGB}{11,0,128}
\definecolor{gren}{RGB}{32,130,63}
\hypersetup{
	bookmarksnumbered,
	pdfstartview={FitH},
	citecolor=royalblue,
	linkcolor=imperialred,
	urlcolor=link,
	linktocpage = true
	pdfpagemode={UseOutlines}
}

\usepackage{xcolor}
\usepackage{comment}

\usepackage[object=pgfhan]{pgfornament}

\usepackage{graphicx}
\usepackage{array}
\usepackage{csquotes}
\usepackage{extarrows}

\usepackage{needspace}

\input{xy}
\xyoption{all}
\usepackage{tikz-cd}
\usepackage{textcomp}
\usepackage{colonequals}

\usepackage{fancyhdr}

\newlength{\outermargin} 
\newlength{\mar} 
\newlength{\len}
\newlength{\temp}
\newtheorem{theorem}{Theorem}[section]
\newtheorem{lemma}[theorem]{Lemma}
\newtheorem{corollary}[theorem]{Corollary}
\newtheorem{proposition}[theorem]{Proposition}

\newtheorem{theoremintro}{Theorem}

\newtheorem{corollaryintro}[theoremintro]{Corollary}

\theoremstyle{definition}
\newtheorem{remark}[theorem]{Remark}

\newtheorem{example}[theorem]{Example}

\newtheorem{construction}[theorem]{Construction}
\newtheorem{definition}[theorem]{Definition}
\newtheorem{notation}[theorem]{Notation}
\newtheorem{question}[theorem]{Question}

\DeclareMathOperator{\Z}{\mathbb{Z}}
\DeclareMathOperator{\Q}{\mathbb{Q}}

\DeclareMathOperator{\N}{\mathbb{N}}
\DeclareMathOperator{\F}{\mathbb{F}}

\DeclareMathOperator{\E}{\mathbb{E}}

\renewcommand{\epsilon}{\varepsilon}

\DeclareFontFamily{U}{MnSymbolC}{}
\DeclareFontShape{U}{MnSymbolC}{m}{n}{
	<-5.5> MnSymbolC5
	<5.5-6.5> MnSymbolC6
	<6.5-7.5> MnSymbolC7
	<7.5-8.5> MnSymbolC8
	<8.5-9.5> MnSymbolC9
	<9.5-11.5> MnSymbolC10
	<11.5-> MnSymbolCb12
}{}

\usepackage[bbgreekl]{mathbbol}
\DeclareSymbolFontAlphabet{\mathbb}{AMSb}
\DeclareSymbolFontAlphabet{\mathbbl}{bbold}
\newcommand{\Prism}{{\mathlarger{\mathbbl{\Delta}}}}

\usepackage{soul}
\usepackage{enumitem}
\numberwithin{equation}{theorem}

\begin{document}
	
	\maketitle
	
	\pagestyle{fancy}
	\fancyhead[EC]{TESS BOUIS}
	\fancyhead[OC]{\uppercase{MOTIVIC COHOMOLOGY OF MIXED CHARACTERISTIC SCHEMES}}
	\fancyfoot[C]{\thepage}
	
	\begin{abstract}
		We introduce a theory of motivic cohomology for quasi-compact quasi-separated schemes, which generalises the construction of Elmanto--Morrow in the case of schemes over a field. Our construction is non-$\mathbb{A}^1$-invariant in general, but it uses the classical $\mathbb{A}^1$-invariant motivic cohomology of smooth $\mathbb{Z}$-schemes as an input. The main new input of our construction is a global filtration on topological cyclic homology, whose graded pieces provide an integral refinement of derived de Rham cohomology and Bhatt--Morrow--Scholze’s syntomic cohomology. Our theory satisfies various expected properties of motivic cohomology, including relations to étale cohomology and to non-connective algebraic $K$-theory.
	\end{abstract}

	{
		\hypersetup{linkcolor=black}
		\tableofcontents
	}
	
	\section{Introduction}
	
	\vspace{-\parindent}
	\hspace{\parindent}
	
	Motivic cohomology is an analogue in algebraic geometry of singular cohomology. It was first envisioned to exist for schemes $X$ of finite type over $\Z$ by Beilinson and Lichtenbaum \cite{lichtenbaum_values_1973,lichtenbaum_values_1984,beilinson_notes_1986,beilinson_height_1987,beilinson_notes_1987}, as a way to better understand the special values of their $L$-functions. Motivic cohomology, in the form of complexes of abelian groups $\Z(i)^{\text{mot}}(X)$ indexed by integers $i \geq 0$, should be an integral interpolation between étale cohomology, and Adams eigenspaces on rationalised algebraic $K$-theory. That is, there should be a natural filtration $\text{Fil}^\star_{\text{mot}} \text{K}(X)$ on the non-connective algebraic $K$-theory $\text{K}(X)$, which splits rationally, and whose shifted graded pieces
	$$\Z(i)^{\text{mot}}(X) \simeq \text{gr}^i_{\text{mot}} \text{K}(X)[-2i]$$
	are given mod $p$, when $p$ is invertible in $X$, and in degrees less than or equal to $i$, by the étale cohomology $R\Gamma_{\text{ét}}(X,\mu_p^{\otimes i})$:
	$$\tau^{\leq i} \F_p(i)^{\text{mot}}(X) \simeq \tau^{\leq i} R\Gamma_{\text{ét}}(X,\mu_p^{\otimes i}).$$
	
	Such a theory was first developed in the smooth case at the initiative of Bloch and Voevodsky \cite{bloch_algebraic_1986,voevodsky_cycles_2000}, using algebraic cycles and $\mathbb{A}^1$\nobreakdash-homo\-topy theory. In this generality, the use of $\mathbb{A}^1$\nobreakdash-invariant techniques is permitted by Quillen's fundamental theorem of algebraic $K$-theory \cite{quillen_higher_1973}, stating that algebraic $K$-theory is $\mathbb{A}^1$-invariant on regular schemes. On more general schemes, algebraic $K$-theory fails to be $\mathbb{A}^1$-invariant, so motivic cohomology itself needs to be non-$\mathbb{A}^1$-invariant in general. The first non-$\mathbb{A}^1$-invariant theory of motivic cohomology was recently introduced by Elmanto and Morrow \cite{elmanto_motivic_2023}, using recent advances in algebraic $K$-theory and topological cyclic homology. Their theory is developed in the generality of quasi-compact quasi-separated (qcqs) schemes over an arbitrary field, and recovers on smooth varieties the classical $\mathbb{A}^1$-invariant theory. An alternative, a posteriori equivalent approach to non-$\mathbb{A}^1$-invariant motivic cohomology was also introduced by Kelly--Saito in terms of the pro cdh topology \cite{kelly_procdh_2024}.
	
	In this article, we extend the work of Elmanto--Morrow to mixed characteristic, thus producing a theory of motivic cohomology in the originally expected generality of Beilinson and Lichtenbaum. Our theory relies on recent progress in integral $p$-adic Hodge theory \cite{bhatt_topological_2019,bhatt_prisms_2022,bhatt_absolute_2022}, and offers in return a complete description of mod $p$ motivic cohomology, even when $p$ is not invertible in the qcqs scheme~$X$. In addition to the results established in this article, we also refer to \cite{bouis_weibel_2025,bouis_beilinson-lichtenbaum_2025} for further properties satisfied by our theory of motivic cohomology, such as the projective bundle formula, a motivic refinement of Weibel's vanishing in algebraic $K$-theory, and a comparison to the classical theory of motivic cohomology in the smooth case.
	
	The starting point of our construction is the following result, due to Kerz--Strunk--Tamme \cite{kerz_algebraic_2018} (who prove that homotopy $K$-theory is the cdh sheafification of algebraic $K$-theory) and Land--Tamme \cite{land_k-theory_2019} (who prove that the fibre $\text{K}^{\text{inf}}$ of the cyclotomic trace map satisfies cdh descent). Note that these results are the natural generalisations of important partial results due to Haesemeyer \cite{haesemeyer_descent_2004} (who proves that homotopy $K$-theory satisfies cdh descent in characteristic zero), Cisinski \cite{cisinski_descente_2013} (who proves that homotopy $K$-theory satisfies cdh descent in general), and Corti\~{n}as--Haesemeyer--Schlichting--Weibel \cite{cortinas_cyclic_2008} (who prove that $\text{K}^{\text{inf}}$ satisfies cdh descent in characteristic zero). 
	
	\begin{theorem}[\cite{kerz_algebraic_2018,land_k-theory_2019}]\label{theoremKST+LT}
		Let $X$ be a qcqs scheme. Then the natural commutative diagram
		$$\begin{tikzcd}
			\emph{K}(X) \arrow{r} \arrow{d} & \emph{TC}(X) \ar[d] \\
			\emph{KH}(X) \arrow[r] & \big(L_{\emph{cdh}} \emph{TC}\big)(X)
		\end{tikzcd}$$
		is a cartesian square of spectra, where $\emph{KH}(X)$ is the homotopy $K$-theory of $X$, $\emph{TC}(X)$ is the topological cyclic homology of $X$, $L_{\emph{cdh}}$ is the cdh sheafification functor, the top horizontal map is the cyclotomic trace map, and the bottom horizontal map is the cdh sheafified cyclotomic trace map.
	\end{theorem}
	
	Theorem~\ref{theoremKST+LT} states that algebraic $K$-theory of schemes can be reconstructed purely in terms of homotopy $K$-theory ({\it i.e.}, information coming from $\mathbb{A}^1$-homotopy theory) and topological cyclic homology ({\it i.e.}, information coming from trace methods). The cdh topology is a Grothendieck topology introduced by Suslin and Voevodsky \cite{suslin_bloch-kato_2000,voevodsky_unstable_2010}, as a way to apply topos theoretic techniques to the study of resolution of singularities. In particular, assuming resolution of singularities, any qcqs scheme would be locally regular in the cdh topology. While homotopy $K$-theory and topological cyclic homology were originally introduced as tools to approximate the existing algebraic $K$-theory, we construct the motivic cohomology of schemes using refinements of homotopy $K$-theory and topological cyclic homology. More precisely, our motivic filtration on algebraic $K$-theory is defined as a pullback of appropriate filtrations on homotopy $K$-theory, topological cyclic homology, and cdh sheafified topological cyclic homology.  
	
	On homotopy $K$-theory, we use the recent work of Bachmann--Elmanto--Morrow \cite{bachmann_A^1-invariant_2025}, who construct a functorial multiplicative $\N$-indexed filtration $\text{Fil}^\star_{\text{cdh}} \text{KH}(X)$ on the homotopy $K$-theory of qcqs schemes $X$.\footnote{This is somehow in contrast with the work of Elmanto--Morrow \cite{elmanto_motivic_2023}, where they use instead the filtration $\text{Fil}^\star_{\mathbb{A}^1} \text{KH}(X)$, defined as the $\mathbb{A}^1$-invariant slice filtration on homotopy $K$-theory, and which is the central object of study in \cite{bachmann_A^1-invariant_2025}. Although the filtration $\text{Fil}^\star_{\text{cdh}} \text{KH}(X)$ will be more suited for our technical purposes, note that the natural $\mathbb{A}^1$\nobreakdash-localisation map $\text{Fil}^\star_{\text{cdh}} \text{KH}(X) \rightarrow \text{Fil}^\star_{\mathbb{A}^1} \text{KH}(X)$ is an equivalence if $X$ is defined over a field, and is conjectured to always be an equivalence (\cite[Section~1.3]{bachmann_A^1-invariant_2025}).} The shifted graded pieces of this filtration, that we will denote by $\Z(i)^{\text{cdh}}(X)$, provide a good theory of {\it cdh-local motivic cohomology} for qcqs schemes. Note that we use here only the construction and some of the basic properties of the filtration $\text{Fil}^\star_{\text{cdh}} \text{KH}(X)$, which we review in Section~\ref{subsectioncdhlocalmotivicfiltration}. In particular, we do not rely on the main theorems of \cite{bachmann_A^1-invariant_2025} on $\mathbb{A}^1$-invariant motivic cohomology, such as the projective bundle formula or the comparison with the slice filtration in stable motivic homotopy theory. 
	
	Our first construction is that of a functorial multiplicative $\Z$-indexed filtration $\text{Fil}^\star_{\text{mot}} \text{TC}(X)$ for qcqs schemes $X$. This filtration recovers the HKR filtration on $\text{HC}^{-}(X/\Q)$ in characteristic zero \cite{antieau_periodic_2019,moulinos_universal_2022,raksit_hochschild_2026}, and the motivic filtration on $\text{TC}(X;\Z_p)$ after $p$-completion \cite{bhatt_topological_2019,morin_topological_2021,bhatt_absolute_2022,hahn_motivic_2022}. We describe the shifted graded pieces $$\Z(i)^{\text{TC}}(X) \simeq \text{gr}^i_{\text{mot}} \text{TC}(X)[-2i]$$ of this filtration in the following definition.
	
	\needspace{5\baselineskip}
	
	\begin{definition}[See Definition~\ref{definitionmotivicfiltrationonTC}]\label{definitionintroZ(i)^TC}
		For every qcqs scheme $X$ and every integer $i \in \Z$, the complex $\Z(i)^{\text{TC}}(X) \in \mathcal{D}(\Z)$ is defined by a natural cartesian square
		$$\begin{tikzcd}
			\Z(i)^{\text{TC}}(X) \ar[r] \ar[d] & R\Gamma_{\text{Zar}}\big(X,\widehat{\mathbb{L}\Omega}^{\geq i}_{-/\Z}\big) \ar[d] \\
			\prod_{p \in \mathbb{P}} \Z_p(i)^{\text{BMS}}(X) \ar[r] & \prod_{p \in \mathbb{P}} R\Gamma_{\text{Zar}}\big(X,(\widehat{\mathbb{L}\Omega}^{\geq i}_{-/\Z})^\wedge_p\big),
		\end{tikzcd}$$
		where $\Z_p(i)^{\text{BMS}}(X)$ is the syntomic cohomology of the $p$-adic formal scheme associated to $X$.
	\end{definition}
	
	It is straightforward from this definition that the presheaf $\Z(i)^{\text{TC}}$ is naturally identified with Bhatt--Morrow--Scholze's syntomic complex $\Z_p(i)^{\text{BMS}}$ in characteristic $p$, and with the Hodge-completed derived de Rham complex $R\Gamma_{\text{Zar}}(-,\widehat{\mathbb{L}\Omega}^{\geq i}_{-/\Q})$ in characteristic zero. Following \cite{elmanto_motivic_2023}, the motivic complex $\Z(i)^{\text{mot}}$ is defined in characteristic $p$ and zero respectively by cartesian squares
	$$\begin{tikzcd}
		\Z(i)^{\text{mot}}(X) \ar[r] \ar[d] & \Z_p(i)^{\text{BMS}}(X) \ar[d] & \Z(i)^{\text{mot}}(X) \ar[r] \ar[d] & R\Gamma_{\text{Zar}}\big(X,\widehat{\mathbb{L}\Omega}^{\geq i}_{-/\Q}\big) \ar[d] \\
		\Z(i)^{\text{cdh}}(X) \ar[r] & \big(L_{\text{cdh}} \Z_p(i)^{\text{BMS}}\big)(X) & \Z(i)^{\text{cdh}}(X) \ar[r] & R\Gamma_{\text{cdh}}\big(X,\widehat{\mathbb{L}\Omega}^{\geq i}_{-/\Q}\big).
	\end{tikzcd}$$
	The following definition is then a natural mixed characteristic generalisation of Elmanto--Morrow's definition over a field.
	
	\begin{definition}[Motivic cohomology; see Section~\ref{subsectiondefinitionmotiviccohomology}]\label{definitionintromotiviccohomology}
		For every qcqs scheme $X$ and every integer $i \in \Z$, the {\it weight-$i$ motivic complex}
		$$\Z(i)^{\text{mot}}(X) \in \mathcal{D}(\Z)$$ of $X$ is defined by a natural cartesian square
		$$\begin{tikzcd}
			\Z(i)^{\text{mot}}(X) \ar[r] \ar[d] & \Z(i)^{\text{TC}}(X) \ar[d] \\
			\Z(i)^{\text{cdh}}(X) \ar[r] & \big(L_{\text{cdh}} \Z(i)^{\text{TC}}\big)(X),
		\end{tikzcd}$$
		where the bottom horizontal map is induced by a filtered refinement of the cdh sheafified cyclotomic trace map.
	\end{definition}
	
	However, proving the expected relation between the motivic complexes $\Z(i)^{\text{mot}}$ and algebraic $K$\nobreakdash-theory (Theorem~\ref{theoremintroAHSSandAdams} below) requires more efforts in mixed characteristic than over a field.
    The main foundational obstacle, directly related to the predicted exhaustiveness of the motivic filtration, is to prove that the presheaves $\Z(i)^{\text{mot}}$ vanish in weights~$i<0$. Given the previous definition, this obstacle becomes a concrete technical problem, which we now explain how to resolve. Note first that, by construction, the presheaves $\Z(i)^{\text{cdh}}$ do vanish on all qcqs schemes in weights $i<0$ (Section~\ref{subsectioncdhlocalmotivicfiltration}).	
	
	In characteristic $p$, the presheaves $\Z_p(i)^{\text{BMS}}$ vanish in weights $i<0$, thus so do the presheaves $\Z(i)^{\text{mot}}$, and the zeroth step of the associated motivic filtration $\text{Fil}^\star_{\text{mot}} \text{K}$ recovers algebraic $K$-theory. Ignoring for a moment the completeness issues for this motivic filtration, this means that the presheaves $\Z(i)^{\text{mot}}$ provide a natural cohomological refinement of algebraic $K$-theory on arbitrary characteristic~$p$ schemes.
	
	In characteristic zero, the presheaves $R\Gamma_{\text{Zar}}(-,\widehat{\mathbb{L}\Omega}^{\geq i}_{-/\Q})$ do not vanish in weights $i<0$. Instead, they are equal to the presheaf $R\Gamma_{\text{Zar}}(-,\widehat{\mathbb{L}\Omega}_{-/\Q})$, which happens to be a cdh sheaf on qcqs $\Q$-schemes, by results of Corti\~{n}as--Haesemeyer--Schlichting--Weibel \cite{cortinas_cyclic_2008}, Antieau \cite{antieau_periodic_2019}, and Elmanto--Morrow \cite{elmanto_motivic_2023}. That is, the right vertical map of the previous diagram is an equivalence in weights $i<0$, so the presheaves $\Z(i)^{\text{mot}}$ vanish in weights $i<0$, and the zeroth step of the associated motivic filtration $\text{Fil}^\star_{\text{mot}} \text{K}$ recovers algebraic $K$-theory. This means that the presheaves $\Z(i)^{\text{mot}}$ provide a natural cohomological refinement of algebraic $K$-theory on arbitrary characteristic zero schemes.
	
	In mixed characteristic, we prove similarly that in weights $i<0$, the presheaves $\Z(i)^{\text{TC}}$ are cdh sheaves on qcqs schemes, {\it i.e.}, that the presheaves $\Z(i)^{\text{mot}}$ vanish. The result modulo a prime number $p$ is a consequence, as in characteristic $p$, of the fact that the presheaves $\Z_p(i)^{\text{BMS}}$ vanish in weights~$i<0$. The difficulty is to then prove that the presheaves $\Q(i)^{\text{TC}}$ are cdh sheaves in weight $i<0$.
	
	The main cdh descent result used in characteristic zero does not hold in mixed characteristic. That is, the presheaf $R\Gamma_{\text{Zar}}(-,\widehat{\mathbb{L}\Omega}_{-/\Z})$ (or its rationalisation) is not a cdh sheaf on qcqs schemes. We avoid this difficulty by proving a rigid-analytic analogue of this cdh descent result. To formulate this result, denote by $$R\Gamma_{\text{Zar}}\big(X,\underline{\widehat{\mathbb{L}\Omega}}^{\geq i}_{-_{\Q_p}/\Q_p}\big)$$ the {\it rigid-analytic derived de Rham cohomology} of a qcqs $\Z_{(p)}$-scheme $X$, which we define as the pushout of the diagram
	$$R\Gamma_{\text{Zar}}\big(X,(\widehat{\mathbb{L}\Omega}^{\geq i}_{-/\Z})^\wedge_p\big) \longleftarrow R\Gamma_{\text{Zar}}\big(X,\widehat{\mathbb{L}\Omega}^{\geq i}_{-/\Z}) \longrightarrow R\Gamma_{\text{Zar}}(X,\widehat{\mathbb{L}\Omega}^{\geq i}_{-_{\Q}/\Q})$$
	in the derived category $\mathcal{D}(\Z)$. Here, we restrict our attention to qcqs $\Z_{(p)}$-schemes for simplicity, and refer to Section~\ref{sectionratonialstructure} for the relevant statements over~$\Z$. As a consequence of Definition~\ref{definitionintroZ(i)^TC}, there is a natural cartesian square
	$$\begin{tikzcd}
		\Q(i)^{\text{TC}}(X) \ar[r] \ar[d] & R\Gamma_{\text{Zar}}(X,\widehat{\mathbb{L}\Omega}^{\geq i}_{-_{\Q}/\Q}) \ar[d] \\
		\Q_p(i)^{\text{BMS}}(X) \ar[r] & R\Gamma_{\text{Zar}}(X,\underline{\widehat{\mathbb{L}\Omega}}^{\geq i}_{-_{\Q_p}/\Q_p})
	\end{tikzcd}$$
	in the derived category $\mathcal{D}(\Q)$. In weights $i<0$, the presheaves $\Q_p(i)^{\text{BMS}}$ vanish. As already mentioned, the presheaf $R\Gamma_{\text{Zar}}(-,\widehat{\mathbb{L}\Omega}_{-_{\Q}/\Q})$ is moreover a cdh sheaf on qcqs schemes. So the fact that the presheaves $\Q(i)^{\text{TC}}$ are cdh sheaves in weights $i<0$ reduces to the following result, which can be seen as a rigid-analytic analogue of the latter cdh descent over $\Q$.
	
	\begin{theoremintro}[Cdh descent for rigid-analytic derived de Rham cohomology; see Corollary~\ref{corollaryHPsolidallprimespisacdhsheafongradedpieces}]\label{theoremintrocdhdescentrigidanalytic}
		For every prime number $p$, the presheaf $R\Gamma_{\emph{Zar}}\big(-,\underline{\widehat{\mathbb{L}\Omega}}_{-_{\Q_p}/\Q_p}\big)$ satisfies cdh descent on qcqs $\Z_{(p)}$\nobreakdash-schemes.
	\end{theoremintro}
	
	The modern proof of the analogous result over $\Q$ relies on the theory of truncating invariants of Land--Tamme \cite{land_k-theory_2019} and on a theorem of Goodwillie \cite{goodwillie_cyclic_1985}, who prove respectively that every truncating invariant is a cdh sheaf on qcqs schemes and that periodic cyclic homology over $\Q$ is a truncating invariant. By definition, a truncating invariant is a localizing invariant $E$ such that for every connective $\mathbb{E}_1$-ring $R$, the natural map $E(R) \rightarrow E(\pi_0(R))$ is an equivalence. To prove Theorem~\ref{theoremintrocdhdescentrigidanalytic}, we then use the condensed mathematics of Clausen--Scholze \cite{clausen_condensed_2019}, and prove that a suitable rigid-analytic variant of periodic cyclic homology is a truncating invariant. In particular, the proof of Theorem~\ref{theoremintrocdhdescentrigidanalytic} relies on a result on associative rings (actually, on general solid connective $\mathbb{E}_1$-rings).
	
	As a consequence of Theorem~\ref{theoremintrocdhdescentrigidanalytic}, we obtain the following cohomological description of rational motivic cohomology.

	\begin{theoremintro}[$p$-adic and rational motivic cohomology; see Corollaries~\ref{corollarymainpadicstructureongradeds}\label{theoremintromaincartesiansquares} and~\ref{corollarycartesiansquarerational}]
		Let $X$ be a qcqs scheme, and $p$ be a prime number. Then for any integers $i \in \Z$ and $k \geq 1$, the natural commutative diagrams
		$$\begin{tikzcd}
			\Z/p^k(i)^{\emph{mot}}(X) \ar[r] \ar[d] & \Z/p^k(i)^{\emph{BMS}}(X) \ar[d] & \Q(i)^{\emph{mot}}(X) \ar[r] \ar[d] & R\Gamma_{\emph{Zar}}\big(X,\widehat{\mathbb{L}\Omega}^{\geq i}_{-_{\Q}/\Q}\big) \ar[d] \\
			\Z/p^k(i)^{\emph{cdh}}(X) \ar[r] & \big(L_{\emph{cdh}} \Z/p^k(i)^{\emph{BMS}}\big)(X) & \Q(i)^{\emph{cdh}}(X) \ar[r] & R\Gamma_{\emph{cdh}}\big(X,\widehat{\mathbb{L}\Omega}^{\geq i}_{-_{\Q}/\Q}\big)
		\end{tikzcd}$$
		are cartesian squares in the derived category $\mathcal{D}(\Z)$. See also Theorem~\ref{theorempadicmotiviccohomologyintermsofsyntomicohomology} and Corollary~\ref{corollaryrationalmainresultongradedpieces} for useful variants of these cartesian squares. 
	\end{theoremintro}
	
	Together, these two cartesian squares recover the cartesian squares of Elmanto--Morrow that define the motivic complexes $\Z(i)^{\text{mot}}$ over a field, and are thus natural mixed characteristic analogues of these. The $p$-adic part of Theorem~\ref{theoremintromaincartesiansquares} is a formal consequence of Definitions~\ref{definitionintroZ(i)^TC} and~\ref{definitionintromotiviccohomology}. The rational part of Theorem~\ref{theoremintromaincartesiansquares} implies that the difference between $\Q(i)^{\text{mot}}(X)$ and $\Q(i)^{\text{cdh}}(X)$ depends only on the rationalisation $X_{\Q}$ of the scheme $X$. If $X$ is regular, this difference should vanish, and this is then more interesting in the presence of singularities. More precisely, Theorem~\ref{theoremintromaincartesiansquares} can be used to capture interesting information about the singularities of an arbitrary commutative ring $R$: cdh sheaves are typically insensitive to singularities, so the singular information in the motivic complex $\Z(i)^{\text{mot}}(R)$ is controlled by the complexes $\Z/p^k(i)^{\text{BMS}}(R)$ and $\mathbb{L}\Omega^{<i}_{(R\otimes_{\Z} \Q)/\Q}$, which are accessible to computation (see \cite{bouis_motivic_2026} for concrete applications of this idea).
	
	Theorem~\ref{theoremintromaincartesiansquares} also implies that the presheaves $\Q(i)^{\text{mot}}$ vanish in weights \hbox{$i<0$}, which was the essential missing part to establish the following fundamental properties of motivic cohomology.
	
	\begin{theoremintro}[Relation to algebraic $K$-theory]\label{theoremintroAHSSandAdams}
		There exists a finitary Nisnevich sheaf of filtered spectra
		$$\emph{Fil}^\star_{\emph{mot}} \emph{K}(-) : \emph{Sch}^{\emph{qcqs,op}} \longrightarrow \emph{FilSp}$$
		with the following properties:
		\begin{enumerate}
			\item \emph{(Atiyah--Hirzebruch spectral sequence; see Section~\ref{subsectionAHSS})} For every qcqs scheme $X$, the filtration $\emph{Fil}^\star_{\emph{mot}} \emph{K}(X)$ is a multiplicative $\N$-indexed filtration on the non-connective algebraic $K$-theory $\emph{K}(X)$, whose graded pieces are naturally given by
			$$\emph{gr}^i_{\emph{mot}} \emph{K}(X) \simeq \Z(i)^{\emph{mot}}(X)[2i], \quad i \geq 0.$$
			In particular, writing $\emph{H}^j_{\emph{mot}}(X,\Z(i)) := \emph{H}^j(\Z(i)^{\emph{mot}}(X))$ for the corresponding \emph{motivic cohomology groups}, there exists an Atiyah--Hirzebruch spectral sequence
			$$E_2^{i,j} = \emph{H}^{i-j}_{\emph{mot}}(X,\Z(-j)) \Longrightarrow \emph{K}_{-i-j}(X).$$
			If $X$ has finite valuative dimension,\footnote{The valuative dimension of a commutative ring, defined in terms of the ranks of certain valuation rings, was introduced by Jaffard in \cite[Chapter IV]{jaffard_theorie_1960}, and generalised to schemes in \cite[Section~$2.3$]{elmanto_cdh_2021}. The valuative dimension of a scheme is always at least equal to its Krull dimension, and both notions agree on noetherian schemes. For our purposes, the valuative dimension of a qcqs scheme $X$ will be used as an upper bound on the cohomological dimension of the cdh topos of $X$ (\cite[Theorem~$2.4.15$]{elmanto_cdh_2021}).} then the filtration $\emph{Fil}^\star_{\emph{mot}} \emph{K}(X)$ is complete, and the Atiyah--Hirzebruch spectral sequence is convergent.
			\item \emph{(Adams decomposition; see Corollaries~\ref{corollaryAHSS} and~\ref{corollaryKtheorysplitsrationally2})} For every qcqs scheme $X$, the Atiyah--Hirzebruch spectral sequence degenerates rationally and, for every integer $n \in \Z$, there is a natural isomorphism of abelian groups
			$$\emph{K}_n(X) \otimes_{\Z} \Q \cong \bigoplus_{i \geq 0} \big(\emph{H}^{2i-n}_{\emph{mot}}(X,\Z(i)) \otimes_{\Z} \Q\big)$$
			induced by the Adams operations on rationalised algebraic $K$-theory.
		\end{enumerate} 
	\end{theoremintro}
	
	One of the main historical motivations for developing motivic cohomology was to apply cohomological techniques to the study of algebraic $K$-theory \cite{beilinson_notes_1987}. The following theorem summarizes our results on the relations between motivic cohomology and previously studied cohomological invariants. When $X$ is smooth over $\Z$, we denote by
	$$\Z(i)^{\text{cla}}(X) := \big(L_{\text{Nis}}\, z^i(-,\bullet)\big)(X)[-2i]$$
	the {\it weight-$i$ classical motivic complex}, where $z^i(X,\bullet)$ is Bloch's cycle complex (and $\bullet$ is the cohomological index).
	
	\needspace{5\baselineskip}
	
	\begin{theoremintro}\label{theoremintrocohomology}
		Let $X$ be a qcqs scheme, and $i \geq 0$ be an integer.
		\begin{enumerate}
			\item \emph{(Weight zero; see Example~\ref{exampleweightzeromotiviccohomology})} There is a natural equivalence
			$$\Z(0)^{\emph{mot}}(X) \xlongrightarrow{\sim} R\Gamma_{\emph{cdh}}(X,\Z)$$
			in the derived category $\mathcal{D}(\Z)$.
			\item \emph{(Weight one; see} \cite{bouis_weibel_2025}\emph{)}\footnote{This result, which only appears in the sequel to this article (\cite[Example~2.15]{bouis_weibel_2025}), is stated here only for completeness, and will not be used in this article.} 
            There is a natural map
			$$R\Gamma_{\emph{Nis}}(X,\mathbb{G}_m)[-1] \longrightarrow \Z(1)^{\emph{mot}}(X)$$
			in the derived category $\mathcal{D}(\Z)$ which is an isomorphism in degrees less than or equal to three. In particular, there are natural isomorphisms of abelian groups $\emph{H}^1_{\emph{mot}}(X,\Z(1)) \cong \mathcal{O}(X)^{\times}$ and $\emph{H}^2_{\emph{mot}}(X,\Z(1)) \cong \emph{Pic}(X)$.
			\item \emph{(\'Etale cohomology; see Corollary~\ref{corollaryladicmotivcohomologylowdegreesisétalecohomology})} For every prime number $\ell$ which is invertible in $X$ and every integer $k \geq 1$, there is a natural map
			$$\Z/\ell^k(i)^{\emph{mot}}(X) \longrightarrow R\Gamma_{\emph{ét}}(X,\mu_{\ell^k}^{\otimes i})$$
			in the derived category $\mathcal{D}(\Z/\ell^k)$ which is an isomorphism in degrees less than or equal to $i$.
			\item \emph{(Syntomic cohomology; see Corollary~\ref{corollarypadiccomparisoninsmalldegreessyntomiccoho})} More generally, for every prime number $p$ and every integer $k \geq 1$, there is a natural map
			$$\Z/p^k(i)^{\emph{mot}}(X) \longrightarrow \Z/p^k(i)^{\emph{syn}}(X)$$
			in the derived category $\mathcal{D}(\Z/p^k)$ which is an isomorphism in degrees less than or equal to $i$, where $\Z/p^k(i)^{\emph{syn}}(X)$ denotes the weight-$i$ syntomic cohomology of $X$ in the sense of \cite{bhatt_absolute_2022}.
			\item \emph{($\mathbb{A}^1$-invariant motivic cohomology; see Theorem~\ref{theoremA1localmotiviccohomologymain})} There is a natural equivalence 
			$$\big(L_{\mathbb{A}^1} \Z(i)^{\emph{mot}}\big)(X) \xlongrightarrow{\sim} \Z(i)^{\mathbb{A}^1}(X)$$
			in the derived category $\mathcal{D}(\Z)$, where $\Z(i)^{\mathbb{A}^1}(X)$ denotes the cohomology
			represented by the slices of the $K$-theory motivic spectrum $\emph{KGL}_X \in \emph{SH}(X)$.
		\end{enumerate}
	\end{theoremintro}

    Also note, although this is only proved in the sequel to this article \cite{bouis_weibel_2025}, that the motivic complexes $\Z(i)^{\text{mot}}$ satisfy the projective bundle formula, and are thus naturally represented in Annala--Hoyois--Iwasa's category of non-$\mathbb{A}^1$-invariant motives \cite{annala_motivic_2023,annala_algebraic_2025,annala_atiyah_2024}. In the equicharacteristic case, this result was proved by Elmanto--Morrow \cite{elmanto_motivic_2023}, relying on the fact that cdh-local motivic cohomology $\Z(i)^{\text{cdh}}$ is naturally identified with $\mathbb{A}^1$-invariant motivic cohomology, which is known to satisfy the projective bundle formula \cite{bachmann_A^1-invariant_2025}. In mixed characteristic, however, the cdh-local motivic complexes $\Z(i)^{\text{cdh}}$ are known to satisfy the projective bundle formula only conditionally on a certain property of valuation rings, called $F$-smoothness \cite{bhatt_syntomic_2023,bachmann_A^1-invariant_2025}. This condition can be proved in mixed characteristic for valuation rings over a perfectoid base: this is the main result of \cite{bouis_cartier_2023}. The case of general valuation rings remaining open, our proof of the projective bundle formula instead relies on our description of motivic cohomology with finite coefficients in terms of syntomic cohomology (Theorem~\ref{theorempadicmotiviccohomologyintermsofsyntomicohomology}) and on a degeneration argument using Clausen--Mathew's Selmer $K$-theory \cite{clausen_hyperdescent_2021}.
    
	The following result is a consequence of Theorem~\ref{theoremintrocohomology}\,$(4)$ and the fact that the $\mathbb{A}^1$-invariant motivic complexes $\Z(i)^{\mathbb{A}^1}$ recover the classical motivic complexes $\Z(i)^{\text{cla}}$ on smooth $\Z$-schemes. In particular, although we expect our motivic complexes $\Z(i)^{\text{mot}}$ to actually coincide with the classical motivic complexes $\Z(i)^{\text{cla}}$ on smooth $\Z$-schemes,\footnote{This is now one of the main results of \cite{bouis_beilinson-lichtenbaum_2025}.} this means that the former at least recover the latter after enforcing $\mathbb{A}^1$-invariance.
	
	\begin{corollaryintro}[See Corollary~\ref{corollary26A1localmotiviccohomologyisclassicalmotiviccohomology}]
		Let $X$ be a smooth scheme over $\Z$. Then for every integer $i \geq 0$, there is a natural equivalence
		$$\Z(i)^{\emph{cla}}(X) \simeq \big(L_{\mathbb{A}^1} \Z(i)^{\emph{mot}}\big)(X)$$
		in the derived category $\mathcal{D}(\Z)$.
	\end{corollaryintro}

    Recall that on smooth $\F_p$\nobreakdash-schemes, the Beilinson--Lichtenbaum conjecture, proved by Suslin and Voevodsky as a consequence of the Bloch--Kato conjecture \cite{suslin_bloch-kato_2000}, computes the $\ell$\nobreakdash-adic part of motivic cohomology in terms of the cohomology of the étale sheaf $\mu_{\ell^k}$ of $\ell^k$\nobreakdash-roots of unity. To describe the $p$\nobreakdash-adic part of motivic cohomology, one needs to replace $\mu_{\ell^k}^{\otimes i}$ (which is zero on smooth varieties when $\ell=p$ and~$i>0$) by the logarithmic de Rham--Witt sheaves $W_k\Omega^i_{-,\text{log}}$ \cite{geisser_k-theory_2000}. The corresponding description of $p$\nobreakdash-adic algebraic $K$\nobreakdash-theory, in terms of the logarithmic de Rham--Witt sheaves, is generalised in \cite{kelly_k-theory_2021} to all Cartier smooth $\F_p$\nobreakdash-algebras, and in particular to all characteristic $p$ valuation rings.
	
	On smooth schemes over a mixed characteristic Dedekind domain, the $p$\nobreakdash-adic part of classical motivic cohomology is similarly described in low degrees by the étale cohomology of the generic fibre \cite{geisser_motivic_2004}. This result is a consequence of the Gersten conjecture proved by Geisser \cite{geisser_motivic_2004}, and is unknown for general regular schemes. Combined with Theorem~\ref{theoremintrocohomology}\,$(3)$, the following result extends this description of classical motivic cohomology to the regular case. More precisely, the notion of $F$-smoothness was introduced by Bhatt--Mathew \cite{bhatt_syntomic_2023} as a non-noetherian generalisation of regular schemes, and our result naturally applies to general $p$-torsionfree $F$-smooth schemes.
	
	\begin{theoremintro}[Beilinson--Lichtenbaum conjecture for $F$-smooth schemes; see Corollary~\ref{corollaryFsmoothnessBeilinsonLichtenbaumcomparison}]\label{theoremintroBeilinsonLichtenbaumconjectureforFsmoothschemes}
		Let $X$ be a $p$-torsionfree $F$-smooth scheme ({\it e.g.}, a regular scheme flat over $\Z$). Then for any integers $i \geq 0$ and $k \geq 1$, the Beilinson--Lichtenbaum comparison map
		$$\Z/p^k(i)^{\emph{mot}}(X) \longrightarrow R\Gamma_{\emph{ét}}(X[\tfrac{1}{p}],\mu_{p^k}^{\otimes i})$$
		is an isomorphism in degrees less than or equal to $i-1$, and is injective in degree $i$.
	\end{theoremintro}
	
	The proof of Theorem~\ref{theoremintroBeilinsonLichtenbaumconjectureforFsmoothschemes} relies on a syntomic-étale comparison theorem of Bhatt--Mathew \cite{bhatt_syntomic_2023} (see also \cite[Theorem~B]{bouis_cartier_2023} for a proof over perfectoid bases using relative prismatic cohomology).

        \medskip
        
	\subsection*{Notation}
	
	\vspace{-\parindent}
	\hspace{\parindent}

    {\bf Algebraic $K$-theory.} 
    By default, algebraic $K$-theory means non-connec\-tive algebraic $K$-theory, as introduced by Thoma\-son--Trobaugh \cite{thomason_higher_1990}. By a theorem of Blumberg--Gepner--Tabuada \cite{blumberg_universal_2013}, non\nobreakdash-connective algebraic $K$-theory is the universal localizing invariant.

	{\bf $\mathbb{A}^1$-invariance.} 
    A presheaf $F(-)$ on schemes is called {\it $\mathbb{A}^1$-invariant} if for every scheme $X$ and every integer~\hbox{$m \geq 0$}, the natural map $F(X) \rightarrow F(\mathbb{A}^m_X)$ is an equivalence. Given a presheaf $F(-)$ on schemes, the {\it $\mathbb{A}^1$\nobreakdash-locali\-sation} $L_{\mathbb{A}^1} F(-)$ of $F(-)$ is the initial $\mathbb{A}^1$-invariant presheaf with a map from $F(-)$. The $\mathbb{A}^1$-localisation functor $L_{\mathbb{A}^1}$ commutes with colimits.

	{\bf Base change.} 
    Given a commutative ring $R$, an $R$-algebra $S$, and a scheme $X$ over $\text{Spec}(R)$, denote by~$X_{S}$ the base change $X \times_{\text{Spec}(R)} \text{Spec}(S)$ of $X$ from $R$ to $S$. If $X$ is a derived scheme, this base change is implicitly the derived base change from $R$ to $S$. We sometimes use the derived base even on classical schemes, and say explicitly when we do so.

	{\bf Bounded torsion.} 
    An abelian group $A$ is said to have {\it bounded torsion} if there exists an integer $N \geq 1$ such that the multiplication by $N$ of every element of $A$ is zero. Given a commutative ring $R$ and an element $d$ of $R$, an $R$-module $M$ is said to have {\it bounded $d$-power torsion} if there exists an integer $n\geq 1$ such that $M[d^m]=M[d^n]$ for all $m \geq n$; this assumption guarantees that the derived $d$\nobreakdash-completion of $M$ is in degree zero, given by the classical $d$\nobreakdash-completion of $M$.

	{\bf Cdh topology.} 
    The cdh topology is a Grothendieck topology introduced by Voevodsky \cite{suslin_bloch-kato_2000,voevodsky_unstable_2010}; see \cite{elmanto_cdh_2021} for the definition and properties of the cdh topology in the generality of qcqs schemes. It is a completely decomposed version of the topology generated by Deligne's hypercoverings. The cdh sheafification functor $L_{\text{cdh}}$ preserves multiplicative structures.

	{\bf Coefficients.} 
    Given a functor $F(-)$ (resp. $\text{Fil}^\star F(-)$) taking values in spectra (resp. filtered spectra) and a prime number $p$, we denote the rationalisation of $F(-)$ by $F(-;\Q)$, its reduction modulo~$p$ by $F(-;\F_p)$, its derived reduction modulo powers of $p$ by $F(-;\Z/p^k)$, its $p$-completion by $F(-;\Z_p)$, and the rationalisation of its $p$-completion by $F(-;\Q_p)$. We adopt a similar notation for a functor $\text{Fil}^\star F(-)$ taking values in filtered spectra, {\it e.g.}, we denote its rationalisation by $\text{Fil}^{\star} F(-;\Q)$. Similarly, if $\Z(i)^F(-)$ is a functor taking values in the derived category $\mathcal{D}(\Z)$, we denote the rationalisation of $\Z(i)^F(-)$ by $\Q(i)^F(-)$, its derived reduction modulo $p$ by $\F_p(i)^F(-)$, etc. Following the same notation, we also write
    $$F(-;\mathbb{A}_f) := \Big(\prod_{p \in \mathbb{P}} F(-;\Z_p)\Big)_{\Q} \quad \Big(\text{resp. } \text{Fil}^\star F(-;\mathbb{A}_f) := \Big(\prod_{p \in \mathbb{P}} \text{Fil}^\star F(-;\Z_p)\Big)_{\Q}\Big).$$
    
	{\bf Derived categories and spectra.} 
    Denote by $\text{Sp}$ the category of spectra. Given a commutative ring $R$, denote by $\mathcal{D}(R)$ the derived category of $R$-modules; it is implicitly the derived $\infty$-category of $R$-modules, and is in particular naturally identified with the category of $R$-linear spectra. Our convention for degrees is by default cohomological. In this context, the notions of fibre and cofibre sequences agree, and the fibre and cofibre of a given map satisfy the relation $\text{fib} \simeq \text{cofib}[-1]$. Given an element $d$ of $R$, also denote by $(-)^\wedge_d$ the $d$-adic completion functor in the derived category~$\mathcal{D}(R)$.

	{\bf Filtrations.} 
    By default, a filtration with values in a category $\mathcal{C}$ is a $\Z$-indexed decreasing filtered object in the category $\mathcal{C}$, {\it i.e.}, a functor from the category $(\Z,\geq)^\text{op}$ to the category $\mathcal{C}$. A filtration is called {\it $\N$-indexed} if it is constant in non-positive degrees. Given a filtered object $\text{Fil}^\star\, C$ and for each integer $n \in \Z$, let $\text{gr}^n\, C \in \mathcal{D}(R)$ denote the cofibre of the transition map $\text{Fil}^{n+1}\,C \rightarrow \text{Fil}^n\,C$. A filtered object $\text{Fil}^\star\, C$ is said to be {\it complete} if the limit $\text{lim}_n\, \text{Fil}^n\, C$ vanishes. For instance, The Hodge filtration on the de Rham complex is given for each $n \in \Z$ by $\text{Fil}_\text{Hod}^n\, \Omega_{-/R} := \Omega_{-/R}^{\geq n}$; the Hodge filtration $\mathbb{L}\Omega^{\geq \star}_{-/R}$ on the derived de Rham complex $\mathbb{L}\Omega_{-/R}$ is defined as the left Kan extension of this filtration. It is $\N$-indexed, but not always complete. Its completion, the Hodge-completed derived de Rham complex, is denoted by $\widehat{\mathbb{L}\Omega}^{\geq \star}_{-/R}$.
	
	Given a commutative ring $R$, denote by $\mathcal{DF}(R) := \text{Fun}((\Z,\geq)^\text{op},\mathcal{D}(R))$ the filtered derived category of $R$-modules. Also denote by $\text{FilSp}$ the category of filtered spectra, and by $\text{biFilSp}$ the category of bifiltered spectra ({\it i.e.}, the category of filtered objects in the category of filtered spectra). A {\it multiplicative filtration} is a commutative algebra object in one of these filtered categories ({\it e.g.}, a multiplicative filtered spectrum is an object of $\text{CAlg}(\text{FilSp})$).

	{\bf Henselian rings.} 
    Given a commutative ring $R$ and an ideal $I$ of $R$, the pair $(R,I)$ is called {\it henselian} if it satisfies Hensel's lemma. A local ring $R$ is called {\it henselian} if the pair $(R,\mathfrak{m})$ is henselian, where $\mathfrak{m}$ is the maximal ideal of $R$. Henselian local rings are the local rings for the Nisnevich topology. A commutative ring $R$ is called {\it $d$-henselian}, for $d$ an element of $R$, if the pair $(R,(d))$ is henselian.
	
	{\bf Left Kan extensions.} 
    Given a commutative ring $R$, an $\infty$-category $\mathcal{D}$ which admits sifted colimits ({\it e.g.}, $\mathcal{D}(R)$ or $\mathcal{DF}(R)$), and a functor $$F : \text{Sm}_R := \{\text{smooth }R\text{-algebras}\} \longrightarrow \mathcal{D},$$ define
	$$\begin{array}{ll} \mathbb{L}F : &R\text{-Alg} \longrightarrow \mathcal{D} \\ &S \longmapsto \underset{P\rightarrow S}{\text{colim}} \text{ } F(P),\end{array}$$
	where the colimit is taken over all free $R$-algebras $P$ mapping to $S$. The functor $\mathbb{L}F$ is called the left Kan extension from polynomial $R$-algebras of~$F$. For instance, the cotangent complex $\mathbb{L}_{-/R} := \mathbb{L}\Omega^1_{-/R}$ is the left Kan extension from polynomial $R$-algebras of the module of Kähler differentials $\Omega^1_{-/R}$, and the derived de Rham complex $\mathbb{L}\Omega_{-/R}$ is the left Kan extension from polynomial $R$-algebras of the de Rham complex $\Omega_{-/R}$. We also consider more general left Kan extensions ({\it e.g.}, from smooth $R$-algebras), which are defined similarly --see \cite[Section~$2.3$ and Remark~$3.4$]{elmanto_motivic_2023} for a quick review of this formalism. The left Kan extension from a category~$\mathcal{C}$ to a category $\mathcal{C'}$, when this makes sense, is denoted by $L_{\mathcal{C}'/\mathcal{C}}$.
	
	{\bf Quasisyntomic rings.} 
    A morphism $R \rightarrow S$ of commutative rings is called {\it $p$-discrete}, for $p$ a prime number, if the derived tensor product $S \otimes_{R}^{\mathbb{L}} R/p \in \mathcal{D}(R/p)$ is concentrated in degree zero, where it is given by $S/p$. It is called {\it $p$-flat} if it is $p$-discrete and if its reduction $R/p \rightarrow S/p$ modulo $p$ is flat. It is called {\it $p$-quasisyntomic} if it is $p$-flat and if the cotangent complex $\mathbb{L}_{(S/p)/(R/p)} \in \mathcal{D}(S/p)$ has Tor-amplitude in $[-1;0]$. 
	
	A commutative ring $R$ is called {\it $p$-quasisyntomic} if it has bounded $p$-power torsion and if the complex $\mathbb{L}_{R/\Z} \otimes_{R}^{\mathbb{L}} R/p \in \mathcal{D}(R/p)$ has Tor-amplitude in $[-1;0]$. Beware that $p$-quasisyntomic $\Z$-algebras are $p$-quasisyntomic rings, but the converse is not true: for instance, $\F_p$ is a $p$\nobreakdash-quasisyntomic ring, but the morphism $\Z \rightarrow \F_p$ is not $p$-discrete. We refer to \cite{bhatt_topological_2019} for the definition of the associated $p$-quasisyntomic topology on $p$-quasisyntomic rings.
    
	{\bf Rigid functor.} 
    A functor $F(-)$ on commutative rings is called {\it rigid} if for every henselian pair $(R,I)$, the natural map $F(R) \rightarrow F(R/I)$ is an equivalence.
	
	{\bf Rings and schemes.} 
    Quasi-compact quasi-separated (derived) schemes are called qcqs (derived) schemes. These include all affine (derived) schemes, {\it i.e.}, (animated) commutative rings. Denote by $\text{Sch}^{\text{qcqs}}$ the category of qcqs schemes, $\text{dSch}^{\text{qcqs}}$ the category of qcqs derived schemes, $\text{Rings}$ the category of commutative rings, $\text{AniRings}$ the category of animated commutative rings. Schemes, resp. commutative rings, are sometimes called {\it classical} to emphasize that we are not working in the generality of derived schemes, resp. animated commutative rings. Given a commutative base ring $R$, also denote by $\text{Sm}_R$ the category of smooth schemes over $\text{Spec}(R)$, $\text{Sch}^{\text{fp}}$ the category of finitely presented schemes over $\text{Spec}(R)$, $\text{Poly}_R$ the category of polynomial $R$-algebras, and $\mathbb{E}_1\text{-Rings}_{R}$ the category of associative $R$-linear ring spectra.

	{\bf Sheafification.} 
    We use several Grothendieck topologies, including the Zariski, Nisnevich, and cdh topologies. Denote by $L_{\text{Zar}}$, $L_{\text{Nis}}$, and $L_{\text{cdh}}$ the sheafification functors for these topologies.
	
	\medskip
	
	\subsection*{Acknowledgements}
	
	\vspace{-\parindent}
	\hspace{\parindent}

    I am very grateful to Matthew Morrow for sharing many insights on motivic cohomology, and for careful readings of this manuscript. I would also like to thank Jacob Lurie and Georg Tamme for many helpful comments and corrections on a preliminary version of this paper, Ben Antieau, Denis\nobreakdash-Charles Cisinski, Dustin Clausen, Frédéric Déglise, Elden Elmanto, Quentin Gazda, Marc \hbox{Hoyois}, Ryomei Iwasa, Shane Kelly, Niklas Kipp, Arnab Kundu, Shuji Saito, and Georg Tamme for helpful discussions, as well as the anonymous referees for numerous comments and suggestions. This project has received funding from the European Research Council (ERC) under the European Union’s Horizon 2020 research and innovation programme (grant agreement No. 101001474).

	\section{The motivic filtration on topological cyclic homology}\label{sectionmotivicfiltrationonTC}
	
	\vspace{-\parindent}
	\hspace{\parindent}
	
	In this section, we introduce a motivic filtration on the topological cyclic homology of qcqs derived schemes (Definition~\ref{definitionmotivicfiltrationonTC}), whose shifted graded pieces $\Z(i)^{\text{TC}}$ will serve as a building block for the definition of the motivic complexes $\Z(i)^{\text{mot}}$ (Remark~\ref{remarkmaincartesiansquareformotiviccohomology}). 
	
	We first explain how to express topological cyclic homology in terms of its profinite completion and of negative cyclic homology. Following \cite{nikolaus_topological_2018}, and given a qcqs derived scheme~$X$ and a prime number~$p$, the $p$-completed topological cyclic homology $\text{TC}(X;\Z_p)$ of $X$ is constructed from its $p$-completed topological negative cyclic homology $\text{TC}^-(X;\Z_p)$ and its $p$-completed topological periodic cyclic homology $\text{TP}(X;\Z_p)$ (see Section~\ref{subsectionBMSfiltrations}). Following \cite[Lemma~$6.4.3.2$]{dundas_local_2013} and \cite[Section~II.$4$]{nikolaus_topological_2018}, the topological cyclic homology $\text{TC}(X)$ of $X$ is then defined by a natural cartesian square of spectra
	$$\begin{tikzcd}
		\text{TC}(X) \ar[r] \ar[d] & \text{TC}^-(X) \ar[d] \\
		\prod_{p \in \mathbb{P}} \text{TC}(X;\Z_p) \ar[r] & \prod_{p \in \mathbb{P}} \text{TC}^-(X;\Z_p).
	\end{tikzcd}$$
	The comparison map $\text{THH}(X) \rightarrow \text{HH}(X)$, induced by extension of scalars along the map of $\mathbb{E}_{\infty}$-rings $\text{THH}(\Z) \rightarrow \Z$, is $\text{S}^1$-equivariant, and for every integer $n \in \Z$, the kernel and cokernel of the induced map on homotopy groups $\text{THH}_n(X) \rightarrow \text{HH}_n(X)$ are killed by an integer depending only on $n$. In particular, the natural commutative diagram
	$$\begin{tikzcd}
		\text{THH}(X) \ar[r] \ar[d] & \text{HH}(X) \ar[d] \\
		\prod_{p \in \mathbb{P}} \text{THH}(X;\Z_p) \ar[r] & \prod_{p \in \mathbb{P}} \text{HH}(X;\Z_p),
	\end{tikzcd}$$
	is a cartesian square of spectra, which in turn defines a natural cartesian square of spectra
	$$\begin{tikzcd}
		\text{TC}^-(X) \ar[r] \ar[d] & \text{HC}^-(X) \ar[d] \\
		\prod_{p \in \mathbb{P}} \text{TC}^-(X;\Z_p) \ar[r] & \prod_{p \in \mathbb{P}} \text{HC}^-(X;\Z_p)
	\end{tikzcd}$$
	by taking homotopy fixed points $(-)^{h\text{S}^1}$. Composing this cartesian square with the cartesian square defining topological cyclic homology then induces a natural cartesian square
	$$\begin{tikzcd}
		\text{TC}(X) \ar[r] \ar[d] & \text{HC}^-(X) \ar[d] \\
		\prod_{p \in \mathbb{P}} \text{TC}(X;\Z_p) \ar[r] & \prod_{p \in \mathbb{P}} \text{HC}^-(X;\Z_p).
	\end{tikzcd}$$
	We will use this cartesian square to define the motivic filtration on $\text{TC}(X)$ (Definition~\ref{definitionmotivicfiltrationonTC}), by glueing existing filtrations on the three other terms; namely, the HKR and BMS filtration.

    \medskip
	
	\subsection{The HKR filtrations}\label{subsectionHKRfiltrations}
	
	\vspace{-\parindent}
	\hspace{\parindent}
	
	In this subsection, we review the HKR filtrations on Hochschild homology and its variants, as defined, in the generality of qcqs derived schemes, by \cite{antieau_periodic_2019} and \cite[Section~$6.3$]{bhatt_absolute_2022}. Only the HKR filtration on negative cyclic homology $\text{HC}^-(-)$ (Definition~\ref{definitionHKRfiltrationonHC-}) will be used to define the motivic filtration on topological cyclic homology $\text{TC}(-)$ (Definition~\ref{definitionmotivicfiltrationonTC}). We will use the other HKR filtrations of this section in Section~\ref{sectionratonialstructure}.
	
	The following result is \cite[Example~$6.1.3$ and Remarks~$6.1.4$ and $6.1.5$]{bhatt_absolute_2022}.
	
	\begin{proposition}[Tate filtration]\label{propositionTatefiltration}
		Let $X$ be a spectrum equipped with an $\emph{S}^1$-action. Then the Tate construction $X^{t\emph{S}^1} \in \emph{Sp}$ is naturally equipped with a $\Z$-indexed filtration $$\emph{Fil}^\star_{\emph{T}} X^{t\emph{S}^1} \in \emph{FilSp}.$$ This filtration is called the \emph{Tate filtration} on $X^{t\emph{S}^1}$, and satisfies the following properties:
		\begin{enumerate}
			\item The filtration $\emph{Fil}^\star_{\emph{T}} X^{t\emph{S}^1} \in \emph{FilSp}$ is complete.
			\item $\emph{Fil}^0_{\emph{T}} X^{t\emph{S}^1}$ is the homotopy fixed point spectrum $X^{h\emph{S}^1}$, which is thus also equipped with an $\N$\nobreakdash-indexed complete filtration $\emph{Fil}^\star_{\emph{T}} X^{h\emph{S}^1}$, which we call the {\it Tate filtration} on $X^{h\emph{S}^1}$.
			\item For every integer $n \in \Z$, the graded piece $\emph{gr}^n_{\emph{T}} X^{t\emph{S}^1}$ is naturally identified with the spectrum~$X[-2n]$.
		\end{enumerate}
	\end{proposition}
	
	Following \cite[Section~$5$]{bhatt_topological_2019}, a filtered spectrum $\text{Fil}^\star X$ is called {\it connective for the Beilinson $t$\nobreakdash-structure} if for every integer $i \in \Z$, the graded piece $\text{gr}^i X \in \text{Sp}$ is in cohomological degrees at most~$i$. For every integer $i \in \Z$, also denote by $\tau_{\geq i}^{\text{B}}$ the truncation functor for the Beilinson $t$-structure on filtered spectra.
	
	\begin{definition}[Décalage filtration]\label{definitiondécalagefiltration}
		Let $\text{Fil}^\star X \in \text{FilSp}$ be a filtered spectrum. The {\it décalage filtration} on $\text{Fil}^\star X$ is the bifiltered spectrum
		$$\text{Fil}^\star_{\text{B}} \text{Fil}^\star X \in \text{biFilSp}$$
		where, for every integer $i \in \Z$, $\text{Fil}^i_{\text{B}} \text{Fil}^\star X$ is the $i$-connective cover of $\text{Fil}^\star X \in \text{FilSp}$ with respect to the Beilinson $t$-structure on the category of filtered spectra:
		$$\text{Fil}^i_{\text{B}} \text{Fil}^\star X := \tau^{\text{B}}_{\geq i} \text{Fil}^\star X.$$
	\end{definition}
	
	\begin{construction}[HKR filtration on HP]\label{constructionHKRfiltrationonHP}
		For every integer $i \in \Z$, let
		$$\text{Fil}^i_{\text{HKR}} \text{Fil}^\star_{\text{T}}\text{HP}(-) := L_{\text{Zar}} L_{\text{dSch}^{\text{qcqs,op}}/\text{Poly}_{\Z}^{\text{op}}} \text{Fil}^i_{\text{B}} \text{Fil}^\star_{\text{T}} \text{HP}(-),$$
		where $\text{Fil}^\star_{\text{T}} \text{HP}(-)$ is the Tate filtration on periodic cyclic homology of qcqs derived schemes, $\text{Fil}^\star_{\text{B}}$ is the décalage filtration of Definition~\ref{definitiondécalagefiltration}, and the left Kan extension $L_{\text{dSch}^{\text{qcqs,op}}/\text{Poly}_{\Z}^{\text{op}}}$ is taken in the category of filtration-complete filtered spectra.
		The {\it HKR filtration on periodic cyclic homology} of qcqs derived schemes is the functor\footnote{Beware that the Tate filtration $\text{Fil}^\star_{\text{T}} X^{t\text{S}^1}$ (Proposition~\ref{propositionTatefiltration}) on a commutative ring spectrum $X$ equipped with an $\text{S}^1$-action need not have the structure of a commutative filtered spectrum (\cite[Warning~6.1.8]{bhatt_absolute_2022}). Using the $\Z$-linearity of Hochschild homology, it was however proved by Antieau (\cite[Theorem~1.1]{antieau_periodic_2019}, see also \cite{toen_algebres_2011}) that the HKR filtrations on $\text{HP}$, $\text{HC}^-$, and $\text{HH}$ are indeed multiplicative.}
		$$\text{Fil}^\star_{\text{HKR}} \text{HP}(-) : \text{dSch}^{\text{qcqs,op}} \longrightarrow \text{CAlg}(\text{FilSp})$$
		defined as the underlying filtered object of the bifiltered functor $\text{Fil}^\star_{\text{HKR}} \text{Fil}^\star_{\text{T}} \text{HP}(-)$:
		$$\text{Fil}^\star_{\text{HKR}} \text{HP}(-) := {\lim\limits_{\text{ }\longrightarrow n}} \text{Fil}^\star_{\text{HKR}} \text{Fil}^n_{\text{T}} \text{HP}(-).$$ 
	\end{construction}
	
	The following definition is the one which will appear explicitly in the definition of the motivic filtration on topological cyclic homology (Definition~\ref{definitionmotivicfiltrationonTC}).
	
	\begin{definition}[HKR filtration on $\text{HC}^-$]\label{definitionHKRfiltrationonHC-}
		The {\it HKR filtration on negative cyclic homology} of qcqs derived schemes is the functor
		$$\text{Fil}^\star_{\text{HKR}} \text{HC}^-(-) : \text{dSch}^{\text{qcqs,op}} \longrightarrow \text{CAlg}(\text{FilSp})$$
		defined as
		$$\text{Fil}^\star_{\text{HKR}} \text{HC}^-(-) := \text{Fil}^\star_{\text{HKR}} \text{Fil}^0_{\text{T}} \text{HP}(-).$$
	\end{definition}
	
	The following definition is motivated by Proposition~\ref{propositionTatefiltration}$\,(3)$.
	
	\begin{definition}[HKR filtration on HH]\label{definitionHKRfiltrationonHH}
		The {\it HKR filtration on Hochschild homology} of qcqs derived schemes is the functor
		$$\text{Fil}^\star_{\text{HKR}} \text{HH}(-) : \text{dSch}^{\text{qcqs,op}} \longrightarrow \text{CAlg}(\text{FilSp})$$
		defined as
		$$\text{Fil}^\star_{\text{HKR}} \text{HH}(-) := \text{Fil}^\star_{\text{HKR}} \text{gr}^0_{\text{T}} \text{HP}(-).$$
	\end{definition}
	
	Cyclic homology $\text{HC}(-)$ is defined as the homotopy orbits $\text{HH}(-)_{h\text{S}^1}$ of the $\text{S}^1$-action on Hochschild homology $\text{HH}(-)$, and is related to negative cyclic homology $\text{HC}^-(-)$ and periodic cyclic homology $\text{HP}(-)$ by a natural fibre sequence (induced by the norm map for the $\text{S}^1$-action, see for instance \cite[Corollary~I.4.3]{nikolaus_topological_2018})
	$$\text{HC}^-(-) \longrightarrow \text{HP}(-) \longrightarrow \text{HC}(-)[2].$$
	
	\begin{definition}[HKR filtration on HC]\label{definitionHKRfiltrationonHC}
		The {\it HKR filtration on cyclic homology} of qcqs derived schemes is the functor
		$$\text{Fil}^\star_{\text{HKR}} \text{HC}(-) : \text{dSch}^{\text{qcqs,op}} \longrightarrow \text{FilSp}$$
		defined, for every integer $i \in \Z$, by
		$$\text{Fil}^i_{\text{HKR}} \text{HC}(-) := \text{cofib}\Big(\text{Fil}^{i}_{\text{HKR}} \text{HC}^-(-) \longrightarrow \text{Fil}^{i}_{\text{HKR}} \text{HP}(-)\Big)[-2],$$
		where the map on the right hand side is induced by Construction~\ref{constructionHKRfiltrationonHP} and Definition~\ref{definitionHKRfiltrationonHC-}.
	\end{definition}
	
	\begin{remark}[Graded pieces of the HKR filtrations]\label{remarkgradedpiecesoftheHKRfiltrations}
		Let $X$ be a qcqs derived scheme. The main result of \cite{antieau_periodic_2019} describes the graded pieces of the HKR filtrations on $\text{HC}^-(X)$, $\text{HP}(X)$, and $\text{HC}(X)$ in terms of the Hodge-completed derived de Rham cohomology of $X$. In particular,  Definition~\ref{definitionHKRfiltrationonHC} provides a filtered refinement of the fibre sequence
		$$\text{HC}^-(-) \longrightarrow \text{HP}(-) \longrightarrow \text{HC}(-)[2],$$
		which induces on graded pieces, for every integer~$i \in \Z$, a natural fibre sequence
		$$R\Gamma_{\text{Zar}}\big(X,\widehat{\mathbb{L}\Omega}^{\geq i}_{-/\Z}\big)[2i] \longrightarrow R\Gamma_{\text{Zar}}\big(X,\widehat{\mathbb{L}\Omega}_{-/\Z}\big)[2i] \longrightarrow R\Gamma_{\text{Zar}}\big(X,\mathbb{L}\Omega^{<i}_{-/\Z}\big)[2i]$$
		in the derived category $\mathcal{D}(\Z)$.
	\end{remark}
	
	\begin{proposition}[\cite{antieau_periodic_2019,bhatt_absolute_2022}]\label{propositionHKRfiltrationonHCrationalisfinitary}
		For every integer $i \in \Z$, the functor $\emph{Fil}^i_{\emph{HKR}} \emph{HC}(-)$, from animated commutative rings to spectra, is left Kan extended from polynomial $\Z$-algebras, commutes with filtered colimits, and its values are in cohomological degrees at most $-i-1$.
	\end{proposition}
	
	\begin{proof}
		On animated commutative rings, the Tate filtrations $$\text{Fil}^{i}_{\text{HKR}} \text{Fil}^\star_{\text{T}} \text{HC}^-(-) \quad \text{and} \quad \text{Fil}^{i}_{\text{HKR}} \text{Fil}^\star_{\text{T}} \text{HP}(-)$$ are by definition left Kan extended, as complete filtered objects, from polynomial $\Z$-algebras, thus so is the Tate filtration $\text{Fil}^i_{\text{HKR}} \text{Fil}^\star_{\text{T}} \text{HC}(-)$. The Tate filtration $\text{Fil}^i_{\text{HKR}} \text{Fil}^\star_{\text{T}} \text{HC}(-)$ is also finite by construction, so the functor $\text{Fil}^i_{\text{HKR}} \text{HC}(-)$ is left Kan extended on animated commutative rings from polynomial $\Z$-algebras. In particular, the functor $$\text{Fil}^i_{\text{HKR}} \text{HC}(-)$$ commutes with filtered colimits of animated commutative rings. For every integer $j \in \Z$, the $j^{\text{th}}$ graded piece of the filtration $\text{Fil}^\star_{\text{HKR}} \text{HC}(-)$ is naturally identified with the functor $R\Gamma_{\text{Zar}}\big(-,\mathbb{L}\Omega^{< j}_{-/\Z}\big)[2j]$ (\cite{antieau_periodic_2019}), whose values are in degrees at most $-j-1$ on animated commutative rings. The filtration $\text{Fil}^\star_{\text{HKR}} \text{HC}(-)$ is moreover complete on animated commutative rings (\cite[Remark~$6.3.5$]{bhatt_absolute_2022}), hence the desired connectivity result.
	\end{proof}
	
	\begin{lemma}[Completeness of the HKR filtrations, after \cite{bhatt_absolute_2022}]\label{lemmaHKRfiltrationonHC-isalwayscomplete}
		Let $X$ be a qcqs derived scheme. Then the HKR filtrations $$\emph{Fil}^\star_{\emph{HKR}} \emph{HP}(X), \text{ }\emph{Fil}^\star_{\emph{HKR}} \emph{HC}^-(X), \text{ } \emph{Fil}^\star_{\emph{HKR}} \emph{HH}(X), \text{ and } \emph{Fil}^\star_{\emph{HKR}} \emph{HC}(X)$$ 
		are complete.
	\end{lemma}

	\begin{proof}
		The result for $\text{HC}^-$ and $\text{HH}$ is a direct consequence of \cite[Remark~$6.3.5$]{bhatt_absolute_2022}. The result for $\text{HC}$ is a consequence of the connectivity result of Proposition~\ref{propositionHKRfiltrationonHCrationalisfinitary}. By Definition~\ref{definitionHKRfiltrationonHC}, the result for $\text{HP}$ is then a consequence of the result for $\text{HC}^-$ and $\text{HC}$.
	\end{proof}
	
	\begin{remark}[Variant over $\Q$]\label{remarkHHandvariantsrelativetoQ}
		Let $X$ be a qcqs derived scheme. By the base change property for Hochschild homology, the natural map
		$$\text{HH}(X) \otimes_{\Z} \Q \longrightarrow \text{HH}(X_{\Q}/\Q)$$
		is an equivalence in the derived category $\mathcal{D}(\Q)$, where $\text{HH}(-/\Q)$ is Hochschild homology relative to~$\Q$. Applying the functors $(-)^{h\text{S}^1}$, $(-)^{t\text{S}^1}$, and $(-)_{h\text{S}^1}$ to this Hochschild homology relative to $\Q$ induces relative variants $\text{HC}^-(-_{\Q}/\Q)$ of negative cyclic homology, $\text{HP}(-_{\Q}/\Q)$ of periodic cyclic homology, and $\text{HC}(-_{\Q}/\Q)$ of cylic homology. One can then define similar HKR filtrations 
		$$\text{Fil}^\star_{\text{HKR}} \text{HH}(X_{\Q}/\Q), \text{ } \text{Fil}^\star_{\text{HKR}} \text{HC}^-(X_{\Q}/\Q), \text{ } \text{Fil}^\star_{\text{HKR}} \text{HP}(-_{\Q}/\Q), \text{ and } \text{Fil}^\star_{\text{HKR}} \text{HC}(-_{\Q}/\Q),$$
		on these functors, whose graded pieces are versions of derived de Rham cohomology relative to $\Q$.
	\end{remark}
	
	To introduce and study the motivic filtration on topological cyclic homology (Definition~\ref{definitionmotivicfiltrationonTC}), we will need some $p$-complete variants of the previous HKR filtrations.
	
	\begin{definition}[HKR filtration on $\text{HC}^-(-;\Z_p)$]\label{definitionHKRfiltrationHC-pcompleted}
		Let $p$ be a prime number. The {\it HKR filtration on $p$-completed negative cyclic homology} of qcqs derived schemes is the functor
		$$\text{Fil}^\star_{\text{HKR}} \text{HC}^-(-;\Z_p) : \text{dSch}^{\text{qcqs,op}} \longrightarrow \text{CAlg}(\text{FilSp})$$
		defined as
		$$\text{Fil}^\star_{\text{HKR}} \text{HC}^-(-;\Z_p) := \big( \text{Fil}^\star_{\text{HKR}} \text{HC}^-(-) \big)^\wedge_p.$$
	\end{definition}
	
	\begin{remark}
		The HKR filtrations on $\text{HP}(-;\Z_p)$, $\text{HC}(-;\Z_p)$, and $\text{HH}(-;\Z_p)$ of qcqs derived schemes are defined as in Definition~\ref{definitionHKRfiltrationHC-pcompleted}, where $\text{HC}^-(-;\Z_p)$ is replaced by $\text{HP}(-)$, $\text{HC}(-)$, or $\text{HH}(-)$. In particular, for every qcqs derived scheme $X$, Definition~\ref{definitionHKRfiltrationonHC} induces a fibre sequence of filtered spectra
		$$\text{Fil}^\star_{\text{HKR}} \text{HC}^-(X;\Z_p) \longrightarrow \text{Fil}^\star_{\text{HKR}} \text{HP}(X;\Z_p) \longrightarrow \text{Fil}^{\star}_{\text{HKR}} \text{HC}(X;\Z_p)[2].$$
	\end{remark}
	
	\begin{lemma}\label{lemmaHKRfiltrationproductpcompletionsiscomplete}
		Let $X$ be a qcqs derived scheme. Then the filtrations $$\prod_{p \in \mathbb{P}} \emph{Fil}^\star_{\emph{HKR}} \emph{HP}(X;\Z_p), \quad \prod_{p \in \mathbb{P}} \emph{Fil}^\star_{\emph{HKR}} \emph{HC}^-(X;\Z_p), \quad \prod_{p \in \mathbb{P}} \emph{Fil}^\star_{\emph{HKR}} \emph{HH}(X;\Z_p),$$ $$\emph{ and } \prod_{p \in \mathbb{P}} \emph{Fil}^\star_{\emph{HKR}} \emph{HC}(X;\Z_p)$$ are complete.
	\end{lemma}
	
	\begin{proof}
		The collection of complete filtered spectra is closed under limits in the category of filtered spectra, so this is a consequence of Lemma~\ref{lemmaHKRfiltrationonHC-isalwayscomplete}.
	\end{proof}
	
	\begin{remark}[Exhaustivity of the HKR filtrations]\label{remarkexhaustivityoftheHKRfiltrations}
		The HKR filtrations $\text{Fil}^\star_{\text{HKR}} \text{HC}^-$ and $\text{Fil}^\star_{\text{HKR}} \text{HP}$ are not exhaustive on general qcqs derived schemes (\cite[Remark~$6.3.6$]{bhatt_absolute_2022}). For the purpose of the motivic filtration on algebraic $K$-theory (Definition~\ref{definitionmotivicfiltrationonKtheoryofschemes}), we will however only need the fact that the HKR filtration $\text{Fil}^\star_{\text{HKR}} \text{HC}$ is always $\N$-indexed, and in particular exhaustive.
	\end{remark}

    \medskip
	
	\subsection{The BMS filtrations}\label{subsectionBMSfiltrations}
	
	\vspace{-\parindent}
	\hspace{\parindent}
	
	Let $p$ be a prime number. In this subsection, we review the BMS filtrations on $p$-completed topological Hochschild homology $\text{THH}(-;\Z_p)$ and its variants, as defined in \cite{bhatt_topological_2019} for $p$\nobreakdash-complete $p$-quasisyntomic rings, and generalised in \cite{antieau_beilinson_2020} to $p$-complete rings and in \cite[Section~$6.2$]{bhatt_absolute_2022} to animated commutative rings. Only the BMS filtration on $p$-completed topological cyclic homology $\text{TC}(-;\Z_p)$ (Definition~\ref{definitionBMSfiltrationonTCpcompletedofqcqsderivedschemes}) will appear in the definition of the motivic filtration on topological cyclic homology $\text{TC}(-)$ (Definition~\ref{definitionmotivicfiltrationonTC}). The other BMS filtrations are necessary to construct the BMS filtration on $p$-completed topological cyclic homology $\text{TC}(-;\Z_p)$.
	
	\begin{construction}[BMS filtration on $\text{Fil}^\star_{\text{T}} \text{TP}(-;\Z_p)$]\label{constructionBMSfiltrationonTPpcompleted}
		Topological Hochschild homology $\text{THH}(-)$ of qcqs derived schemes admits a natural $\text{S}^1$-action, inducing a natural Tate filtration $\text{Fil}^\star_{\text{T}} \text{TP}(-)$ on topological periodic cyclic homology $\text{TP}(-) := \text{THH}(-)^{t\text{S}^1}$ (Proposition~\ref{propositionTatefiltration}). The Tate filtration 
		$$\text{Fil}^\star_{\text{T}} \text{TP}(-;\Z_p) : \text{dSch}^{\text{qcqs,op}} \longrightarrow \text{FilSp}$$
		is then defined as the $p$-completion of the Tate filtration $\text{Fil}^\star_{\text{T}} \text{TP}(-)$. For every quasiregular semiperfectoid ring $R$ and every integer $i \in \Z$, define the filtered spectrum
		$$\text{Fil}^i_{\text{BMS}} \text{Fil}^\star_{\text{T}} \text{TP}(R;\Z_p) := \tau_{\geq 2i} \text{Fil}^\star_{\text{T}} \text{TP}(R;\Z_p).$$
		The filtered object
		$$\text{Fil}^i_{\text{BMS}} \text{Fil}^\star_{\text{T}} \text{TP}(-;\Z_p) : \text{dSch}^{\text{qcqs,op}} \longrightarrow \text{FilSp}$$
		is then first defined on $p$-quasisyntomic rings as the unique such functor satisfying $p$-complete faithfully flat descent (the existence and unicity of such a functor is \cite[Proposition~$4.31$]{bhatt_topological_2019}). In general, polynomial $\Z$-algebras are $p$-quasisyntomic rings, and this filtered object is defined as the Zariski sheafification of its left Kan extension from polynomial $\Z$-algebras
		$$\text{Fil}^i_{\text{BMS}} \text{Fil}^\star_{\text{T}} \text{TP}(-;\Z_p) := L_{\text{Zar}} L_{\text{AniRings}/\text{Poly}_{\Z}} \text{Fil}^i_{\text{BMS}} \text{Fil}^\star_{\text{T}} \text{TP}(-;\Z_p),$$
		where the left Kan extension is taken in the category of $p$-complete filtration-complete spectra. By \cite[Theorem~$6.2.4$]{bhatt_absolute_2022}, the resulting functor is still given by the double-speed Postnikov filtration on quasiregular semiperfectoid rings and, as a functor from animated commutative rings to $p$-complete filtration-complete spectra, commutes with sifted colimits and satisfies $p$-complete faithfully flat descent.
	\end{construction}

	\begin{remark}\label{remarkderivedpcompletionwithBMS}
		The BMS filtrations were first defined in \cite{bhatt_topological_2019} in the generality of $p$\nobreakdash-com\-plete $p$\nobreakdash-quasisyntomic rings. On general animated commutative rings $R$, the BMS filtrations, by construction, depend only on the $p$-completion of $R$ --and in particular vanish on animated commutative $\Z[\tfrac{1}{p}]$-algebras. Here the $p$-completion is the derived $p$-completion, even on classical commutative rings. On commutative rings with bounded $p$-power torsion ({\it e.g.}, on $p$-quasisyntomic rings), the derived and classical $p$-completions naturally coincide, and there is no conflict between the two definitions.
	\end{remark}
	
	\begin{definition}[BMS filtration on $\text{TP}(-;\Z_p)$]\label{definitionBMSfiltrationonTPpcompletedofqcqsderivedschemes}
		The {\it BMS filtration on $p$-completed topological periodic cyclic homology} of qcqs derived schemes is the functor\footnote{As in Section~\ref{subsectionHKRfiltrations}, the multiplicativity of this BMS filtration, as defined here in terms of the Tate filtrations, is a priori unclear. However, the BMS filtrations on $p$-complete THH, TP, $\text{TC}^-$, and TC are all multiplicative, as a consequence of their alternative description as a $p$-quasisyntomic-local double-speed Postnikov filtration (\cite[Theorem~1.12]{bhatt_topological_2019}) and of the $\mathbb{E}_\infty$-ring structures on THH, TP, $\text{TC}^-$, and TC.} 
		$$\text{Fil}^\star_{\text{BMS}} \text{TP}(-;\Z_p) : \text{dSch}^{\text{qcqs,op}} \longrightarrow \text{CAlg}(\text{FilSp})$$
		defined as the underlying filtered object of the bifiltered functor $\text{Fil}^\star_{\text{BMS}} \text{Fil}^\star_{\text{T}} \text{TP}(-;\Z_p)$ of Construction~\ref{constructionBMSfiltrationonTPpcompleted}:
		$$\text{Fil}^\star_{\text{BMS}} \text{TP}(-;\Z_p) := {\lim\limits_{\text{ }\longrightarrow n}} \text{Fil}^\star_{\text{BMS}} \text{Fil}^n_{\text{T}} \text{TP}(-;\Z_p).$$
	\end{definition}
	
	Topological negative cyclic homology is to topological periodic cyclic homology what negative cyclic homology is to periodic cyclic homology. Given Definition~\ref{definitionBMSfiltrationonTPpcompletedofqcqsderivedschemes}, the following definition then mimics Definition~\ref{definitionHKRfiltrationonHC-}.
	
	\begin{definition}[BMS filtration on $\text{TC}^-(-;\Z_p)$]\label{definitionBMSfiltrationonTC-pcompleted}
		The {\it BMS filtration on $p$-completed topological negative cyclic homology} of qcqs derived schemes is the functor
		$$\text{Fil}^\star_{\text{BMS}} \text{TC}^-(-;\Z_p) : \text{dSch}^{\text{qcqs,op}} \longrightarrow \text{CAlg}(\text{FilSp})$$
		defined as
		$$\text{Fil}^\star_{\text{BMS}} \text{TC}^-(-;\Z_p) := \text{Fil}^\star_{\text{BMS}} \text{Fil}^0_{\text{T}} \text{TP}(-;\Z_p).$$
	\end{definition}
	
	Similarly, topological Hochschild homology is to topological periodic and topological negative cyclic homologies what Hochschild homology is to periodic and negative cyclic homologies, and the following definition mimics Definition~\ref{definitionHKRfiltrationonHH}. 
	
	\begin{definition}[BMS filtration on $\text{THH}(-;\Z_p)$]\label{definitionBMSfiltrationonTHHpcompleted}
		The {\it BMS filtration on $p$-completed topological Hochschild homology} of qcqs derived schemes is the functor
		$$\text{Fil}^\star_{\text{BMS}} \text{THH}(-;\Z_p) : \text{dSch}^{\text{qcqs,op}} \longrightarrow \text{CAlg}(\text{FilSp})$$
		defined as
		$$\text{Fil}^\star_{\text{BMS}} \text{THH}(-;\Z_p) := \text{Fil}^\star_{\text{BMS}} \text{gr}^0_{\text{T}} \text{TP}(-;\Z_p).$$
	\end{definition}
	 
    Topological cyclic homology is however not defined as the homotopy orbits $\text{THH}(-)_{h\text{S}^1}$ of the $\text{S}^1$-action on topogical Hochschild homology $\text{THH}(-)$. Following \cite{nikolaus_topological_2018}, it is rather defined, after $p$-completion, by a fibre sequence
	$$\text{TC}(-;\Z_p) \longrightarrow \text{TC}^-(-;\Z_p) \xlongrightarrow{\phi_p - \text{can}} \text{TP}(-;\Z_p).$$
	Unwinding the previous definitions, the map $\phi_p - \text{can} : \text{TC}^-(-;\Z_p) \rightarrow \text{TP}(-;\Z_p)$
	admits a unique refinement as a filtered map
	$$\phi_p - \text{can} : \text{Fil}^\star_{\text{BMS}} \text{TC}^-(-;\Z_p) \longrightarrow \text{Fil}^\star_{\text{BMS}} \text{TP}(-;\Z_p).$$
	
	\begin{definition}[BMS filtration on $\text{TC}(-;\Z_p)$]\label{definitionBMSfiltrationonTCpcompletedofqcqsderivedschemes}
		The {\it BMS filtration on $p$-completed topological cyclic homology} of qcqs derived schemes is the functor
		$$\text{Fil}^\star_{\text{BMS}} \text{TC}(-;\Z_p) : \text{dSch}^{\text{qcqs,op}} \longrightarrow \text{CAlg}(\text{FilSp})$$
		defined as
		$$\text{Fil}^\star_{\text{BMS}} \text{TC}(-;\Z_p) := \text{fib}\Big(\phi_p - \text{can} : \text{Fil}^\star_{\text{BMS}} \text{TC}^-(-;\Z_p) \longrightarrow \text{Fil}^\star_{\text{BMS}} \text{TP}(-;\Z_p)\Big).$$
	\end{definition}
	
	The BMS filtration on $p$-completed topological cyclic homology is always complete, as a consequence of the connectivity result \cite[Theorem~5.1]{antieau_beilinson_2020}. We will need the following slightly more precise result when studying the completeness of the motivic filtration on algebraic $K$-theory.
	
	\begin{lemma}\label{lemmaBMSfiltrationproductallprimesiscomplete}
		Let $X$ be a qcqs derived scheme. Then the filtrations $\prod_{p \in \mathbb{P}} \emph{Fil}^\star_{\emph{BMS}} \emph{TC}(X;\Z_p)$ and $\big(\prod_{p \in \mathbb{P}} \emph{Fil}^\star_{\emph{BMS}} \emph{TC}(X;\Z_p)\big)_{\Q}$ are complete. More precisely, for every integer $i \in \Z$, the values of the presheaves $\prod_{p \in \mathbb{P}} \emph{Fil}^i_{\emph{BMS}} \emph{TC}(-;\Z_p)$ and $\big(\prod_{p \in \mathbb{P}} \emph{Fil}^i_{\emph{BMS}} \emph{TC}(-;\Z_p)\big)_{\Q}$ are in cohomological degrees at most $-i+1$ on affine derived schemes.
	\end{lemma}
	
	\begin{proof}
		The presheaves $$\prod_{p \in \mathbb{P}} \text{Fil}^\star_{\text{BMS}} \text{TC}(-;\Z_p) \quad \text{and} \quad \Big(\prod_{p \in \mathbb{P}} \text{Fil}^\star_{\text{BMS}} \text{TC}(-;\Z_p)\Big)_{\Q}$$ are Zariski sheaves by construction, so it suffices to prove the result for affine derived schemes $X$. Let $R$ be an animated commutative ring, and $i \in \Z$ be an integer. The spectrum $\text{Fil}^i_{\text{BMS}} \text{TC}(R;\Z_p)$ is in cohomological degrees at most $-i+1$ for every prime number~$p$ (\cite[Theorem~$5.1$]{antieau_beilinson_2020}). Taking the product over all primes $p$ and rationalisation, this implies that the spectra $\prod_{p \in \mathbb{P}} \text{Fil}^i_{\text{BMS}} \text{TC}(R;\Z_p)$ and $\Big(\prod_{p \in \mathbb{P}} \text{Fil}^i_{\text{BMS}} \text{TC}(R;\Z_p)\Big)_{\Q}$ are also in cohomological degrees at most $-i+1$, which implies that the associated filtrations are complete. 
	\end{proof}
	
	\begin{remark}[Exhaustivity of the BMS filtrations]\label{remarkexhaustivityofBMSfiltrations}
		The BMS filtrations $\text{Fil}^\star_{\text{BMS}} \text{TP}(-;\Z_p)$ and $\text{Fil}^{\star}_{\text{BMS}} \text{TC}^-(-;\Z_p)$ are not exhaustive on general qcqs derived schemes (\cite[Warning~$6.2.7$]{bhatt_absolute_2022}). For the purpose of the motivic filtration on algebraic $K$-theory (Definition~\ref{definitionmotivicfiltrationonKtheoryofschemes}), we will however only need the fact that the BMS filtration $\text{Fil}^{\star}_{\text{BMS}} \text{TC}(-;\Z_p)$ is always $\N$\nobreakdash-indexed (\cite[proof of Proposition~$7.16$]{bhatt_topological_2019}), and in particular exhaustive.
	\end{remark}
	
	We refer to \cite{bhatt_topological_2019,bhatt_prisms_2022,bhatt_absolute_2022} for the relation between prismatic cohomology and the graded pieces of the BMS filtrations on $\text{TP}(-;\Z_p)$, $\text{TC}^-(-;\Z_p)$, and $\text{THH}(-;\Z_p)$. We only define here the shifted graded pieces of the BMS filtration on $\text{TC}(-;\Z_p)$, which are a version of syntomic cohomology (see Remark~\ref{remarksyntomiccohomologyBMSandBL}), and which will serve as a building block for the $p$-adic motivic complexes (Corollary~\ref{corollarymainpadicstructureongradeds}).
	
	\begin{definition}[BMS syntomic cohomology]\label{definitionsyntomiccohomologyintermsofTC}
		For every integer $i \in \Z$, the {\it syntomic complex}
		$$\Z_p(i)^{\text{BMS}}(-) : \text{dSch}^{\text{qcqs,op}} \longrightarrow \mathcal{D}(\Z)$$
		is the functor defined as the shifted graded piece of the BMS filtration on $\text{TC}(-;\Z_p)$:
		$$\Z_p(i)^{\text{BMS}}(-) := \text{gr}^i_{\text{BMS}} \text{TC}(-;\Z_p)[-2i].$$
	\end{definition}
	
	\begin{remark}\label{remarksyntomiccohomologyBMSandBL}
		Syntomic cohomology $\Z_p(i)^{\text{syn}}(X)$ of qcqs derived (formal) schemes $X$ is defined in \cite[Section~$8.4$]{bhatt_absolute_2022} (see also Notation~\ref{notationsyntomiccohomology}), in terms of the syntomic complexes of Definition~\ref{definitionsyntomiccohomologyintermsofTC} and of étale cohomology. From this perspective, the syntomic complexes $\Z_p(i)^{\text{BMS}}(X)$ of Definition~\ref{definitionsyntomiccohomologyintermsofTC} correspond to the syntomic cohomology $\Z_p(i)^{\text{syn}}(\mathfrak{X})$ of the derived $p$-adic formal scheme $\mathfrak{X}$ associated to $X$.
	\end{remark}
	
	\begin{theorem}\label{theoremBMSfiltrationonTCpcompletedsatisfiesquasisyntomicdescentandisLKEfrompolynomialalgebras}
		\begin{enumerate}
			\item (\cite{bhatt_topological_2019,bhatt_absolute_2022}) The functor $\emph{Fil}^\star_{\emph{BMS}} \emph{TC}(-;\Z_p)$, viewed as a functor from $p$\nobreakdash-quasisyntomic rings to $p$-complete filtered spectra, satisfies descent for the $p$\nobreakdash-quasisyntomic topology.
			\item (\cite{antieau_beilinson_2020,bhatt_absolute_2022}) The functor $\emph{Fil}^\star_{\emph{BMS}} \emph{TC}(-;\Z_p)$, viewed as a functor from animated commutative rings to $p$-complete filtered spectra, is left Kan extended from polynomial $\Z$-algebras.
		\end{enumerate}
	\end{theorem}
	
	\begin{proof}
		$(1)$ The filtration $\text{Fil}^\star_{\text{BMS}} \text{TC}(-;\Z_p)$ is complete on $p$-quasisyntomic rings (Lemma~\ref{lemmaBMSfiltrationproductallprimesiscomplete}), so it suffices to prove the result on graded pieces. The result on graded pieces is a special case of \cite[Proposition~$7.4.7$]{bhatt_absolute_2022}.
		
		$(2)$ By \cite[Theorem~$5.1\,(2)$]{antieau_beilinson_2020}, the functor $\text{Fil}^\star_{\text{BMS}} \text{TC}(-;\Z_p)$, viewed as a functor from $p$\nobreakdash-complete animated commutative rings to $p$-complete filtered spectra, is left Kan extended from $p$-complete polynomial $\Z$-algebras.\footnote{More precisely, it is proved to be left Kan extended from $p$\nobreakdash-complete polynomial $\Z$-algebras to $p$\nobreakdash-complete $p$\nobreakdash-quasisyntomic rings. By definition, $p$\nobreakdash-quasisyntomic rings have bounded $p$-power torsion. Hence, their derived and classical $p$-completions are naturally identified, and the left Kan extension to $p$-complete animated commutative rings agrees with the left Kan extension to $p$-complete classical rings on $p$-complete $p$-quasisyntomic rings.} Let $R$ be an animated commutative ring, and $R^\wedge_p$ be its (derived) $p$-completion. The natural map
		$$\text{Fil}^\star_{\text{BMS}} \text{TC}(R;\Z_p) \longrightarrow \text{Fil}^\star_{\text{BMS}} \text{TC}(R^\wedge_p;\Z_p)$$
		is an equivalence of filtered spectra. Indeed, the filtrations $$\text{Fil}^\star_{\text{BMS}} \text{TC}(R;\Z_p) \quad \text{and} \quad \text{Fil}^\star_{\text{BMS}} \text{TC}(R^\wedge_p;\Z_p)$$ are $\N$-indexed and complete (Lemma~\ref{lemmaBMSfiltrationproductallprimesiscomplete} and Remark~\ref{remarkexhaustivityofBMSfiltrations}), so it suffices to prove the result on graded pieces, where this is a direct consequence of \cite[Corollary~$7.4.11$]{bhatt_absolute_2022}. This implies the desired left Kan extension property.
	\end{proof}
	
	\begin{corollary}\label{corollaryBMSsyntomiccohomologyhasquasisyntomicdescentandLKEfrompolynomialZalgebras}
		Let $i \in \Z$ be an integer.
		\begin{enumerate}
			\item The functor $\Z_p(i)^{\emph{BMS}}(-)$, viewed as a functor from $p$-quasisyntomic rings to $p$-complete objects in the derived category $\mathcal{D}(\Z)$, satisfies descent for the $p$\nobreakdash-quasisyntomic topology.
			\item The functor $\Z_p(i)^{\emph{BMS}}(-)$, viewed as a functor from animated commutative rings to $p$\nobreakdash-complete objects in the derived category $\mathcal{D}(\Z)$, is left Kan extended from polynomial $\Z$-algebras.
		\end{enumerate}
	\end{corollary}
	
	\begin{proof}
		$(1)$ was already part of the proof of Theorem~\ref{theoremBMSfiltrationonTCpcompletedsatisfiesquasisyntomicdescentandisLKEfrompolynomialalgebras}\,$(1)$. 
		
		$(2)$ is a direct consequence of Theorem~\ref{theoremBMSfiltrationonTCpcompletedsatisfiesquasisyntomicdescentandisLKEfrompolynomialalgebras}\,$(2)$.
	\end{proof}
	
	\begin{theorem}[\cite{antieau_beilinson_2020}]\label{theoremAMMNrigidity}
		Let $(A,I)$ be a henselian pair of commutative rings. Then for any integers $i \geq 0$ and $k \geq 1$, the fibre of the natural map
		$$\Z/p^k(i)^{\emph{BMS}}(A) \longrightarrow \Z/p^k(i)^{\emph{BMS}}(A/I)$$
		in the derived category $\mathcal{D}(\Z/p^k)$ is in degrees at most $i$.
	\end{theorem}
	
	\begin{proof}
		By \cite[Theorem~$5.2$]{antieau_beilinson_2020}, for every henselian pair $(A,I)$ such that the commutative rings $A$ and $A/I$ are (classically) $p$-complete, the fibre of the natural map $$\Z/p^k(i)^{\text{BMS}}(A) \longrightarrow \Z/p^k(i)^{\text{BMS}}(A/I)$$ is in degrees at most $i$. The proof of \cite[Theorem~$5.2$]{antieau_beilinson_2020} proves more generally that for $(A,I)$ a general henselian pair of commutative rings, the fibre of the natural map $$\Z/p^k(i)^{\text{BMS}}(A^\wedge_p) \longrightarrow \Z/p^k(i)^{\text{BMS}}((A/I)^\wedge_p)$$ where $(-)^\wedge_p$ is the derived $p$\nobreakdash-completion, is in degrees at most $i$. By \cite[Corollary~$7.4.11$]{bhatt_absolute_2022} (see also the proof of Theorem~\ref{theoremBMSfiltrationonTCpcompletedsatisfiesquasisyntomicdescentandisLKEfrompolynomialalgebras}\,$(2)$), the natural map $\Z/p^k(i)^{\text{BMS}}(-) \rightarrow \Z/p^k(i)^{\text{BMS}}((-)^\wedge_p)$ is an equivalence on animated commutative rings, hence for every henselian pair $(A,I)$ of commutative rings, the fibre of the natural map $$\Z/p^k(i)^{\text{BMS}}(A) \longrightarrow \Z/p^k(i)^{\text{BMS}}(A/I)$$ is in degrees at most $i$.
	\end{proof}
	
	\subsection{The motivic filtration on TC}\label{subsectionmotivicfiltrationonTC}
	
	\vspace{-\parindent}
	\hspace{\parindent}
	
	In this subsection, we introduce the motivic filtration on topological cyclic homology $\text{TC}(-)$ of general qcqs derived schemes (Definition~\ref{definitionmotivicfiltrationonTC}).
	
	The following proposition is \cite[Proposition~$6.4.1$]{bhatt_absolute_2022}.
	
	\begin{proposition}[\cite{bhatt_absolute_2022}]\label{propositionfilteredmapTC-pcompletedtoHC-pcompleted}
		Let $p$ be a prime number. The map $$\emph{Fil}^\star_{\emph{T}} \emph{TP}(-;\Z_p) \longrightarrow \emph{Fil}^\star_{\emph{T}} \emph{HP}(-;\Z_p),$$ viewed as a map of filtered spectra-valued presheaves on the category of qcqs derived schemes, admits a unique, multiplicative extension to a map of bifiltered presheaves of spectra
		$$\emph{Fil}^\star_{\emph{BMS}} \emph{Fil}^\star_{\emph{T}} \emph{TP}(-;\Z_p) \longrightarrow \emph{Fil}^\star_{\emph{HKR}} \emph{Fil}^\star_{\emph{T}} \emph{HP}(-;\Z_p).$$
	\end{proposition}
	
	\begin{construction}[BMS-HKR comparison map]\label{constructionfilteredmapTCpcompletedtoHC-pcompleted}
		Let $p$ be a prime number. The {\it BMS-HKR comparison map} is the map
		$$\text{Fil}^\star_{\text{BMS}} \text{TC}(-;\Z_p) \longrightarrow \text{Fil}^\star_{\text{HKR}} \text{HC}^{-}(-;\Z_p)$$
		of functors from (the opposite category of) qcqs derived schemes to the category of multiplicative filtered spectra defined as the composite
		$$\text{Fil}^\star_{\text{BMS}} \text{TC}(-;\Z_p) \longrightarrow \text{Fil}^\star_{\text{BMS}} \text{TC}^{-}(-;\Z_p) \longrightarrow \text{Fil}^\star_{\text{HKR}} \text{HC}^{-}(-;\Z_p)$$
		of the maps given by Definition~\ref{definitionBMSfiltrationonTCpcompletedofqcqsderivedschemes}, and Proposition~\ref{propositionfilteredmapTC-pcompletedtoHC-pcompleted} after restricting to the zeroth step of the Tate filtration.
	\end{construction}
	
	\begin{definition}[Motivic filtration on TC]\label{definitionmotivicfiltrationonTC}
		The {\it motivic filtration on topological cyclic homology} of qcqs derived schemes
		$$\text{Fil}^\star_{\text{mot}} \text{TC}(-) : \text{dSch}^{\text{qcqs},\text{op}} \longrightarrow \text{CAlg}(\text{FilSp})$$
		is the functor defined by the cartesian square in filtered $\mathbb{E}_\infty$-rings
		$$\begin{tikzcd}
			\text{Fil}^\star_{\text{mot}} \text{TC}(-) \ar[r] \ar[d] & \text{Fil}^\star_{\text{HKR}} \text{HC}^{-}(-) \ar[d] \\
			\prod_{p \in \mathbb{P}} \text{Fil}^\star_{\text{BMS}} \text{TC}(-;\Z_p) \ar[r] & \prod_{p \in \mathbb{P}} \text{Fil}^\star_{\text{HKR}} \text{HC}^{-}(-;\Z_p),
		\end{tikzcd}$$
		where the bottom horizontal map is the map of Construction~\ref{constructionfilteredmapTCpcompletedtoHC-pcompleted}, and the right vertical map is profinite completion.\footnote{See also Remark~\ref{remarkcomparisontoBLglobalfiltrations} for an alternative definition of the filtration $\text{Fil}^\star_{\text{mot}} \text{TC}$.} For every integer $i \in \Z$, also define the functor
		$$\Z(i)^{\text{TC}}(-) : \text{dSch}^{\text{qcqs},\text{op}} \longrightarrow \mathcal{D}(\Z)\footnote{Every value of the functor $\Z(i)^{\text{TC}}(-)$ has a natural module structure over the $\mathbb{E}_{\infty}$-ring $\Z(0)^{\text{TC}}(\Z)$, which, by unwinding the definition, is naturally identified with the $\mathbb{E}_{\infty}$-ring $\text{H}\!\Z$. This implies that the spectra-valued functor $\Z(i)^{\text{TC}}(-)$ takes values in $\text{H}\!\Z$-linear spectra, {\it i.e.}, in the derived category $\mathcal{D}(\Z)$.}$$
		as the shifted graded piece of this motivic filtration:
		$$\Z(i)^{\text{TC}}(-) := \text{gr}^i_{\text{mot}} \text{TC}(-)[-2i].$$
	\end{definition}

	\begin{remark}[Comparison to \cite{elmanto_motivic_2023}]\label{remarkcomparisontoEMfiltrationonTCoverafield}
        In the remark, we explain why the previous filtration on topological cyclic homology recovers the filtrations $\text{Fil}^\star_{\text{HKR}} \text{HC}^-(X/\Q)$ for characteristic zero schemes and $\text{Fil}^\star_{\text{BMS}} \text{TC}(X;\Z_p)$ for characteristic $p$ schemes used by Elmanto--Morrow to define motivic cohomology in equicharacteristic (\cite[Definitions~4.9 and~4.28]{elmanto_motivic_2023}).
        
		For every qcqs derived scheme $X$ over $\Q$, the filtered spectrum $\text{Fil}^\star_{\text{HKR}} \text{HC}^-(X)$ is $\Q$-linear by construction, so its profinite completion vanishes. The filtration $\prod_{p \in \mathbb{P}} \text{Fil}^\star_{\text{BMS}} \text{TC}(X;\Z_p)$ also vanishes (Remark~\ref{remarkderivedpcompletionwithBMS}), and the natural map
		$$\text{Fil}^\star_{\text{mot}} \text{TC}(X) \longrightarrow \text{Fil}^\star_{\text{HKR}} \text{HC}^-(X/\Q)$$
		is then an equivalence of filtered spectra.
		
		Similarly, for every prime number $p$ and every qcqs derived scheme $X$ over~$\F_p$, the filtered spectrum $\text{Fil}^\star_{\text{HKR}} \text{HC}^-(X)$ is $\Z$-linear and $p$-complete, so it is naturally identified with its profinite completion. Again using Remark~\ref{remarkderivedpcompletionwithBMS}, the natural map
		$$\text{Fil}^\star_{\text{mot}} \text{TC}(X) \longrightarrow \text{Fil}^\star_{\text{BMS}} \text{TC}(X;\Z_p)$$
		is then an equivalence of filtered spectra.
	\end{remark}

	\begin{remark}[Comparison to \cite{morin_topological_2021,bhatt_absolute_2022}]\label{remarkcomparisontoBLglobalfiltrations}
		In \cite[Section~3.3]{morin_topological_2021} and independently in \cite[Section~$6.4$]{bhatt_absolute_2022} (and later revisited by Hahn--Raksit--Wilson \cite[Section~4.2]{hahn_motivic_2022}), Morin and Bhatt--Lurie define filtered spectra $\text{Fil}^\star_{\text{mot}} \text{TC}^-(X)$ and $\text{Fil}^\star_{\text{mot}} \text{TP}(X)$ for qcqs derived schemes $X$, with shifted graded pieces called the global prismatic complexes $\mathcal{N}^{\geq i} \widehat{\Prism}^{\text{gl}}_X\{i\}$ and $\widehat{\Prism}^{\text{gl}}_X\{i\}$ respectively. These filtrations can be used to obtain an alternative definition of the $\text{Fil}^\star_{\text{mot}} \text{TC}(X)$ of Definition~\ref{definitionmotivicfiltrationonTC}. More precisely, for every prime number $p$ the $p$-completion of Bhatt--Lurie's filtration $\text{Fil}^\star_{\text{mot}} \text{TP}(X)$ is the filtration $\text{Fil}^\star_{\text{BMS}} \text{TP}(X;\Z_p)$ of Definition~\ref{definitionBMSfiltrationonTPpcompletedofqcqsderivedschemes}, and there is a natural fibre sequence
		$$\text{Fil}^\star_{\text{mot}} \text{TC}(X) \longrightarrow \text{Fil}^\star_{\text{mot}} \text{TC}^-(X) \longrightarrow \prod_{p \in \mathbb{P}} \text{Fil}^\star_{\text{BMS}} \text{TP}(X;\Z_p)$$
		of filtered spectra. Note here that the compatibility between the definition of $\text{Fil}^\star_{\text{mot}} \text{TC}(X)$ (Definition~\ref{definitionmotivicfiltrationonTC}) and Bhatt--Lurie's definition of $\text{Fil}^\star_{\text{mot}} \text{TC}^-(X)$ and $\text{Fil}^\star_{\text{mot}} \text{TP}(X)$ (\cite[Examples~6.4.19 and~6.4.20]{bhatt_absolute_2022}) is a consequence of their description of these filtrations as a pullback \cite[Construction~6.4.2]{bhatt_absolute_2022}. In particular, for every integer $i \in \Z$, this induces a natural fibre sequence
		$$\Z(i)^{\text{TC}}(X) \longrightarrow \mathcal{N}^{\geq i} \widehat{\Prism}^{\text{gl}}_X\{i\} \longrightarrow \prod_{p \in \mathbb{P}} \widehat{\Prism}_{X,p}\{i\}$$
		in the derived category $\mathcal{D}(\Z)$, where $\widehat{\Prism}_{X,p}$ denotes the $p$-adic Nygaard-completed absolute prismatic cohomology of $X$. 
	\end{remark}
	
	\begin{proposition}\label{propositionpcompletionofmotiviconTCisBMS}
		Let $X$ be a qcqs derived scheme, and $p$ be a prime number. Then the natural map
		$$\emph{Fil}^\star_{\emph{mot}} \emph{TC}(X;\Z_p) \longrightarrow \emph{Fil}^\star_{\emph{BMS}} \emph{TC}(X;\Z_p)$$
		is an equivalence of filtered $\mathbb{E}_\infty$-rings.
	\end{proposition}
	
	\begin{proof}
		By definition, the natural map
		$$\text{Fil}^\star_{\text{HKR}} \text{HC}^-(X) \longrightarrow \prod_{\ell \in \mathbb{P}} \text{Fil}^\star_{\text{HKR}}\text{HC}^-(X;\Z_{\ell})$$
		in Definition~\ref{definitionmotivicfiltrationonTC} is profinite completion, so its fibre becomes zero after $p$-completion.
	\end{proof}
	
	\begin{corollary}
		Let $X$ be a qcqs derived scheme, and $p$ be a prime number. Then the natural map
		$$\Z_p(i)^{\emph{TC}}(X) \longrightarrow \Z_p(i)^{\emph{BMS}}(X)$$
		is an equivalence in the derived category $\mathcal{D}(\Z_p)$.
	\end{corollary}
	
	\begin{proof}
		This is a direct consequence of Proposition~\ref{propositionpcompletionofmotiviconTCisBMS}.
	\end{proof}

	\begin{proposition}\label{propositionfilteredTCandrationalfilteredTCarecomplete}
		Let $X$ be a qcqs derived scheme. Then the filtrations $$\emph{Fil}^\star_{\emph{mot}} \emph{TC}(X) \text{ and } \emph{Fil}^\star_{\emph{mot}} \emph{TC}(X;\Q)$$ are complete.
	\end{proposition}
	
	\begin{proof}
		The filtrations $\text{Fil}^\star_{\text{HKR}} \text{HC}^-(X)$, $\prod_{p \in \mathbb{P}} \text{Fil}^\star_{\text{HKR}} \text{HC}^-(X;\Z_p)$, and $\prod_{p \in \mathbb{P}} \text{Fil}^\star_{\text{BMS}} \text{TC}(X;\Z_p)$ are complete by Lemmas~\ref{lemmaHKRfiltrationonHC-isalwayscomplete}, \ref{lemmaHKRfiltrationproductpcompletionsiscomplete}, and \ref{lemmaBMSfiltrationproductallprimesiscomplete} respectively. By Definition~\ref{definitionmotivicfiltrationonTC}, the filtration $\text{Fil}^\star_{\text{mot}}\text{TC}(X)$ is then complete, as a pullback of three complete filtrations. To prove that the filtration $\text{Fil}^\star_{\text{mot}}\text{TC}(X;\Q)$ is complete, consider the cartesian square of filtered spectra
		$$\begin{tikzcd}
			\text{Fil}^\star_{\text{mot}} \text{TC}(X;\Q) \ar[r] \ar[d] & \text{Fil}^\star_{\text{HKR}} \text{HC}^-(X;\Q) \ar[d] \\
			\Big(\prod_{p \in \mathbb{P}} \text{Fil}^\star_{\text{BMS}} \text{TC}(X;\Z_p)\Big)_{\Q} \ar[r] & \Big(\prod_{p \in \mathbb{P}} \text{Fil}^\star_{\text{HKR}} \text{HC}^-(X;\Z_p)\Big)_{\Q}
		\end{tikzcd}$$
		induced by taking the rationalisation of Definition~\ref{definitionmotivicfiltrationonTC}. The filtration $$\big(\prod_{p \in \mathbb{P}} \text{Fil}^\star_{\text{BMS}} \text{TC}(X;\Z_p)\big)_{\Q}$$ is complete by Lemma~\ref{lemmaBMSfiltrationproductallprimesiscomplete}. The fibre of the natural map 
		$$\text{Fil}^\star_{\text{HKR}} \text{HC}^-(X) \longrightarrow \prod_{p \in \mathbb{P}} \text{Fil}^\star_{\text{HKR}} \text{HC}^-(X;\Z_p)$$
		is complete as an object of the filtered derived category $\mathcal{DF}(\Z)$ (Lemmas~\ref{lemmaHKRfiltrationonHC-isalwayscomplete} and \ref{lemmaHKRfiltrationHC-solidproductallprimesiscomplete}), and is zero modulo $p$ for every prime number $p$ by construction. In particular, it is naturally identified with the fibre of the natural map
		$$\text{Fil}^\star_{\text{HKR}} \text{HC}^-(X;\Q) \longrightarrow \Big(\prod_{p \in \mathbb{P}} \text{Fil}^\star_{\text{HKR}} \text{HC}^-(X;\Z_p)\Big)_{\Q},$$
		which is thus complete as an object of the filtered derived category $\mathcal{DF}(\Q)$. This implies, by the previous cartesian square, that the filtration $\text{Fil}^\star_{\text{mot}} \text{TC}(X;\Q)$ is also complete. 
	\end{proof}
	
	\begin{remark}[Exhaustivity of the motivic filtration on TC]\label{remarkexhaustivitymotivicfiltrationonTC}
		The motivic filtration $\text{Fil}^\star_{\text{mot}} \text{TC}$ is not exhaustive on general qcqs derived scheme. Although this will not be necessary to prove that the motivic filtration on algebraic $K$-theory is exhaustive (Theorem~\ref{propositionmotivicfiltrationisexhaustive}), one can prove, using \cite[Lemma~$4.10$]{antieau_periodic_2019} and its proof, that if $X$ is a quasi-lci $\Z$-scheme,\footnote{By this, we mean that Zariski-locally on the qcqs scheme $X$, the cotangent complex $\mathbb{L}_{-/\Z}$ has Tor-amplitude in~$[-1;0]$.} then the motivic filtration $\text{Fil}^\star_{\text{mot}} \text{TC}(X)$ is exhaustive.
	\end{remark}

	\begin{proposition}\label{propositionétaledescentfiltrationonTC}
		For every integer $i \in \Z$, the presheaf
		$$\emph{Fil}^i_{\emph{mot}} \emph{TC}(-) : \emph{dSch}^{\emph{qcqs,op}} \longrightarrow \emph{Sp}$$
		is an étale sheaf.
	\end{proposition}

	\begin{proof}
		By Definition~\ref{definitionmotivicfiltrationonTC}, it suffices to prove that the presheaves $$\prod_{p \in \mathbb{P}} \text{Fil}^i_{\text{BMS}} \text{TC}(-;\Z_p), \text{ } \text{Fil}^i_{\text{HKR}} \text{HC}^-(-), \text{ and } \prod_{p \in \mathbb{P}} \text{Fil}^i_{\text{HKR}} \text{HC}^-(-;\Z_p)$$ are étale sheaves. A product of sheaves is a sheaf, and these BMS and HKR filtrations are complete (Lemmas~\ref{lemmaBMSfiltrationproductallprimesiscomplete}, \ref{lemmaHKRfiltrationonHC-isalwayscomplete}, and~\ref{lemmaHKRfiltrationproductpcompletionsiscomplete}). It then suffices to prove that for every prime number $p$, the presheaves $$\Z_p(i)^{\text{BMS}}(-), \text{ } R\Gamma_{\text{Zar}}\big(-,\widehat{\mathbb{L}\Omega}^{\geq i}_{-/\Z}\big), \text{ and } R\Gamma_{\text{Zar}}\big(-,(\widehat{\mathbb{L}\Omega}^{\geq i}_{-/\Z})^\wedge_p\big)$$ are étale sheaves. The statement for $\Z_p(i)^{\text{BMS}}(-)$ is a consequence of $p$-fpqc descent (\cite[Proposition~$7.4.7$]{bhatt_absolute_2022}). The statement for the other two presheaves reduces to the fpqc descent for the powers of the cotangent complex (\cite[Theorem~$3.1$]{bhatt_topological_2019}).
	\end{proof}

	\begin{corollary}\label{corollaryétaledescentforZ(i)TC}
		For every integer $i \in \Z$, the presheaf
		$$\Z(i)^{\emph{TC}}(-) : \emph{dSch}^{\emph{qcqs,op}} \longrightarrow \mathcal{D}(\Z)$$
		is an étale sheaf.
	\end{corollary}
	
	\begin{proof}
		This is a direct consequence of Proposition~\ref{propositionétaledescentfiltrationonTC}.
	\end{proof}

	\subsection{The motivic filtration on \texorpdfstring{$L_{\text{cdh}} \text{TC}$}{TEXT}}
	
	\vspace{-\parindent}
	\hspace{\parindent}
	
	In this subsection, we introduce the motivic filtration on the cdh sheafification of topological cyclic homology of general qcqs schemes (Definition~\ref{definitionmotivicfiltrationonLcdhTC}).
	
	\begin{definition}[Motivic filtration on $L_{\text{cdh}} \text{TC}$]\label{definitionmotivicfiltrationonLcdhTC}
		The {\it motivic filtration on cdh sheafified topological cyclic homology} of qcqs schemes
		$$\text{Fil}^\star_{\text{mot}} L_{\text{cdh}} \text{TC}(-) : \text{Sch}^{\text{qcqs},\text{op}} \longrightarrow \text{CAlg}(\text{FilSp})$$
		is the functor defined as the cdh sheafification of the motivic filtration on topological cyclic homology (Definition~\ref{definitionmotivicfiltrationonTC})
		$$\text{Fil}^\star_{\text{mot}} L_{\text{cdh}} \text{TC}(-) := \big(L_{\text{cdh}} \text{Fil}^\star_{\text{mot}} \text{TC}\big)(-).$$    
	\end{definition}
	
	\begin{remark}[Graded pieces of $\text{Fil}^\star_{\text{mot}} L_{\text{cdh}} \text{TC}$]
		Let $X$ be a qcqs scheme. For every integer $i \in \Z$, the canonical map
		$$\big(L_{\text{cdh}} \Z(i)^{\text{TC}}\big)(X) \longrightarrow \text{gr}^i_{\text{mot}} L_{\text{cdh}} \text{TC}(R)[-2i]$$
		is an equivalence in the derived category $\mathcal{D}(\Z)$. We will usually refer to these shifted graded pieces by the complexes $\big(L_{\text{cdh}} \Z(i)^{\text{TC}}\big)(X)$.
	\end{remark} 
	
	\begin{remark}[Completeness of $\text{Fil}^\star_{\text{mot}} L_{\text{cdh}} \text{TC}$]
		It is not clear {\it a priori} that the filtered spectrum $\text{Fil}^\star_{\text{mot}} L_{\text{cdh}} \text{TC}(X)$ is complete, even for qcqs schemes of finite valuative dimension. Modulo any prime number $p$, this is a consequence of the connectivity bound \cite[Theorem~$5.1\,(1)$]{antieau_beilinson_2020} and \cite[Theorem~$2.4.15$]{elmanto_cdh_2021}. The integral statement will be a consequence of certain cdh descent results in Section~\ref{sectionratonialstructure}.
	\end{remark}
	
	\begin{remark}[Exhaustivity of $\text{Fil}^\star_{\text{mot}} L_{\text{cdh}} \text{TC}$]
		The filtration $\text{Fil}^\star_{\text{mot}} L_{\text{cdh}} \text{TC}$ is not exhaustive on general qcqs derived schemes. We will prove however, in Section~\ref{sectionratonialstructure}, that the fibre of the natural map $\text{Fil}^\star_{\text{mot}} \text{TC} \rightarrow \text{Fil}^\star_{\text{mot}} L_{\text{cdh}} \text{TC}$ is $\N$-indexed, and in particular exhaustive.
	\end{remark}

	\section{Definition of motivic cohomology}
	
	\vspace{-\parindent}
	\hspace{\parindent}
	
	In this section, we introduce motivic cohomology of general quasi-compact quasi-separated derived schemes (Definition~\ref{definitionmotiviccohomologyofderivedschemes}) and establish some of the fundamental properties of the associated motivic filtration.

    \medskip
	
	\subsection{Classical motivic cohomology}
	
	\vspace{-\parindent}
	\hspace{\parindent}
	
	In this subsection, we review the classical definition of motivic cohomology of smooth schemes in mixed characteristic. Following \cite{bloch_algebraic_1986,levine_techniques_2001,geisser_motivic_2004}, the motivic cohomology of smooth $\Z$-schemes $X$ is classically defined in terms of Bloch's cycle complexes $z^i(X,\bullet)$. Recall that Bloch's cycle complex is a simplicial abelian group defined in terms of algebraic cycles. The homotopy groups of Bloch's cycle complexes, called Bloch's higher Chow groups, are a generalisation of Chow groups that are designed to generalise the relation between the $\text{K}_0$ and the Chow groups of a quasi-projective variety to higher $K$-groups. Via the Dold--Kan correspondence, we view Bloch's cycle complexes as objects of the derived category $\mathcal{D}(\Z)$.

    In the following definition, we use the slice filtration in stable homotopy theory, as introduced by Voevodsky \cite{voevodsky_open_2022,voevodsky_possible_2002,bachmann_norms_2021}.
	
	\begin{definition}[Motivic filtration on $K$-theory of smooth schemes]\label{definitionclassicalmotivicfiltration}
		Let $B$ be a field or a mixed characteristic Dedekind domain. The {\it classical motivic filtration} on algebraic $K$-theory of smooth $B$-schemes is the functor
		$$\text{Fil}^\star_{\text{cla}} \text{K}(-) : \text{Sm}_{B}^{\text{op}} \longrightarrow \text{CAlg}(\text{FilSp})$$
		defined as the image, via the mapping spectrum construction $\omega^{\infty} : \text{SH}(B) \rightarrow \text{PSh}(\text{Sm}_{B}, \text{Sp})$, of the slice filtration $s^\star \text{KGL}_B$ on the $K$-theory motivic spectrum $\text{KGL}_B \in \text{SH}(B)$.
	\end{definition}

	\begin{definition}[Classical motivic cohomology of smooth schemes]\label{definitionclassicalmotiviccohomology}
		Let $B$ be a field or a mixed characteristic Dedekind domain ({\it e.g.}, $B=\Z$), and $X$ be a smooth $B$-scheme. For any integer $i \in \Z$, the {\it classical motivic complex}\footnote{Every value of the functor $\text{gr}^i_{\text{cla}} \text{K}(-)[-2i]$ has a natural module structure over the $\mathbb{E}_{\infty}$-ring $\text{gr}^0_{\text{cla}} \text{K}(\Z)$, which, by \cite[Proposition~$6.1$]{spitzweck_commutative_2018} and \cite{bachmann_very_2022}, is naturally identified with the $\mathbb{E}_{\infty}$-ring $\text{H}\!\Z$. This implies that the spectra-valued functor $\text{gr}^i_{\text{cla}}\text{K}(-)[-2i]$ takes values in $\text{H}\!\Z$-linear spectra, {\it i.e.}, in the derived category $\mathcal{D}(\Z)$.}
        $$\Z(i)^{\text{cla}} : \text{Sm}_B^{\text{op}} \longrightarrow \mathcal{D}(\Z)$$
        is the shifted graded piece of the classical motivic filtration of Definition~\ref{definitionclassicalmotivicfiltration}:
        $$\Z(i)^{\text{cla}}(-) := \text{gr}^i_{\text{cla}} \text{K}(-)[-2i].$$
        By results of Voevodsky \cite{voevodsky} and Levine \cite{levine_homotopy_2008} when $B$ is a field, and of Spitzweck \cite{spitzweck_commutative_2018} and Bachmann \cite{bachmann_very_2022} in mixed characteristic, for every smooth $B$-scheme $X$, there is an equivalence
        $$\Z(i)^{\text{cla}}(X) \simeq \big(L_{\text{Nis}}\, z^i(-,\bullet)\big)(X)[-2i]$$
        in the derived category $\mathcal{D}(\Z)$, where $\bullet$ is the cohomological index.
	\end{definition}

    Note that, by this identification with Bloch's cycle complex, the classical motivic complexes $\Z(i)^{\text{cla}}$ vanish in degrees more than $2i$, and in all degrees for weights $i<0$. 
	\begin{remark}\label{remarkclassicalBlochrelatedtoslice}
		The pullback of algebraic cycles being well-defined only along flat maps, it is not straightforward to prove that Bloch's cycle complexes are functorial. In Definition~\ref{definitionclassicalmotiviccohomology}, we follow Voevodsky's point of view that ($\mathbb{A}^1$-invariant) motivic cohomology should rather be defined in terms of the zeroth slice of the $K$-theory motivic spectrum $\text{KGL}$ in $\text{SH}$. Via the pointwise identification of this definition with Bloch's cycle complexes, this implies that Bloch's cycle complexes are a posteriori functorial and multiplicative, when seen as objects of the derived category $\mathcal{D}(\Z)$.
	\end{remark}
	
	\begin{example}[Weight zero classical motivic cohomology]\label{exampleweightzeroclassicalmotiviccohomology}
        Let $B$ be a field or a mixed characteristic Dedekind domain ({\it e.g.}, $B=\Z$). For every regular noetherian scheme $X$, there is a natural comparison map
        $$R\Gamma_{\text{Nis}}(-,\Z) \simeq s^0_u\, \omega^\infty \text{KGL}_X \longrightarrow \omega^\infty s^0 \text{KGL}_X$$
        of Nisnevich sheaves of $\mathbb{E}_\infty$-rings on smooth $X$-schemes (\cite[Construction~4.35]{bachmann_A^1-invariant_2025}), where $s^\star_u$ is the unstable slice filtration of \cite[Section~4.2.1]{bachmann_A^1-invariant_2025}, induced by the rank map on $K$-theory. If $X$ is moreover smooth over $B$, then this comparison map is an equivalence (\cite[Corollary~6.26]{bachmann_A^1-invariant_2025}), thus inducing a natural equivalence
        $$\Z(0)^{\text{cla}}(X) \simeq R\Gamma_{\text{Nis}}(X,\Z)$$
        of $\mathbb{E}_\infty$-$\Z$-algebras.
	\end{example}
	
	\begin{example}[Weight one classical motivic cohomology]\label{exampleweight1classicalmotiviccohomology}
        For every regular noetherian scheme $X$, there is similarly a natural comparison map
        $$R\Gamma_{\text{Nis}}(-,\mathbb{G}_m)[1] \simeq s^1_u \omega^\infty \text{KGL}_X \longrightarrow \omega^\infty s^1 \text{KGL}_X$$
        of presheaves of spectra on smooth $X$-schemes (\cite[Construction~4.35]{bachmann_A^1-invariant_2025}), induced by the determinant map on $K$-theory. If $X$ is smooth over a field or a mixed characteristic Dedekind domain, then this comparison map is an equivalence (\cite[Corollary~6.26]{bachmann_A^1-invariant_2025}), thus inducing a natural equivalence
        $$\Z(1)^{\text{cla}}(X) \simeq R\Gamma_{\text{Nis}}(X,\mathbb{G}_m)[-1]$$
        in the derived category $\mathcal{D}(\Z)$. In particular, the complex $\Z(1)^{\text{cla}}(X)$ is concentrated in degrees one and two, where it is given by
		$$\text{H}^1(\Z(1)^{\text{cla}}(X)) \cong \mathcal{O}(X)^{\times} \quad \text{ and } \quad \text{H}^2(\Z(1)^{\text{cla}}(X)) \cong \text{Pic}(X).$$
	\end{example}

    \medskip
	
	\subsection{Bachmann--Elmanto--Morrow's cdh-local motivic filtration}\label{subsectioncdhlocalmotivicfiltration}
	
	\vspace{-\parindent}
	\hspace{\parindent}
	
	The constructions and results of this subsection are due to Bachmann--Elmanto--Morrow. More precisely, following \cite[Section~7.2]{bachmann_A^1-invariant_2025}, we review the cdh-local motivic filtration on homotopy $K$\nobreakdash-theory $\text{KH}(-)$ of qcqs schemes (Definition~\ref{definitionmotivicfiltrationonKHtheoryofschemes}), whose shifted graded pieces $\Z(i)^{\text{cdh}}$ will serve as a building block for the definition of the motivic complexes $\Z(i)^{\text{mot}}$ (Remark~\ref{remarkmaincartesiansquareformotiviccohomology}). We first define the lisse motivic complexes $\Z(i)^{\text{lisse}}$ as an intermediate construction, and as a practical tool for later sections.
	
	The following definition is motivated by the observation of Bhatt--Lurie that connective algebraic $K$-theory $\text{K}^{\text{conn}}(-)$ is left Kan extended on animated commutative rings from smooth $\Z$-algebras \cite[Proposition~A$.0.4$]{elmanto_modules_2020}.
	
	\begin{definition}[Motivic filtration on connective $K$-theory of animated rings]\label{definitionlissemotivicfiltrationconnectiveKtheory}
		The {\it motivic filtration on connective algebraic $K$-theory} of animated commutative rings is the functor
		$$\text{Fil}^\star_{\text{lisse}} \text{K}^{\text{conn}}(-) : \text{AniRings} \longrightarrow \text{CAlg}(\text{FilSp})$$
		defined as the left Kan extension of the classical motivic filtration on algebraic $K$-theory of smooth $\Z$-algebras
		$$\text{Fil}^\star_{\text{lisse}} \text{K}^{\text{conn}}(-) := \big(L_{\text{AniRings}/\text{Sm}_{\Z}} \text{Fil}^\star_{\text{cla}} \text{K}\big)(-).$$
	\end{definition}

	Note that connective algebraic $K$-theory is not a Zariski sheaf on commutative rings. Most of our results on the following lisse motivic complexes $\Z(i)^{\text{lisse}}$ will be formulated in the generality of local rings.
	
	\begin{definition}[Lisse motivic cohomology of animated rings]\label{definitionlissemotiviccohomology}
		For any integer $i \in \Z$, the {\it lisse motivic complex}
		$$\Z(i)^{\text{lisse}}(-) : \text{AniRings} \longrightarrow \mathcal{D}(\Z)$$
		is the shifted graded piece of the motivic filtration of Definition~\ref{definitionlissemotivicfiltrationconnectiveKtheory}:
		$$\Z(i)^{\text{lisse}}(-) := \text{gr}^i_{\text{lisse}} \text{K}^{\text{conn}}(-) [-2i].$$
	\end{definition}
	
	Note that the lisse motivic complexes $\Z(i)^{\text{lisse}}$ are the left Kan extension of the classical motivic complexes $\Z(i)^{\text{cla}}$, and in particular vanish in weights $i<0$.
	
	\begin{example}[Weight zero lisse motivic cohomology]\label{exampleweightzerolissemotiviccohomology}
		For every local ring $R$, there is a natural equivalence
		$$\Z(0)^{\text{lisse}}(R) \simeq \Z[0]$$
		of $\mathbb{E}_\infty$-$\Z$-algebras. 
        This is a consequence of Example~\ref{exampleweightzeroclassicalmotiviccohomology}, by using that the functor $\Z(0)^{\text{lisse}}(-)$ is left Kan extended on local rings from its restriction to local essentially smooth $\Z$\nobreakdash-algebras.
	\end{example}
	
	\begin{example}[Weight one lisse motivic cohomology]\label{exampleweight1lissemotiviccohomology}
		For every commutative ring $R$, there is a natural equivalence
		$$\Z(1)^{\text{lisse}}(R) \simeq \big(\tau^{\leq 1} R\Gamma_{\text{Zar}}(R,\mathbb{G}_m)\big)[-1]$$
		in the derived category $\mathcal{D}(\Z)$. In particular, the complex $\Z(1)^{\text{lisse}}(R) \in \mathcal{D}(\Z)$ is concentrated in degrees one and two, where it is given by
		$$\text{H}^1(\Z(1)^{\text{lisse}}(R)) \cong \mathcal{O}(R)^{\times} \quad \text{ and } \quad \text{H}^2(\Z(1)^{\text{lisse}}(R)) \cong \text{Pic}(R).$$
		This is a consequence of Example~\ref{exampleweight1classicalmotiviccohomology}, by using that the left Kan extension of a functor taking values in degrees at most two takes values in degrees at most two, and that the functor $\tau^{\leq 1} R\Gamma_{\text{Zar}}(-,\mathbb{G}_m)$ on commutative rings is left Kan extended from its restriction to smooth $\Z$-algebras. Here the latter left Kan extension property is a consequence of the same left Kan extension property for the functors $\mathbb{G}_m(-)$ (which is a special case of Mathew's criterion \cite[Proposition~A$.0.4$]{elmanto_modules_2020}) and $\text{Pic}(-)$ (which is a consequence of rigidity, see \cite[Lemma~$7.6$]{elmanto_motivic_2023}).
	\end{example}
	
	\begin{definition}[Cdh-local motivic filtration on $KH$-theory of schemes]\label{definitionmotivicfiltrationonKHtheoryofschemes}
		The {\it cdh-local motivic filtration on homotopy $K$-theory} of qcqs schemes is the functor
		$$\text{Fil}^\star_{\text{cdh}} \text{KH}(-) : \text{Sch}^{\text{qcqs,op}} \longrightarrow \text{CAlg}(\text{FilSp})$$
		defined as
		$$\text{Fil}^\star_{\text{cdh}} \text{KH}(-) := \big(L_{\text{cdh}} \text{Fil}^\star_{\text{lisse}} \text{K}^{\text{conn}}\big)(-) = \big(L_{\text{cdh}} L_{\text{Sch}^{\text{qcqs,op}}/\text{Sm}_{\Z}^{\text{op}}} \text{Fil}^\star_{\text{cla}} \text{K}\big)(-).$$
	\end{definition}
	
	\begin{remark}\label{remarkclassicalcdhlocalmotivicfiltrationcomparisonmap}
		By construction of the cdh-local motivic filtration (Definition~\ref{definitionmotivicfiltrationonKHtheoryofschemes}), there is a natural multiplicative comparison map of presheaves
		$$\text{Fil}^\star_{\text{cla}} \text{K}(-) \longrightarrow \text{Fil}^\star_{\text{cdh}} \text{KH}(-)$$
		on smooth $\Z$-schemes.
	\end{remark}
	
	\begin{definition}[Cdh-local motivic cohomology of schemes]
		For any integer $i \in \Z$, the {\it cdh-local motivic complex}
		$$\Z(i)^{\text{cdh}}(-) : \text{Sch}^{\text{qcqs,op}} \longrightarrow \mathcal{D}(\Z)$$
		is the shifted graded piece of the motivic filtration of Definition~\ref{definitionmotivicfiltrationonKHtheoryofschemes}:
		$$\Z(i)^{\text{cdh}}(-) := \text{gr}^i_{\text{cdh}} \text{KH}(-)[-2i].$$
	\end{definition}
	
    We will refer throughout the text to the following properties of these cdh-local motivic complexes.

    \begin{theorem}[\cite{bachmann_A^1-invariant_2025}]\label{theoremBEM}
        Let $X$ be a qcqs scheme, and $i \geq 0$ be an integer.
        \begin{enumerate}
            \item The filtration $\emph{Fil}^\star_{\emph{cdh}} \emph{KH}(X)$ is multiplicative and $\N$-indexed. Moreover, if $X$ is of valuative dimension at most $d$, then the spectrum $\emph{Fil}^i_{\emph{cdh}} \emph{KH}(X)$ is in cohomological degrees at most $-i+d$; in particular, the filtration $\emph{Fil}^\star_{\emph{cdh}} \emph{KH}(X)$, and its rationalisation $\emph{Fil}^\star_{\emph{cdh}} \emph{KH}(X;\Q)$, are complete.
            \item For every integer $m \in \Z \setminus \{0\}$, there exists a natural endomorphism $\psi^m$ of the filtered $\mathbb{E}_\infty$-ring $\emph{Fil}^\star_{\emph{cdh}} \emph{KH}(X)[\tfrac{1}{m}]$ such that the induced endomorphism of the graded $\mathbb{E}_\infty$-$\Z$-algebra $\Z[\tfrac{1}{m}](\star)^{\emph{cdh}}(X)[2\star]$ is naturally homotopic to its endomorphism multiplication-by-$m^\star$.
            \item For every prime number $p$ and every integer $k \geq 1$, there is a natural equivalence
            $$\Z/p^k(i)^{\emph{cdh}}(X) \simeq \big(L_{\emph{cdh}} \Z/p^k(i)^{\emph{syn}}\big)(X)$$
            in the derived category $\mathcal{D}(\Z/p^k)$.
        \end{enumerate}
    \end{theorem}

    \begin{proof}
        (1) is a part of \cite[Theorem~7.12]{bachmann_A^1-invariant_2025}, (2) follows from \cite[Theorem~4.48 and Remark~4.55]{bachmann_A^1-invariant_2025}, and (3) is \cite[Theorem~7.16]{bachmann_A^1-invariant_2025}.
    \end{proof}

    \begin{remark}
        With the exception of Section~\ref{sectionA1invariantmotiviccohomology} on the comparison with Bachmann--Elmanto--Morrow's $\mathbb{A}^1$-motivic complexes $\Z(i)^{\mathbb{A}^1}$, where we use that $\Z(i)^{\mathbb{A}^1}$ is the $\mathbb{A}^1$-localisation of the cdh-local motivic complex $\Z(i)^{\text{cdh}}$, we do not use any of the main results established in \cite{bachmann_A^1-invariant_2025}. In particular, we do not use the $\mathbb{A}^1$-invariance or the projective bundle formula for the cdh-local motivic complexes $\Z(i)^{\text{cdh}}$, nor Voevodsky's slice conjectures.
    \end{remark}

    \medskip

	\subsection{Definition of motivic cohomology}\label{subsectiondefinitionmotiviccohomology}
	
	\vspace{-\parindent}
	\hspace{\parindent}
	
	In this subsection, we introduce the motivic filtration on algebraic $K$-theory of qcqs derived schemes (Definitions~\ref{definitionmotivicfiltrationonKtheoryofschemes} and \ref{definitionmotivicfiltrationonKtheoryofderivedschemes}), by constructing a filtered extension of the cartesian square of Theorem~\ref{theoremKST+LT}. To do so, we first define a filtered cdh-local cyclotomic trace map $\text{Fil}^\star_{\text{cdh}} \text{KH}(-) \rightarrow \text{Fil}^\star_{\text{mot}} L_{\text{cdh}} \text{TC}(-)$ (Construction~\ref{constructionfilteredcdhlocalcyclotomictrace}).
	
	\begin{theorem}[Filtered cyclotomic trace in the smooth case]\label{theoremfilteredcyclotomictraceinthesmoothcase}
		The cyclotomic trace map $$\emph{K}(-) \longrightarrow \emph{TC}(-),$$ viewed as a map of spectra-valued presheaves on the category of smooth $\Z$-schemes, admits a unique, multiplicative extension to a map $$\emph{Fil}^\star_{\emph{cla}} \emph{K}(-) \longrightarrow \emph{Fil}^\star_{\emph{mot}} \emph{TC}(-)$$
		of filtered presheaves of spectra.
	\end{theorem}
	
	\begin{proof}
		By Definition~\ref{definitionmotivicfiltrationonTC}, the natural cartesian square
		$$\begin{tikzcd}
			\text{TC}(-) \ar[r] \ar[d] & \text{HC}^-(-) \ar[d] \\
			\prod_{p \in \mathbb{P}} \text{TC}(-;\Z_p) \ar[r] & \prod_{p \in \mathbb{P}} \text{HC}^-(-;\Z_p)
		\end{tikzcd}$$
		admits a multiplicative filtered extension
		$$\begin{tikzcd}
			\text{Fil}^\star_{\text{mot}} \text{TC}(-) \ar[r] \ar[d] & \text{Fil}^\star_{\text{HKR}} \text{HC}^-(-) \ar[d] \\
			\prod_{p \in \mathbb{P}} \text{Fil}^\star_{\text{BMS}} \text{TC}(-;\Z_p) \ar[r] & \prod_{p \in \mathbb{P}} \text{Fil}^\star_{\text{HKR}} \text{HC}^-(-;\Z_p)
		\end{tikzcd}$$
		on qcqs derived schemes. Let $p$ be a prime number. It then suffices to prove that the natural maps $$\text{K}(-) \longrightarrow \text{HC}^-(-), \text{ } \text{K}(-) \longrightarrow \text{HC}^-(-;\Z_p), \text{ and }\text{K}(-) \longrightarrow \text{TC}(-;\Z_p),$$ viewed as maps of spectra-valued presheaves on the category of smooth $\Z$-schemes, admit unique multiplicative filtered extensions to maps of filtered presheaves of spectra $$\text{Fil}^\star_{\text{cla}} \text{K}(-) \longrightarrow \text{Fil}^\star_{\text{HKR}} \text{HC}^-(-), \quad \text{Fil}^\star_{\text{cla}} \text{K}(-) \longrightarrow \text{Fil}^\star_{\text{HKR}} \text{HC}^-(-;\Z_p), \text{ and }$$ $$ \text{Fil}^\star_{\text{cla}} \text{K}(-) \longrightarrow \text{Fil}^\star_{\text{BMS}} \text{TC}(-;\Z_p)$$ respectively. 
		
		The proof of \cite[Proposition~$4.6$]{elmanto_motivic_2023}, which is stated over a quasismooth map of commutative rings $k_0 \rightarrow k$ such that $k$ is a field, applies readily to the case where $k_0=k=\Z$. More precisely, we use in this proof the Gersten conjecture for classical motivic cohomology on smooth $\Z$-schemes, which is \cite[Theorem~$1.1$]{geisser_motivic_2004}, in order to ensure the vanishing of motivic cohomology in high degrees by reduction to the case of discrete valuation rings. In particular, the natural map $\text{K}(-) \rightarrow \text{HC}^-(-)$, viewed as a map of spectra-valued presheaves on the category of smooth $\Z$-schemes, admits a unique multiplicative filtered extension to a map of filtered presheaves of spectra $\text{Fil}^\star_{\text{cla}} \text{K}(-) \rightarrow \text{Fil}^\star_{\text{HKR}} \text{HC}^-(-)$. 
		
		Similarly, the $p$-completed cotangent complex $(\mathbb{L}_{R/\Z})^\wedge_p$ of a smooth $\Z$\nobreakdash-alge\-bra $R$ is concentrated in degree zero, given by the $p$-flat $R$-module $(\Omega^1_{R/\Z})^\wedge_p$. So for every integer $i \in \Z$, this implies that the object $\text{gr}^i_{\text{HKR}} \text{HC}^-(R;\Z_p) \in \mathcal{D}(\Z)$ is in cohomological degrees less than~$-i$. The proof of \cite[Proposition~$4.6$]{elmanto_motivic_2023} then adapts readily to prove that the natural map $\text{K}(-) \rightarrow \text{HC}^-(-;\Z_p)$, viewed as a map of spectra-valued presheaves on the category of smooth $\Z$-schemes, admits a unique multiplicative extension to a map of filtered presheaves of spectra $\text{Fil}^\star_{\text{cla}} \text{K}(-) \rightarrow \text{Fil}^\star_{\text{HKR}} \text{HC}^-(-;\Z_p)$.
		
		Finally, the natural map $\text{K}(-) \rightarrow \text{TC}(-;\Z_p)$, viewed as a map of spectra-valued presheaves on the category of smooth $\Z$-schemes, admits a unique multiplicative extension to a map of filtered presheaves of spectra $\text{Fil}^\star_{\text{cla}} \text{K}(-) \rightarrow \text{Fil}^\star_{\text{BMS}} \text{TC}(-;\Z_p)$, by \cite[Proposition~$6.12$]{annala_atiyah_2024}.
	\end{proof}
    \begin{proposition}\label{remarkBMSfiltrationonTCcanbereconstructedfromclassicalmotivicfiltration}
        For every prime number $p$, the functor $\emph{Fil}^\star_{\emph{BMS}} \emph{TC}(-;\Z_p)$ is identified, on animated commutative rings, and via the comparison map
        $$\emph{Fil}^\star_{\emph{cla}} \emph{K}(-) \longrightarrow \emph{Fil}^\star_{\emph{BMS}} \emph{TC}(-;\Z_p)$$
        of Theorem~\ref{theoremfilteredcyclotomictraceinthesmoothcase}, with the $p$-completed left Kan extension of the étale sheafification of the functor $\emph{Fil}^\star_{\emph{cla}} \emph{K}(-)$ from smooth $\Z$-algebras to animated commutative rings.
    \end{proposition}

    \begin{proof}
        On smooth $\Z$-algebras, the $p$-completed étale sheafification of the classical motivic filtration $\text{Fil}^\star_{\text{cla}} \text{K}(-)$ corresponds to the motivic filtration on $p$\nobreakdash-adic Selmer $K$-theory $\text{Fil}^\star_{\text{mot}} \text{K}^{\text{Sel}}(-;\Z_p)$ (\cite[Remark~8.4.3]{bhatt_absolute_2022}), whose graded pieces are the syntomic complexes of Bhatt--Lurie (\cite[Construction~8.4.1]{bhatt_absolute_2022}). More precisely, it suffices to prove this statement at the level of graded pieces, which follows from the results of Geisser (\cite[Theorems~1.2\,(2) and 1.3]{geisser_motivic_2004}, stating that $p$-adic étale motivic cohomology agrees with Sato's $p$-adic étale Tate twists) and Bhatt--Mathew (\cite[Theorem~5.8 and Example~5.9]{bhatt_syntomic_2023}, stating that Sato's $p$-adic étale Tate twists agree with Bhatt--Lurie's syntomic complexes). 
        
        Moreover, the functor $\text{Fil}^\star_{\text{mot}}\text{K}^{\text{Sel}}(-;\Z_p)$, from animated commutative rings to $p$-complete filtered spectra, is left Kan extended from smooth $\Z$-algebras (\cite[Proposition~8.4.10]{bhatt_absolute_2022}), and is naturally identified with the functor $\text{Fil}^\star_{\text{BMS}}\text{TC}(-;\Z_p)$ on $p$-complete animated commutative rings (\cite[Remark~8.4.2]{bhatt_absolute_2022}). Finally, the functor $\text{Fil}^\star_{\text{BMS}} \text{TC}(-;\Z_p)$ on animated commutative rings depends only on the $p$-completion of its input (Remark~\ref{remarkderivedpcompletionwithBMS}), hence the desired result. 
    \end{proof}

    Note that the Proposition~\ref{remarkBMSfiltrationonTCcanbereconstructedfromclassicalmotivicfiltration} implies, via Theorem~\ref{theoremBMSfiltrationonTCpcompletedsatisfiesquasisyntomicdescentandisLKEfrompolynomialalgebras}, that the BMS filtration $\text{Fil}^\star_{\text{BMS}} \text{TC}(-;\Z_p)$ on qcqs derived schemes (Definition~\ref{definitionBMSfiltrationonTCpcompletedofqcqsderivedschemes}) can be reconstructed from the classical motivic filtration $\text{Fil}^\star_{\text{cla}} \text{K}(-)$ on smooth $\Z$\nobreakdash-schemes (Definition~\ref{definitionclassicalmotivicfiltration}). This will be used in Section~\ref{sectionratonialstructure} to construct Adams operations on the BMS filtration $\text{Fil}^\star_{\text{BMS}} \text{TC}(-;\Z_p)$.
    
	\begin{construction}[Filtered cdh-local cyclotomic trace]\label{constructionfilteredcdhlocalcyclotomictrace}
		The {\it filtered cdh-local cyclotomic trace map} is the multiplicative map
		$$\text{Fil}^\star_{\text{cdh}} \text{KH}(-) \longrightarrow \text{Fil}^\star_{\text{mot}} L_{\text{cdh}} \text{TC}(-)$$
		of functors from (the opposite category of) qcqs schemes to the category of multiplicative filtered spectra defined as the cdh sheafification of the composite
		$$\big(L_{\text{Sch}^{\text{qcqs,op}}/\text{Sm}_{\Z}^{\text{op}}} \text{Fil}^\star_{\text{cla}} \text{K}\big)(-) \longrightarrow \big(L_{\text{Sch}^{\text{qcqs,op}}/\text{Sm}_{\Z}^{\text{op}}} \text{Fil}^\star_{\text{mot}} \text{TC}\big)(-) \longrightarrow \text{Fil}^\star_{\text{mot}} \text{TC}(-),$$
		where the first map is the map induced by Theorem~\ref{theoremfilteredcyclotomictraceinthesmoothcase} and the second map is the canonical map. Note here that sheafification is a multiplicative operation, and that the compatibility between left Kan extension and multiplicative structures is ensured by \cite[Corollary~$3.2.3.2$]{lurie_higher_2017} (see also \cite[$2.3.2$]{elmanto_motivic_2023}).
	\end{construction}

	\needspace{5\baselineskip}
	
	\begin{definition}[Motivic filtration on $K$-theory of schemes]\label{definitionmotivicfiltrationonKtheoryofschemes}
		The {\it motivic filtration on non-connective algebraic $K$-theory} of qcqs schemes is the functor
		$$\text{Fil}^\star_{\text{mot}} \text{K}(-) : \text{Sch}^{\text{qcqs,op}}  \longrightarrow \text{CAlg}(\text{FilSp})$$
		defined by the cartesian square of functors of multiplicative filtered spectra
		$$\begin{tikzcd}
			\text{Fil}^\star_{\text{mot}} \text{K}(-) \ar[r] \ar[d] & \text{Fil}^\star_{\text{mot}} \text{TC}(-) \ar[d] \\
			\text{Fil}^\star_{\text{cdh}} \text{KH}(-) \ar[r] & \text{Fil}^\star_{\text{mot}} L_{\text{cdh}} \text{TC}(-), 
		\end{tikzcd}$$
		where the bottom horizontal map is the map of Construction \ref{constructionfilteredcdhlocalcyclotomictrace}, and the right vertical map is cdh sheafification.
	\end{definition}
	
	\begin{definition}[Motivic filtration on $K$-theory of derived schemes]\label{definitionmotivicfiltrationonKtheoryofderivedschemes}
		The {\it motivic filtration on non-connective algebraic $K$-theory} of qcqs derived schemes is the functor
		$$\text{Fil}^\star_{\text{mot}} \text{K}(-) : \text{dSch}^{\text{qcqs,op}} \longrightarrow \text{CAlg}(\text{FilSp})$$
		defined by the cartesian square of functors of multiplicative filtered spectra
		$$\begin{tikzcd}
			\text{Fil}^\star_{\text{mot}} \text{K}(-) \ar[r] \ar[d] & \text{Fil}^\star_{\text{mot}} \text{TC}(-) \ar[d] \\
			\text{Fil}^\star_{\text{mot}} \text{K}((-)^{\text{cl}}) \ar[r] & \text{Fil}^\star_{\text{mot}} \text{TC}((-)^{\text{cl}})
		\end{tikzcd}$$
		where $(-)^\text{cl} : \text{dSch} \rightarrow \text{Sch}$ is restriction to the classical locus, the filtration on $\text{K}((-)^{\text{cl}})$ is given by Definition~\ref{definitionmotivicfiltrationonKtheoryofschemes}, and the filtrations on $\text{TC}(-)$ and $\text{TC}((-)^{\text{cl}})$ are given by Definition~\ref{definitionmotivicfiltrationonTC}. Note that this definition is motivated by the Dundas--Goodwillie--McCarthy theorem (\cite[Theorem~7.0.0.2]{dundas_local_2013}), which states that the corresponding square without filtrations is indeed cartesian.
	\end{definition}
	
	\begin{definition}[Motivic cohomology of derived schemes]\label{definitionmotiviccohomologyofderivedschemes}
		For any integer $i \in \Z$, the {\it motivic complex}
		$$\Z(i)^{\text{mot}} : \text{dSch}^{\text{qcqs,op}} \longrightarrow \mathcal{D}(\Z)$$
		is the shifted graded piece of the motivic filtration of Definition~\ref{definitionmotivicfiltrationonKtheoryofderivedschemes}:
		$$\Z(i)^{\text{mot}}(-) := \text{gr}^i_{\text{mot}} \text{K}(-)[-2i].$$
		For every qcqs derived scheme $X$, also denote by
		$$\text{H}^n_{\text{mot}}(X,\Z(i)) := \text{H}^n(\Z(i)^{\text{mot}}(X)) \quad n\in \Z$$
		the {\it motivic cohomology groups} of $X$.
	\end{definition}
	
	\begin{remark}[Motivic cohomology of schemes]\label{remarkmaincartesiansquareformotiviccohomology}
		Let $X$ be a qcqs scheme, and $i \in \Z$ be an integer. By Definition \ref{definitionmotivicfiltrationonKtheoryofderivedschemes}, there is a natural cartesian square
		$$\begin{tikzcd}
			\Z(i)^{\text{mot}}(X) \arrow{r} \arrow{d} & \Z(i)^{\text{TC}}(X) \ar[d] \\
			\Z(i)^{\text{cdh}}(X) \arrow{r} & \big(L_{\text{cdh}} \Z(i)^{\text{TC}}\big)(X)
		\end{tikzcd}$$
		in the derived category $\mathcal{D}(\Z)$, where the bottom horizontal map is induced by Construction~\ref{constructionfilteredcdhlocalcyclotomictrace} and the right vertical map is cdh sheafification. This cartesian square can serve as a definition for the motivic cohomology of the scheme $X$.
	\end{remark}
	
	We now construct, for later use, a comparison map from classical motivic cohomology to the motivic cohomology of Definition~\ref{definitionmotiviccohomologyofderivedschemes}.
	
	\begin{definition}[Filtered classical-motivic comparison map]\label{definitionfilteredclassicalmotiviccomparisonmap}
		The {\it filtered classical-motivic comparison map} is the multiplicative map of presheaves
		$$\text{Fil}^\star_{\text{cla}} \text{K}(-) \longrightarrow \text{Fil}^\star_{\text{mot}} \text{K}(-)$$
		on smooth $\Z$-schemes induced by the maps Remark~\ref{remarkclassicalcdhlocalmotivicfiltrationcomparisonmap} and Theorem~\ref{theoremfilteredcyclotomictraceinthesmoothcase}. Note here that the compatibility between these two maps and Definition~\ref{definitionmotivicfiltrationonKtheoryofschemes} is automatic by Construction~\ref{constructionfilteredcdhlocalcyclotomictrace}.
	\end{definition}
	
	\begin{definition}[Classical-motivic comparison map]\label{definitionclassicalmotiviccomparisonmap}
		For any integer $i \in \Z$, the {\it classical-motivic comparison map} is the map of $\mathcal{D}(\Z)$-valued presheaves
		$$\Z(i)^{\text{cla}}(-) \longrightarrow \Z(i)^{\text{mot}}(-)$$
		on smooth $\Z$-schemes induced by taking the $i^{\text{th}}$ shifted graded piece of the filtered map of Definition~\ref{definitionfilteredclassicalmotiviccomparisonmap}. Taking the direct sum over integers $i \in \Z$, note that this map promotes, by construction, to a map of presheaves of graded $\E_\infty$-$\Z$-rings. 
	\end{definition}
	
	In the rest of this section, we discuss some of the first properties of the motivic filtration.

    \begin{remark}[Graded $\mathbb{E}_\infty$-$\Z$-algebra structure]
        Let $X$ be a qcqs derived scheme. The pullback in Definition~\ref{definitionmotivicfiltrationonKtheoryofderivedschemes} being taken in filtered $\mathbb{E}_\infty$-rings, the motivic cohomology $\Z(\star)^{\text{mot}}(X)$ (Definition~\ref{definitionmotiviccohomologyofderivedschemes}) admits, by construction, the structure of a graded $\mathbb{E}_\infty$-ring. It is in fact even a graded $\mathbb{E}_\infty$-$\Z$-algebra. Indeed, it suffices to prove the case where $X=\text{Spec}(\Z)$, where this is a consequence of Example~\ref{exampleweightzeroclassicalmotiviccohomology} and Definition~\ref{definitionclassicalmotiviccomparisonmap}.
    \end{remark}
	
	\begin{remark}[Comparison to cdh-local motivic cohomology]\label{remarkcomparisontocdhlocalmotiviccohomology}
		By construction (Definition~\ref{definitionmotivicfiltrationonKHtheoryofschemes}), the cdh-local motivic complex
		$$\Z(i)^{\text{cdh}} : \text{Sch}^{\text{qcqs,op}} \longrightarrow \mathcal{D}(\Z)$$
		is a cdh sheaf, so the common fibre of the horizontal maps in the cartesian square of Remark~\ref{remarkmaincartesiansquareformotiviccohomology} is also a cdh sheaf. In particular, for every qcqs scheme $X$, the left vertical map of this cartesian square exhibits cdh-local motivic cohomology as the cdh sheafification of motivic cohomology:
		$$\Z(i)^{\text{cdh}}(X) \simeq \big(L_{\text{cdh}} \Z(i)^{\text{mot}}\big)(X).$$
	\end{remark}
	
	\begin{proposition}\label{propositionpadicstructuremain}
		Let $X$ be a qcqs derived scheme, and $p$ be a prime number. Then for every integer $k \geq 1$, there is a natural cartesian square
		$$\begin{tikzcd}
			\emph{Fil}^\star_{\emph{mot}} \emph{K}(X;\Z/p^k) \arrow{r} \arrow{d} & \emph{Fil}^\star_{\emph{BMS}} \emph{TC}(X;\Z/p^k) \ar[d] \\
			\emph{Fil}^\star_{\emph{cdh}} \emph{KH}(X;\Z/p^k) \arrow{r} & \emph{Fil}^\star_{\emph{BMS}} L_{\emph{cdh}} \emph{TC}(X;\Z/p^k)
		\end{tikzcd}$$
		of filtered $\mathbb{E}_\infty$-rings.
	\end{proposition}
	
	\begin{proof}
		This is a consequence of Proposition~\ref{propositionpcompletionofmotiviconTCisBMS} and Definition~\ref{definitionmotivicfiltrationonKtheoryofderivedschemes}.
	\end{proof}
	
	\begin{corollary}\label{corollarymainpadicstructureongradeds}
		Let $X$ be a qcqs derived scheme, and $p$ be a prime number. Then for any integers $i \in \Z$ and $k \geq 1$, there is natural cartesian square
		$$\begin{tikzcd}
			\Z/p^k(i)^{\emph{mot}}(X) \arrow{r} \arrow{d} & \Z/p^k(i)^{\emph{BMS}}(X) \ar[d] \\
			\Z/p^k(i)^{\emph{cdh}}(X) \arrow{r} & \big(L_{\emph{cdh}} \Z/p^k(i)^{\emph{BMS}}\big)(X)
		\end{tikzcd}$$
		in the derived category $\mathcal{D}(\Z/p^k)$.
	\end{corollary}
	
	\begin{proof}
		This is a direct consequence of Proposition~\ref{propositionpadicstructuremain}.
	\end{proof}
	
	\begin{remark}[$\ell$-adic motivic cohomology]\label{remarkladicmotiviccohomology}
		For any prime number $p$ and integer $k \geq 1$, the filtered presheaf $\text{Fil}^\star_{\text{BMS}} \text{TC}(-;\Z/p^k)$ and its cdh sheafification vanish on qcqs derived $\Z[\tfrac{1}{p}]$-schemes. In particular, Proposition~\ref{propositionpadicstructuremain} implies that for every qcqs derived $\Z[\tfrac{1}{p}]$-scheme $X$, the natural map
		$$\text{Fil}^\star_{\text{mot}} \text{K}(X;\Z/p^k) \longrightarrow \text{Fil}^\star_{\text{cdh}} \text{KH}(X;\Z/p^k)$$
		is an equivalence of filtered $\mathbb{E}_\infty$-rings.    
	\end{remark}
	
	\begin{remark}[Completeness and exhaustivity of $\text{Fil}^\star_{\text{mot}} \text{K}$]
		The filtration $\text{Fil}^\star_{\text{mot}} \text{K}(X)$ of Definition~\ref{definitionmotivicfiltrationonKtheoryofderivedschemes} will be proved to be $\N$-indexed, hence exhaustive, on general qcqs derived schemes $X$ (Theorem~\ref{propositionmotivicfiltrationisexhaustive}), and complete on qcqs schemes of finite valuative dimension (Proposition~\ref{propositionmotivicfiltrationiscompleteonqcqsschemesoffinitevaluativedimension}). Note that these results can already be proved modulo any prime number~$p$, as a formal consequence of Proposition~\ref{propositionpadicstructuremain} and Section~\ref{sectionmotivicfiltrationonTC}.
	\end{remark}
	
	The following result is a filtered version of the classical Dundas--Goodwillie--McCarthy theorem (\cite[Theorem~$7.0.0.2$]{dundas_local_2013}).
	
	\begin{proposition}\label{theoremfilteredDGMtheorem}
        Let $Y \rightarrow X$ be a map of qcqs derived schemes which is Zariski-locally given by a map of animated commutative rings that is surjective with nilpotent kernel on $\pi_0$. Then the natural commutative diagram
        $$\begin{tikzcd}
			\emph{Fil}^\star_{\emph{mot}} \emph{K}(X) \arrow{r} \arrow{d} & \emph{Fil}^\star_{\emph{mot}} \emph{TC}(X) \ar[d] \\
			\emph{Fil}^\star_{\emph{mot}} \emph{K}(Y) \arrow{r} & \emph{Fil}^\star_{\emph{mot}} \emph{TC}(Y)
		\end{tikzcd}$$
        is a cartesian square of filtered $\mathbb{E}_\infty$-rings.
	\end{proposition}
	
	\begin{proof}
        
		By Definition~\ref{definitionmotivicfiltrationonKtheoryofderivedschemes}, it suffices to prove that the natural commutative diagram
        $$\begin{tikzcd}
			\text{Fil}^\star_{\text{mot}} \text{K}(X^{\text{cl}}) \arrow{r} \arrow{d} & \text{Fil}^\star_{\text{mot}} \text{TC}(X^{\text{cl}}) \ar[d] \\
			\text{Fil}^\star_{\text{mot}} \text{K}(Y^{\text{cl}}) \arrow{r} & \text{Fil}^\star_{\text{mot}} \text{TC}(Y^{\text{cl}})
		\end{tikzcd}$$
		is a cartesian square of filtered spectra.
		For every cdh sheaf $F$ defined on qcqs schemes, the natural map $F(X^{\text{cl}}) \rightarrow F(Y^{\text{cl}})$ is an equivalence by Zariski descent and invariance under nilpotent extensions of commutative rings for cdh sheaves. The result is then a consequence of Definition~\ref{definitionmotivicfiltrationonKtheoryofschemes}.
	\end{proof}

	\section{Rational structure of motivic cohomology}\label{sectionratonialstructure}
	
	\vspace{-\parindent}
	\hspace{\parindent}
	
	In this section, we prove some fundamental properties of the motivic filtration that we introduced in the previous section: namely, that it is always exhaustive and finitary, and complete on schemes of finite valuative dimension. These properties modulo a prime number $p$ are a formal consequence of the analogous properties for the BMS filtration on $p$-completed topological cyclic homology. Proving these results integrally will however require more understanding of the rational part of this motivic filtration. Our two main results on rational motivic cohomology are as follows.
	
	The first main result is the following generalisation of the rational splitting of the classical motivic filtration. 
	
	\begin{theorem}[The motivic filtration is rationally split]\label{theorem21'rationalmotivicfiltrationsplits}
		Let $X$ be a qcqs derived scheme. Then there exists a natural multiplicative equivalence of filtered spectra
		$$\emph{Fil}^\star_{\emph{mot}} \emph{K}(X;\Q) \simeq \bigoplus_{j \geq \star} \Q(j)^{\emph{mot}}(X)[2j].$$
	\end{theorem}
	
	As for the classical motivic filtration, this splitting is induced by suitable Adams operations, which we construct in the generality of qcqs derived schemes in the first section. 
	
	This result is however not enough to prove the exhaustivity, completeness, or finitariness of the motivic filtration. Instead, these will be proved along the way to the following second main result on rational motivic cohomology.
	
	\begin{theorem}\label{theoremrationalstructuremain}
		Let $X$ be a qcqs scheme. Then there is a natural fibre sequence of filtered spectra
		$$\emph{Fil}^\star_{\emph{mot}} \emph{K}(X;\Q) \rightarrow \emph{Fil}^\star_{\emph{cdh}} \emph{KH}(X;\Q) \rightarrow \emph{cofib}\big( \emph{Fil}^{\star}_{\emph{HKR}} \emph{HC}(X_{\Q}/\Q) \rightarrow \emph{Fil}^{\star}_{\emph{HKR}} L_{\emph{cdh}} \emph{HC}(X_{\Q}/\Q)\big)[1].$$
	\end{theorem}
	
	To prove this result, we use a classical argument of Weibel at the level of $K$-theory (Lemma~\ref{lemma2WeibelargumentrationalKtheory}) and the rational splitting Theorem~\ref{theorem21'rationalmotivicfiltrationsplits} to reduce the statement to the case of characteristic zero, where the result is essentially \cite[Theorem~$4.10\,(3)$]{elmanto_motivic_2023}. 
	
	The key step in this argument, in order to pass from a statement at the level of $K$\nobreakdash-theory to a filtered statement, is to prove beforehand that the motivic filtration $\text{Fil}^\star_{\text{mot}} \text{K}$ is $\N$\nobreakdash-indexed (Theorem~\ref{propositionmotivicfiltrationisexhaustive}). The strategy to prove this is as follows. We first introduce a new rigid-analytic variant of the HKR filtrations in the generality of qcqs derived schemes, whose graded pieces are rigid-analytic variants of derived de Rham cohomology. We then adapt a theorem of Goodwillie --stating in modern language that periodic cyclic homology is truncating in characteristic zero-- to this rigid-analytic variant of periodic cyclic homology. This rigid-analytic Goodwillie theorem implies, by the work of Land--Tamme on truncating invariants, that the rigid-analytic variant of periodic cyclic homology is a cdh sheaf on qcqs schemes. A filtered consequence of this cdh descent result then formally implies the desired result, {\it i.e.}, that the motivic filtration $\text{Fil}^\star_{\text{mot}} \text{K}$ is $\N$\nobreakdash-indexed.

    \medskip
	
	\subsection{Adams operations}\label{subsectionAdamsoperations}
	
	\vspace{-\parindent}
	\hspace{\parindent}
	
	In this subsection we construct Adams operations on filtered algebraic $K$-theory of qcqs derived schemes (Construction~\ref{constuction22'AdamsonKtheory}).
	
	In \cite[Section~3.3]{bachmann_eta-periodic_2020} (see also \cite[Appendix~B.1]{elmanto_motivic_2023} for a careful treatment of the multiplicative structures), Bachmann--Hopkins construct Adams operations $\psi^m$ on the $m$-periodic filtered $K$-theory $\text{Fil}^\star_{\text{cla}} \text{K}(X)[\tfrac{1}{m}]$ of smooth $\Z$-schemes~$X$, acting on the $i^{\text{th}}$ graded piece as multiplication by $m^i$. For the highly structured multiplicative definition of the endomorphism multiplica\-tion-by-$m^\star$ that we will use in this subsection, we refer to Raksit's construction \cite[Construction~6.4.9 and Remark~6.4.10]{raksit_hochschild_2026} (see also \cite[Remark~B.1]{elmanto_motivic_2023}). Using these Adams operations, we construct Adams operations $\psi^m$ on the filtered $p$-completed topological cyclic homology $\text{Fil}^\star_{\text{BMS}} \text{TC}(X;\Z_p)$ of qcqs $\Z[\tfrac{1}{m}]$-schemes $X$.
    
    \begin{proposition}[Adams operations on filtered $\text{TC}(-;\Z_p)$]\label{proposition24AdamsoprationsonfilteredTCpcompleted}
        Let $m \in \Z \setminus \{0\}$ be an integer, $X$ be a qcqs derived $\Z[\tfrac{1}{m}]$-scheme, and $p$ be a prime number. Then there exists a natural endomorphism $\psi^m$ of the filtered $\mathbb{E}_\infty$-ring $\emph{Fil}^\star_{\emph{BMS}} \emph{TC}(X;\Z_p)$ such that the induced endomorphism of the graded $\mathbb{E}_\infty$-$\Z$-algebra $\Z_p(\star)^{\emph{BMS}}(X)[2\star]$ is naturally homotopic to its endomorphism multiplication-by-$m^\star$. Moreover, this endomorphism $\psi^m$ is uniquely determined by its naturality and the fact that on smooth $\Z[\tfrac{1}{m}]$-schemes $X$, the diagram of filtered $\mathbb{E}_\infty$-rings
        $$\begin{tikzcd}
			\emph{Fil}^\star_{\emph{cla}} \emph{K}(X)[\tfrac{1}{m}] \arrow{r} \arrow[d,"\psi^m"] & \emph{Fil}^\star_{\emph{BMS}} \emph{TC}(X;\Z_p) \arrow[d,"\psi^m"] \\
			\emph{Fil}^\star_{\emph{cla}} \emph{K}(X)[\tfrac{1}{m}] \arrow{r} & \emph{Fil}^\star_{\emph{BMS}} \emph{TC}(X;\Z_p)
		\end{tikzcd}$$
		where the horizontal maps are induced by Theorem~\ref{theoremfilteredcyclotomictraceinthesmoothcase}, and the left vertical map is the Adams operation of \cite[Appendix~B.1]{elmanto_motivic_2023}, is commutative.
    \end{proposition}
	
	\begin{proof}
		The functor $\text{Fil}^\star_{\text{BMS}} \text{TC}(-;\Z_p)$ on animated commutative rings is naturally identified with the $p$-completed left Kan extension of the étale sheafification of the functor $\text{Fil}^\star_{\text{cla}} \text{K}(-)$ from smooth $\Z$-algebras to animated commutative rings (Proposition~\ref{remarkBMSfiltrationonTCcanbereconstructedfromclassicalmotivicfiltration}). In particular, the Adams operation $\psi^m$ on the functor $\text{Fil}^\star_{\text{cla}} \text{K}(-)[\tfrac{1}{m}]$ on smooth $\Z$-algebras (\cite[Appendix~B.1]{elmanto_motivic_2023}) induces (by étale-sheafifying, $p$-completely left Kan extending to animated commutative rings, and finally Zariski-sheafifying), a natural endomorphism $\psi^m$ of the presheaf of filtered $\mathbb{E}_\infty$-rings $\text{Fil}^\star_{\text{BMS}} \text{TC}(-;\Z_p)$ on qcqs derived schemes, such that the induced endomorphism of the presheaf of graded $\mathbb{E}_\infty$-$\Z$-algebra $\Z_p(\star)^{\text{BMS}}(-)[2\star]$ is naturally homotopic to its endomorphism multiplication-by-$m^\star$. The last claim is similarly a consequence of Proposition~\ref{remarkBMSfiltrationonTCcanbereconstructedfromclassicalmotivicfiltration}.
	\end{proof}
	
	The following result is \cite[Construction~$6.4.8$ and Proposition~$6.4.12$]{raksit_hochschild_2026}. Note that if $X$ is a qcqs derived $\Z[\tfrac{1}{m}]$-scheme, then the filtered spectrum $\text{Fil}^\star_{\text{HKR}} \text{HC}^-(X)$ is naturally $\Z[\tfrac{1}{m}]$-linear.

    \begin{proposition}[Adams operations on filtered $\text{HC}^-$, \cite{raksit_hochschild_2026}]\label{proposition23raksitAdamsoperationsonHC-withmintertibleinthescheme}
        Let $m \in \Z \setminus \{0\}$ be an integer, and $X$ be a qcqs derived $\Z[\tfrac{1}{m}]$-scheme. Then there exists a natural endomorphism $\psi^m$ of the filtered $\mathbb{E}_\infty$-ring $\emph{Fil}^\star_{\emph{HKR}} \emph{HC}^-(X)$ such that the induced endomorphism of the graded $\mathbb{E}_\infty$-$\Z$-algebra $R\Gamma_{\emph{Zar}}(X,\widehat{\mathbb{L}\Omega}^{\geq \star}_{-/\Z[\tfrac{1}{m}]})[2\star]$ is naturally homotopic to its endomorphism multiplication-by-$m^\star$.
	\end{proposition}
	\begin{lemma}\label{lemma35compatibilityAdamsTCpcompletedandHC-pcompleted}
		Let $m \geq 2$ be an integer, $X$ be a qcqs derived $\Z[\tfrac{1}{m}]$-scheme, and $p$ be a prime number. Then the natural diagram of filtered $\mathbb{E}_\infty$-rings
		$$\begin{tikzcd}
			\emph{Fil}^\star_{\emph{BMS}} \emph{TC}(X;\Z_p) \arrow{r} \arrow[d,"\psi^m"] & \emph{Fil}^\star_{\emph{HKR}} \emph{HC}^-(X;\Z_p) \arrow[d,"\psi^m"] \\
			\emph{Fil}^\star_{\emph{BMS}} \emph{TC}(X;\Z_p) \arrow{r} & \emph{Fil}^\star_{\emph{HKR}} \emph{HC}^-(X;\Z_p)
		\end{tikzcd}$$
		where the horizontal maps are defined in Construction~\ref{constructionfilteredmapTCpcompletedtoHC-pcompleted}, the left map is the map defined in Proposition~\ref{proposition24AdamsoprationsonfilteredTCpcompleted}, and the right map is the map induced by Proposition~\ref{proposition23raksitAdamsoperationsonHC-withmintertibleinthescheme}, is commutative.
	\end{lemma}
	
	\begin{proof}
        By \cite[Corollary~B.19]{elmanto_motivic_2023}, for every smooth $\Z[\tfrac{1}{m}]$-scheme $X$,\footnote{Note here that although \cite[Appendix~B.4]{elmanto_motivic_2023} is formulated for smooth $k$-schemes $X$ where $k$ is a commutative $\Q$-algebra, the proof adapts readily to smooth $\Z[\tfrac{1}{m}]$-schemes (see in particular \cite[Remark~B.4]{elmanto_motivic_2023}).} the Adams operations on $\text{K}(X)[\tfrac{1}{m}]$ and on $\text{HC}^-(X)$ are multiplicatively compatible. As in \cite[proof of Theorem~4.14]{elmanto_motivic_2023}, this compatibility passes to the filtered level by Theorem~\ref{theoremfilteredcyclotomictraceinthesmoothcase}, {\it i.e.}, the natural diagram of filtered $\mathbb{E}_\infty$-rings
		$$\begin{tikzcd}
			\text{Fil}^\star_{\text{cla}} \text{K}(X)[\frac{1}{m}] \arrow{r} \arrow[d,"\psi^m"] & \text{Fil}^\star_{\text{HKR}} \text{HC}^-(X) \arrow[d,"\psi^m"] \\
			\text{Fil}^\star_{\text{cla}} \text{K}(X)[\tfrac{1}{m}] \arrow{r} & \text{Fil}^\star_{\text{HKR}} \text{HC}^-(X)
		\end{tikzcd}$$
		is commutative for every smooth $\Z[\tfrac{1}{m}]$-scheme $X$. The result is then a consequence of Proposition~\ref{proposition24AdamsoprationsonfilteredTCpcompleted}, where the compatibility between the filtered maps is again a consequence of the proof of Theorem~\ref{theoremfilteredcyclotomictraceinthesmoothcase}.
	\end{proof}
	
	\begin{construction}[Adams operations on filtered TC]\label{construction31'AdamsonTC}
		Let $m \in \Z \setminus \{0\}$ be an integer, and $X$ be a qcqs derived $\Z[\tfrac{1}{m}]$-scheme. The endomorphism $\psi^m$ of the filtered spectrum $\text{Fil}^\star_{\text{mot}} \text{TC}(X)$ is the endomorphism defined by pullback along the natural cartesian square of filtered $\mathbb{E}_\infty$-rings
		$$\begin{tikzcd}
			\text{Fil}^\star_{\text{mot}} \text{TC}(X) \arrow{r} \arrow{d} & \text{Fil}^\star_{\text{HKR}} \text{HC}^-(X) \ar[d] \\
			\prod_{p \in \mathbb{P}} \text{Fil}^\star_{\text{BMS}} \text{TC}(X;\Z_p) \arrow{r} & \prod_{p \in \mathbb{P}} \text{Fil}^\star_{\text{HKR}} \text{HC}^-(X;\Z_p),
		\end{tikzcd}$$
		where the endomorphism $\psi^m$ of $\text{Fil}^\star_{\text{HKR}} \text{HC}^-(X)$ is the endomorphism of Proposition~\ref{proposition23raksitAdamsoperationsonHC-withmintertibleinthescheme}, the endomorphism $\psi^m$ of $\prod_{p \in \mathbb{P}} \text{Fil}^\star_{\text{HKR}} \text{HC}^-(X;\Z_p)$ is induced by the endomorphism $\psi^m$ of $\text{Fil}^\star_{\text{HKR}} \text{HC}^-(X)$, and the endomorphism $\psi^m$ of $\prod_{p \in \mathbb{P}} \text{Fil}^\star_{\text{BMS}} \text{TC}(X;\Z_p)$ is the endomorphism of Proposition~\ref{proposition24AdamsoprationsonfilteredTCpcompleted}. Note here that the compatibility between the endomorphisms $\psi^m$ and the bottom map is given by Lemma~\ref{lemma35compatibilityAdamsTCpcompletedandHC-pcompleted}.
	\end{construction}

	Note in the following result that if $X$ is a qcqs derived $\Z[\tfrac{1}{m}]$-scheme, then the filtered spectrum $\text{Fil}^\star_{\text{mot}} \text{TC}(X)$ is naturally $\Z[\tfrac{1}{m}]$-linear.
	
	\begin{corollary}\label{corollary31''AdamsoperationsongradedpiecesoffilteredTC}
		Let $m \in \Z \setminus \{0\}$ be an integer, and $X$ be a qcqs derived $\Z[\tfrac{1}{m}]$-scheme. Then for every integer $i \in \Z$, the endomorphism $\psi^m$ on the graded $\mathbb{E}_\infty$-$\Z$-algebra $\Z(\star)^{\emph{TC}}(X)[2\star]$, induced by the endomorphism $\psi^m$ of $\emph{Fil}^\star_{\emph{mot}} \emph{TC}(X)$ of Construction~\ref{construction31'AdamsonTC}, is naturally homotopic to its endomorphism multiplication-by-$m^\star$. Moreover, if $X$ is smooth over $\Z[\tfrac{1}{m}]$, then the natural diagram of filtered $\mathbb{E}_\infty$\nobreakdash-rings
		$$\begin{tikzcd}
			\emph{Fil}^\star_{\emph{cla}} \emph{K}(X)[\tfrac{1}{m}] \arrow{r} \arrow[d,"\psi^m"] & \emph{Fil}^\star_{\emph{mot}} \emph{TC}(X) \arrow[d,"\psi^m"] \\
			\emph{Fil}^\star_{\emph{cla}} \emph{K}(X)[\tfrac{1}{m}] \arrow{r} & \emph{Fil}^\star_{\emph{mot}} \emph{TC}(X)
		\end{tikzcd}$$
		where the horizontal maps are defined in Theorem~\ref{theoremfilteredcyclotomictraceinthesmoothcase}, and the left vertical map is defined in \cite[Appendix~B.1]{elmanto_motivic_2023}, is commutative.
	\end{corollary}
	
	\begin{proof}
		The identification of the endomorphism $\psi^m$ on the graded pieces is a consequence of Propositions~\ref{proposition24AdamsoprationsonfilteredTCpcompleted} and \ref{proposition23raksitAdamsoperationsonHC-withmintertibleinthescheme}, where we use the same argument as in \cite[proof of Theorem~4.14]{elmanto_motivic_2023} to check that the homotopy between the two endomorphisms is multiplicative. The second statement is a consequence of the analogous compatibilities for $\prod_{p \in \mathbb{P}} \text{Fil}^\star_{\text{BMS}} \text{TC}(X;\Z_p)$ (Proposition~\ref{proposition24AdamsoprationsonfilteredTCpcompleted}) and for $\text{Fil}^\star_{\text{HKR}} \text{HC}^-(X)$ (\cite[Appendix~B.4]{elmanto_motivic_2023}), and of Lem\-ma~\ref{lemma35compatibilityAdamsTCpcompletedandHC-pcompleted}.
	\end{proof}
	
	\begin{lemma}\label{lemma36compatibilityAdamsoperationsonKHandLcdhTC}
		Let $m \in \Z \setminus \{0\}$ be an integer, and $X$ be a qcqs $\Z[\tfrac{1}{m}]$-scheme. Then the natural diagram of filtered $\mathbb{E}_\infty$-rings
		$$\begin{tikzcd}
			\emph{Fil}^\star_{\emph{cdh}} \emph{KH}(X)[\tfrac{1}{m}] \arrow{r} \arrow[d,"\psi^m"] & \emph{Fil}^\star_{\emph{mot}} L_{\emph{cdh}} \emph{TC}(X) \arrow[d,"\psi^m"] \\
			\emph{Fil}^\star_{\emph{cdh}} \emph{KH}(X)[\frac{1}{m}] \arrow{r} & \emph{Fil}^\star_{\emph{mot}} L_{\emph{cdh}} \emph{TC}(X)
		\end{tikzcd}$$
		where the horizontal maps are defined in Construction~\ref{constructionfilteredcdhlocalcyclotomictrace}, the left map is the map of Theorem~\ref{theoremBEM}\,(2), and the right map is the map induced by Construction~\ref{construction31'AdamsonTC}, is commutative.
	\end{lemma}
	
	\begin{proof}
		The left map $\psi^m$ of this diagram is defined by cdh sheafifying the left Kan extension from smooth $\Z[\tfrac{1}{m}]$-schemes to qcqs $\Z[\tfrac{1}{m}]$-schemes of the endomorphism $\psi^m$ on $\text{Fil}^\star_{\text{cla}} \text{K}(-)[\tfrac{1}{m}]$ (\cite[Appendix~B.1]{elmanto_motivic_2023}). The result then follows from Construction~\ref{constructionfilteredcdhlocalcyclotomictrace} and Corollary~\ref{corollary31''AdamsoperationsongradedpiecesoffilteredTC}. 
	\end{proof}
	
	\begin{construction}[Adams operations on filtered $K$-theory]\label{constuction22'AdamsonKtheory}
		Let $m \in \Z \setminus \{0\}$ be an integer. Following Definition~\ref{definitionmotivicfiltrationonKtheoryofschemes}, if $X$ is a qcqs $\Z[\tfrac{1}{m}]$-scheme, the endomorphism $\psi^m$ of the filtered $\mathbb{E}_\infty$-ring $\text{Fil}^\star_{\text{mot}} \text{K}(X)[\frac{1}{m}]$ is the endomorphism defined by pullback along the natural cartesian square of filtered spectra
		$$\begin{tikzcd}
			\text{Fil}^\star_{\text{mot}} \text{K}(X)[\tfrac{1}{m}] \arrow{r} \arrow{d} & \text{Fil}^\star_{\text{mot}} \text{TC}(X) \ar[d] \\
			\text{Fil}^\star_{\text{cdh}} \text{KH}(X)[\tfrac{1}{m}] \arrow{r} & \text{Fil}^\star_{\text{mot}} L_{\text{cdh}} \text{TC}(X),
		\end{tikzcd}$$
		where the endomorphism $\psi^m$ of $\text{Fil}^\star_{\text{mot}} \text{TC}(X)$ is the endomorphism of Construction~\ref{construction31'AdamsonTC}, the endomorphism $\psi^m$ of $\text{Fil}^\star_{\text{mot}} L_{\text{cdh}} \text{TC}(X)$ is defined by cdh sheafifying the endomorphism $\psi^m$ of $\text{Fil}^\star_{\text{mot}} \text{TC}(-)$, and the endomorphism $\psi^m$ of $\text{Fil}^\star_{\text{cdh}} \text{KH}(X)[\tfrac{1}{m}]$ is the endomorphism of Theorem~\ref{theoremBEM}\,(2). Note here that the compatibility between the endomorphisms $\psi^m$ and the bottom map is given by Lemma~\ref{lemma36compatibilityAdamsoperationsonKHandLcdhTC}.
		
		Following Definition~\ref{definitionmotivicfiltrationonKtheoryofderivedschemes}, if $X$ is a qcqs derived $\Z[\tfrac{1}{m}]$-scheme, the endomorphism $\psi^m$ of the filtered $\mathbb{E}_\infty$-ring $\text{Fil}^\star_{\text{mot}} \text{K}(X)[\tfrac{1}{m}]$ is the endomorphism defined by pullback along the natural cartesian square of filtered $\mathbb{E}_\infty$-rings
		$$\begin{tikzcd}
			\text{Fil}^\star_{\text{mot}} \text{K}(X)[\tfrac{1}{m}] \ar[r] \ar[d] & \text{Fil}^\star_{\text{mot}} \text{TC}(X) \ar[d] \\
			\text{Fil}^\star_{\text{mot}} \text{K}(X^{\text{cl}})[\tfrac{1}{m}] \ar[r] & \text{Fil}^\star_{\text{mot}} \text{TC}(X^{\text{cl}}), 
		\end{tikzcd}$$
		where the endomorphisms $\psi^m$ of $\text{Fil}^\star_{\text{mot}} \text{TC}(X)$ and $\text{Fil}^\star_{\text{mot}} \text{TC}(X^{\text{cl}})$ are defined in Construction~\ref{construction31'AdamsonTC}, and the endomorphism $\psi^m$ of $\text{Fil}^\star_{\text{mot}} \text{K}(X^{\text{cl}})[\tfrac{1}{m}]$ is the endomorphism of the previous paragraph. Note here that the compatibility between the endomorphisms $\psi^m$ is automatic by construction.
	\end{construction}
	
	\begin{corollary}\label{corollary22''AdamsoperationsongradedpiecesoffilteredKtheory}
        Let $m \in \Z \setminus \{0\}$ be an integer, and $X$ be a qcqs derived $\Z[\tfrac{1}{m}]$-scheme. Then the endomorphism $\psi^m$ on the graded $\mathbb{E}_\infty$-$\Z$-algebra $\Z[\tfrac{1}{m}](\star)^{\emph{mot}}(X)[2\star]$, induced by the endomorphism $\psi^m$ of $\emph{Fil}^\star_{\emph{mot}} \emph{K}(X)[\tfrac{1}{m}]$ of Construction~\ref{constuction22'AdamsonKtheory}, is naturally homotopic to its endomorphism multiplication-by\nobreakdash-$m^\star$.
	\end{corollary}
	
	\begin{proof}
		This is a consequence of Theorem~\ref{theoremBEM}\,(2) and~Corollary~\ref{corollary31''AdamsoperationsongradedpiecesoffilteredTC}. 
	\end{proof}
	
	The following lemma explains how to use Adams operations to deduce splitting results on the rationalisation of certain filtrations.
	
	\begin{lemma}\label{lemma29howtouseAdamsoperations}
		Let $$\emph{Fil}^\star F(-) : \emph{dSch}^{\emph{qcqs,op}} \longrightarrow \emph{CAlg}(\emph{FilSp})$$ be a Zariski sheaf of filtered $\mathbb{E}_\infty$-rings. For each integer $m \geq 2$, let $\psi^m$ be a natural multiplicative endomorphism of the filtered spectrum $\emph{Fil}^\star F(X)$ on qcqs derived $\Z[\tfrac{1}{m}]$-schemes $X$, satisfying the following properties:
		\begin{enumerate}[label=(\roman*)]
			\item for every qcqs derived scheme $X$, the rationalised filtration $\emph{Fil}^\star F(X;\Q)$ is complete;
			\item for every integer $m \geq 2$ and every qcqs derived $\Z[\tfrac{1}{m}]$-scheme $X$, the graded $\mathbb{E}_\infty$-ring $\emph{gr}^\star F(X)$ is $\Z$-linear, and the induced endomorphism $\psi^m$ of the graded $\mathbb{E}_\infty$-$\Z$-algebra $\emph{gr}^\star F(X)$ is naturally homotopic to its endomorphism multiplication-by-$m^\star$;
			\item for any integers $m, m' \geq 2$, and every qcqs derived $\Z[\tfrac{1}{mm'}]$-scheme $X$, the space of natural transformations from $\psi^m \circ \psi^{m'}$ to $\psi^{mm'}$, as endomorphisms of the filtered $\mathbb{E}_\infty$-ring $\emph{Fil}^\star F(X)$, is contractible.
		\end{enumerate}
		Then for every qcqs derived scheme $X$, and every integer $k \geq 0$, there exists a natural equivalence of filtered $\mathbb{E}_\infty$-rings
        $$\emph{Fil}^\star F(X;\Q) \simeq \big(\bigoplus_{\star \leq j < \star + k} \emph{gr}^j F(X;\Q)\big) \oplus \emph{Fil}^{\star + k} F(X;\Q).$$
	\end{lemma}
	
	\begin{proof}
		For every spectrum $C$ equipped with a map $F : C \rightarrow C$, denote by $C^F$ the homotopy fibre of the map $F$. Let $m \geq 2$ be an integer, $i,k \in \Z$ be integers such that $k \geq 0$, and $X$ be a qcqs derived $\Z[\tfrac{1}{m}]$-scheme, for which we first construct the equivalence of spectra
		$$\text{Fil}^i F(X;\Q) \simeq \big(\bigoplus_{i \leq j < i+k} \text{gr}^j F(X;\Q)\big) \oplus \text{Fil}^{i+k} F(X;\Q).$$
        The fact that this equivalence of spectra naturally upgrades to the desired equivalence of filtered $\mathbb{E}_\infty$\nobreakdash-rings can be proved exactly as in \cite[proof of Theorem~4.13.2]{elmanto_motivic_2023}; we focus here on the argument at the level of filtered spectra for expositional purposes, and refer to their proof for a careful treatment of multiplicative structures. 
		We first prove that the spectrum $\big(\text{Fil}^{i+1} F(X;\Q)\big)^{\psi^m-m^i}$ is zero. The filtration $\text{Fil}^{i+1+\star} F(X;\Q)$ induced on the spectrum $\text{Fil}^{i+1} F(X;\Q)$ is complete (hypothesis~$(i)$), so it suffices to prove that the natural map
		$$\psi^m - m^i : \text{gr}^j F(X;\Q) \longrightarrow \text{gr}^j F(X;\Q)$$
		is an equivalence of spectra for every integer $j \geq i+1$. For every integer $j \geq i+1$, this map can be identified with multiplication by the nonzero integer $m^i(m^{j-i}-1)$ on the $\Q$-linear spectrum $\text{gr}^j F(X;\Q)$ (hypothesis ($ii$)), and is thus an equivalence. Taking the fibre of the natural map $\psi^m - m^i$ on the fibre sequence of spectra
		$$\text{Fil}^{i+1} F(X;\Q) \longrightarrow \text{Fil}^i F(X;\Q) \longrightarrow \text{gr}^i F(X;\Q),$$
		this implies that the natural map
		$$\big(\text{Fil}^i F(X;\Q)\big)^{\psi^m-m^i} \longrightarrow \big(\text{gr}^i F(X;\Q)\big)^{\psi^m-m^i}$$
		is an equivalence of spectra. The spectrum $\big(\text{gr}^i F(X;\Q)\big)^{\psi^m-m^i}$ can be naturally identified with the spectrum $\text{gr}^i F(X;\Q) \oplus \text{gr}^i F(X;\Q)[-1]$ (hypothesis ($ii$)), and the induced composite map
		$$\text{gr}^i F(X;\Q) \longrightarrow \big(\text{gr}^i F(X;\Q)\big)^{\psi^m-m^i} \xlongrightarrow{\sim} \big(\text{Fil}^i F(X;\Q)\big)^{\psi^m-m^i} \xlongrightarrow{\text{can}} \text{Fil}^i F(X;\Q)$$
		induces a natural splitting of spectra
		$$\text{Fil}^i F(X;\Q) \simeq \text{gr}^i F(X;\Q) \oplus \text{Fil}^{i+1} F(X;\Q).$$
		By induction, this implies that for every integer $k \geq 0$, there is a natural equivalence of spectra
		$$\text{Fil}^i F(X;\Q) \simeq \big(\bigoplus_{i \leq j < i+k} \text{gr}^j F(X;\Q)\big) \oplus \text{Fil}^{i+k} F(X;\Q).$$
		
		We now prove the desired equivalence of filtered $\mathbb{E}_\infty$-rings for a general qcqs derived scheme~$X$. The presheaf $\text{Fil}^\star F(-)$ being a Zariski sheaf of filtered $\mathbb{E}_\infty$-rings, the presheaves 
        $$\text{Fil}^\star F(-;\Q) \text{ and } \big(\bigoplus_{\star \leq j < \star + k} \text{gr}^j F(-;\Q)\big) \oplus \text{Fil}^{\star + k} F(-;\Q)$$ 
        are Zariski sheaves of filtered $\mathbb{E}_\infty$-rings. Let $m$ and $m'$ be two distinct prime numbers. By Zariski excision, any Zariski sheaf $E$ on qcqs derived schemes satisfies that for every qcqs derived scheme $X$, the natural commutative diagram
        $$\begin{tikzcd}
            E(X) \ar[r] \ar[d] & E\big(X_{\Z[\tfrac{1}{m}]}\big) \ar[d] \\
            E\big(X_{\Z[\tfrac{1}{m'}]}\big) \ar[r] & E\big(X_{\Z[\tfrac{1}{mm'}]}\big)
        \end{tikzcd}$$
        is cartesian. We use this cartesian square to prove the desired equivalence of filtered $\mathbb{E}_\infty$-rings for a general qcqs derived scheme $X$, by pulling back the equivalences constructed in the previous paragraph. The compatibility between the constructions of the equivalences for $m$, $m'$, and $mm'$ only depend on the multiplicative choices, for every integer $i \in \Z$, of the identification of the maps $\psi^m$, $\psi^{m'}$, and $\psi^{mm'}$ on the $i^{\text{th}}$ graded piece with multiplication by $m^i$, $(m')^i$, and $(mm')^i$ respectively. These multiplicative choices are compatible up to homotopy (hypothesis ($iii$)), hence defining the desired equivalence for $X$. The fact that the constructed equivalence does not depend on the choices of the prime numbers $m$ and $m'$ is similarly a consequence of Zariski excision and of hypothesis $(iii)$.
	\end{proof}

	\begin{remark}\label{remarklemmahowtouseAdamsonclassicalschemes}
		Lemma~\ref{lemma29howtouseAdamsoperations} can also be proved, with the same proof, for Zariski sheaves of filtered $\mathbb{E}_\infty$-rings $\text{Fil}^\star F(-)$ that are defined on qcqs schemes, on qcqs schemes of finite valuative dimension, on noetherian schemes of finite dimension, or on smooth schemes over a given commutative ring.
	\end{remark}

	\begin{remark}
		Hypothesis ($iii$) in Lemma~\ref{lemma29howtouseAdamsoperations} follows from the construction of the Adams operations on the filtrations $\text{Fil}^\star_{\text{cla}} \text{K}(-)$ and $\text{Fil}^\star_{\text{HKR}} \text{HC}^-(-)$, when these are defined. This hypothesis ($iii$) then follows formally for all the other filtrations considered in this subsection.
	\end{remark}

	\begin{proposition}\label{propositionhowtouseAdamsoperations}
		Let $$\emph{Fil}^\star F(-) : \emph{Sch}^{\emph{qcqs,op}} \longrightarrow \emph{CAlg}(\emph{FilSp})$$ be a finitary Zariski sheaf of filtered $\mathbb{E}_\infty$-rings. For each integer $m \geq 2$, let $\psi^m$ be a natural multiplicative endomorphism of the filtered $\mathbb{E}_\infty$-ring $\emph{Fil}^\star F(X)$ on qcqs $\Z[\tfrac{1}{m}]$-schemes $X$, satisfying the following properties:
		\begin{enumerate}[label=(\roman*)]
			\item for every noetherian scheme $X$ of finite dimension, there exists an integer $d \in \Z$ such that for every integer $i \in \Z$, the spectrum $\emph{Fil}^i F(X;\Q)$ is in cohomological degrees at most $-i+d$;
			\item for every integer $m \geq 2$ and every qcqs derived $\Z[\tfrac{1}{m}]$-scheme $X$, the graded $\mathbb{E}_\infty$-ring $\emph{gr}^\star F(X)$ is $\Z$-linear, and the induced endomorphism $\psi^m$ of the graded $\mathbb{E}_\infty$-$\Z$-algebra $\emph{gr}^\star F(X)$ is naturally homotopic to its endomorphism multiplication-by-$m^\star$;
			\item for any integers $m, m' \geq 2$, and every qcqs $\Z[\tfrac{1}{mm'}]$-scheme $X$, the space of natural transformations from $\psi^m \circ \psi^{m'}$ to $\psi^{mm'}$, as endomorphisms of the filtered $\mathbb{E}_\infty$-ring $\emph{Fil}^\star F(X)$, is contractible.
		\end{enumerate}
		Then for every qcqs scheme $X$, there exists a natural multiplicative equivalence of filtered spectra
		$$\emph{Fil}^\star F(X;\Q) \simeq \bigoplus_{j \geq \star} \emph{gr}^j F(X;\Q).$$
	\end{proposition}

	\begin{proof}
		By finitariness, it suffices to prove the result on noetherian schemes of finite dimension. Hypothesis $(i)$ implies that the filtration $\text{Fil}^\star F(X;\Q)$ is complete on such schemes $X$. Lemma~\ref{lemma29howtouseAdamsoperations} and Remark~\ref{remarklemmahowtouseAdamsonclassicalschemes} then imply that there exist natural equivalences of filtered $\mathbb{E}_\infty$-rings
        $$\text{Fil}^\star F(X;\Q) \simeq \big(\oplus_{\star \leq j < \star + k} \text{gr}^j F(X;\Q)\big) \oplus \text{Fil}^{\star + k} F(X;\Q)$$
		for all integers $i,k \in \Z$ such that $k \geq 0$. Again using completeness, taking the limit over $k \rightarrow +\infty$ induces a natural equivalence of filtered $\mathbb{E}_\infty$-rings
		$$\text{Fil}^\star F(X;\Q) \simeq \prod_{j \geq \star} \text{gr}^j F(X;\Q).$$
        For every integer $n \in \Z$, applying the cohomological truncation $\tau^{\geq -n}$ to this equivalence induces the equivalence of filtered $\mathbb{E}_\infty$-rings
        $$\tau^{\geq -n} \big(\text{Fil}^\star F(X;\Q)\big) \simeq \tau^{\geq -n} \big(\bigoplus_{j \ge \star} \text{gr}^j F(X;\Q)\big),$$
        where the product becomes a direct sum because only finitely many factors are nonzero by Hypothesis~$(i)$. Taking the limit over $n \rightarrow +\infty$ then induces, by Postnikov left-completeness of the category of spectra, the desired equivalence
		$$\text{Fil}^\star F(X;\Q) \simeq \bigoplus_{j \geq \star} \text{gr}^j F(X;\Q)$$
        of filtered $\mathbb{E}_\infty$-rings.
	\end{proof}

	\begin{remark}
		Let $X$ be a qcqs derived scheme. The argument of Proposition~\ref{propositionhowtouseAdamsoperations}, where the necessary hypotheses are satisfied by Corollary~\ref{corollary31''AdamsoperationsongradedpiecesoffilteredTC} and Proposition~\ref{propositionfilteredTCandrationalfilteredTCarecomplete}, implies that there is a natural multiplicative equivalence of filtered spectra
		$$\text{Fil}^\star_{\text{mot}} \text{TC}(X;\Q) \simeq \prod_{j \geq \star} \Q(j)^{\text{TC}}(X)[2i].$$
        Note here that having a product (rather than a direct sum) in this equivalence is related to the non-finitariness of rational topological cyclic homology.
	\end{remark}

	\medskip
	
	\subsection{Rigid-analytic HKR filtrations}\label{sectionrigidanalyticdR}
	
	\vspace{-\parindent}
	\hspace{\parindent}
	
	In this subsection, we define variants, which we call rigid-analytic, of the HKR filtrations. We start by explaining the relevant objects at the level of Hochschild homology.
	
	For every qcqs derived scheme $X$, there is a natural $\text{S}^1$-equivariant arithmetic fracture square
	\begin{equation}\label{equationarithmeticfracturesquareforHH}
	\begin{tikzcd}
		\text{HH}(X) \ar[r] \ar[d] & \text{HH}(X_{\Q}/\Q) \ar[d] \\
		\prod_{p \in \mathbb{P}} \text{HH}(X;\Z_p) \ar[r] & \text{HH}(X;\mathbb{A}_f)
	\end{tikzcd}
	\end{equation}
	in the derived category $\mathcal{D}(\Z)$, where we use base change for Hochschild homology for the top right corner, and the convention adopted in the Notation part for the bottom left and bottom right corners. Applying homotopy fixed points $(-)^{h\text{S}^1}$ for the $\text{S}^1$-action induces a natural cartesian square
	$$\begin{tikzcd}
		\text{HC}^-(X) \ar[r] \ar[d] & \text{HC}^-(X_{\Q}/\Q) \ar[d] \\
		\prod_{p \in \mathbb{P}} \text{HC}^-(X;\Z_p) \ar[r] & \text{HH}(X;\mathbb{A}_f)^{h\text{S}^1}
	\end{tikzcd}$$
	in the derived category $\mathcal{D}(\Z)$, where we use that taking homotopy fixed points $(-)^{h\text{S}^1}$ commutes with limits for the bottom left corner. We call {\it rigid-analytic variant of Hochschild homology and of negative cyclic homology} the bottom right corners of the previous two cartesian squares, respectively. This terminology should find some justification in Section~\ref{subsectionrigidanalyticGoodwillietheorem}. Following Section~\ref{subsectionHKRfiltrations}, we use the previous cartesian square to introduce a new HKR filtration on this rigid-analytic variant of negative cyclic homology $\text{HH}(X;\mathbb{A}_f)^{h\text{S}^1}$.
	
	\begin{definition}[HKR filtration on rigid-analytic $\text{HC}^-$]\label{definition11HKRfiltrationonHC-solid}
		The {\it HKR filtration on rigid-analytic negative cyclic homology} of qcqs derived schemes is the functor
		$$\text{Fil}^\star_{\text{HKR}} \text{HH}(-;\mathbb{A}_f)^{h\text{S}^1} : \text{dSch}^{\text{qcqs,op}} \longrightarrow \text{FilSp}$$
		defined by the cocartesian square\footnote{Given this definition, it is a priori unclear that this filtration is multiplicative. Although we will not need this, one can prove that this filtration indeed comes equipped with a natural multiplicative structure by using the solid point of view on rigid-analytic Hochschild homology (Section~\ref{subsectionrigidanalyticGoodwillietheorem}). From this perspective, the desired multiplicative structure is a consequence of the HKR filtered version of Lemma~\ref{lemmasolidHHisHHpcompleted}.}
		$$\begin{tikzcd}
			\text{Fil}^\star_{\text{HKR}} \text{HC}^-(-) \ar[d] \ar[r] & \text{Fil}^\star_{\text{HKR}} \text{HC}^-(-_{\Q}/\Q) \ar[d] \\
			\prod_{p \in \mathbb{P}} \text{Fil}^\star_{\text{HKR}} \text{HC}^-(-;\Z_p) \ar[r] & \text{Fil}^\star_{\text{HKR}} \text{HH}(-;\mathbb{A}_f)^{h\text{S}^1}.
		\end{tikzcd}$$
	\end{definition}
	
	\begin{remark}\label{remark12HKRfiltrationonHC-solid}
		Let $p$ be a prime number, and $X$ be a qcqs derived scheme over $\Z_{(p)}$. The natural map
		$$\prod_{\ell \in \mathbb{P}} \text{Fil}^\star_{\text{HKR}} \text{HC}^-(X;\Z_{\ell}) \longrightarrow \text{Fil}^\star_{\text{HKR}} \text{HC}^-(X;\Z_p)$$
		is then an equivalence of filtered spectra, and we then denote by $\text{Fil}^\star_{\text{HKR}} \text{HH}(X;\Q_p)^{h\text{S}^1}$ the filtered spectrum $\text{Fil}^\star_{\text{HKR}} \text{HH}(X;\mathbb{A}_f)^{h\text{S}^1}$. In particular, the natural commutative diagram
		$$\begin{tikzcd}
			\text{Fil}^\star_{\text{HKR}} \text{HC}^-(X) \arrow{r} \arrow{d} & \text{Fil}^\star_{\text{HKR}} \text{HC}^-(X_{\Q}/\Q) \ar[d] \\
			\text{Fil}^\star_{\text{HKR}} \text{HC}^-(X;\Z_p) \arrow{r} & \text{Fil}^\star_{\text{HKR}} \text{HH}(X;\Q_p)^{h\text{S}^1}
		\end{tikzcd}$$
		is a cartesian square of filtered spectra.
	\end{remark}
	
	Similarly, one can apply homotopy orbits $(-)_{h\text{S}^1}$ for the $\text{S}^1$-action to the arithmetic fracture square for Hochschild homology (\ref{equationarithmeticfracturesquareforHH}). The functor $(-)_{h\text{S}^1}$ does not commute with limits, but the $\text{S}^1$-action on the product $\prod_{p \in \mathbb{P}} \text{HH}(X;\Z_p)$ is diagonal. Using the natural fibre sequence (see for instance \cite[Theorem~2.2.1]{loday_cyclic_1998} or \cite[Lemma~2.11\,(2)]{clausen_k-theory_2021})
	$$\text{HH}(X;\Z_p) \longrightarrow \text{HH}(X;\Z_p)_{h\text{S}^1} \longrightarrow \text{HH}(X;\Z_p)_{h\text{S}^1}[2]$$
	in the derived category $\mathcal{D}(\Z)$ and the fact that the functors $\text{HH}(-;\Z_p)$ and $\text{HH}(-;\Z_p)_{h\text{S}^1}$ are in non-positive cohomological degrees on animated commutative rings, one can prove that the complex \hbox{$\text{HH}(X;\Z_p)_{h\text{S}^1} \in \mathcal{D}(\Z)$} is derived $p$-complete, hence the natural map
	$$\text{HC}(X;\Z_p) \longrightarrow \text{HH}(X;\Z_p)_{h\text{S}^1}$$
	is an equivalence in the derived category $\mathcal{D}(\Z)$. In particular, applying homotopy orbits $(-)_{h\text{S}^1}$ to the arithmetic fracture square for Hochschild homology (\ref{equationarithmeticfracturesquareforHH}) induces a natural cartesian square
	$$\begin{tikzcd}
		\text{HC}(X) \ar[r] \ar[d] & \text{HC}(X_{\Q}/\Q) \ar[d] \\
		\prod_{p \in \mathbb{P}} \text{HC}(X;\Z_p) \ar[r] & \text{HH}(X;\mathbb{A}_f)_{h\text{S}^1}
	\end{tikzcd}$$
	in the derived category $\mathcal{D}(\Z)$. We use this cartesian square to introduce the following HKR filtration on the bottom right corner.
	
	\begin{definition}[HKR filtration on rigid-analytic HC]\label{definition13HKRfiltrationonHCsolid} The {\it HKR filtration on rigid-analytic cyclic homology} of qcqs derived schemes is the functor
		$$\text{Fil}^\star_{\text{HKR}} \text{HH}(-;\mathbb{A}_f)_{h\text{S}^1} : \text{dSch}^{\text{qcqs,op}} \longrightarrow \text{FilSp}$$
		defined by the cocartesian square
		$$\begin{tikzcd}
			\text{Fil}^\star_{\text{HKR}} \text{HC}(-) \ar[d] \ar[r] & \text{Fil}^\star_{\text{HKR}} \text{HC}(-_{\Q}/\Q) \ar[d] \\
			\prod_{p \in \mathbb{P}} \text{Fil}^\star_{\text{HKR}} \text{HC}(-;\Z_p) \ar[r] & \text{Fil}^\star_{\text{HKR}} \text{HH}(-;\mathbb{A}_f)_{h\text{S}^1}.
		\end{tikzcd}$$
	\end{definition}
	
	\begin{remark}\label{remark15'HKRfiltrationonHCsolidallprimesprestrictedproduct}
		Taking homotopy orbits $(-)_{h\text{S}^1}$ commutes with colimits, so the natural map
		$$\text{Fil}^\star_{\text{HKR}} \text{HC}(X;\Q) \longrightarrow \text{Fil}^\star_{\text{HKR}} \text{HC}(X_{\Q}/\Q)$$
		is an equivalence in the filtered derived category $\mathcal{DF}(\Q)$. Upon applying rationalisation to the natural cartesian square
		$$\begin{tikzcd}
			\text{Fil}^\star_{\text{HKR}} \text{HC}(X) \arrow{r} \arrow{d} & \text{Fil}^\star_{\text{HKR}} \text{HC}(X_{\Q}/\Q) \ar[d] \\
			\prod_{p \in \mathbb{P}} \text{Fil}^\star_{\text{HKR}} \text{HC}(X;\Z_p) \arrow{r} & \text{Fil}^\star_{\text{HKR}} \text{HH}(X;\mathbb{A}_f)_{h\text{S}^1},
		\end{tikzcd}$$
		this implies that the natural map
		$$\text{Fil}^\star_{\text{HKR}} \text{HC}(X;\mathbb{A}_f) := \big(\prod_{p \in \mathbb{P}} \text{Fil}^\star_{\text{HKR}} \text{HC}(X;\Z_p)\big)_{\Q} \longrightarrow \text{Fil}^\star_{\text{HKR}} \text{HH}(X;\mathbb{A}_f)_{h\text{S}^1}$$
		is an equivalence in the filtered derived category $\mathcal{DF}(\Q)$.
	\end{remark}
	
	\begin{remark}\label{remark15''HKRfiltrationonHCsolid}
		Let $p$ be a prime number, and $X$ be a qcqs derived scheme over $\Z_{(p)}$. As in Remark~\ref{remark12HKRfiltrationonHC-solid}, we denote the filtered complex $\text{Fil}^\star_{\text{HKR}} \text{HH}(X;\mathbb{A}_f)_{h\text{S}^1}$ by $$\text{Fil}^\star_{\text{HKR}} \text{HH}(X;\Q_p)_{h\text{S}^1} \in \mathcal{DF}(\Q).$$ In particular, the natural map
		$$\text{Fil}^\star_{\text{HKR}} \text{HC}(X;\Q_p) \longrightarrow \text{Fil}^\star_{\text{HKR}} \text{HH}(X;\Q_p)_{h\text{S}^1}$$
		is an equivalence in the filtered derived category $\mathcal{DF}(\Q)$ by Remark~\ref{remark15'HKRfiltrationonHCsolidallprimesprestrictedproduct}.
	\end{remark}
	
	The Tate construction $(-)^{t\text{S}^1}$ is by definition the cofibre of the norm map $$(-)_{h\text{S}^1}[1] \rightarrow (-)^{h\text{S}^1}.$$ Applying the Tate construction $(-)^{t\text{S}^1}$ to the arithmetic fracture square for Hochschild homology (\ref{equationarithmeticfracturesquareforHH}) then induces a natural cartesian square
	$$\begin{tikzcd}
		\text{HP}(X) \ar[r] \ar[d] & \text{HP}(X_{\Q}/\Q) \ar[d] \\
		\prod_{p \in \mathbb{P}} \text{HP}(X;\Z_p) \ar[r] & \text{HH}(X;\mathbb{A}_f)^{t\text{S}^1}
	\end{tikzcd}$$
	in the derived category $\mathcal{D}(\Z)$, where we use the analogous cartesian squares for $\text{HC}^-$ and $\text{HC}$ to identify the bottom left corner.
	
	\begin{definition}[HKR filtration on rigid-analytic HP]\label{definition13HKRfiltrationonHPsolid} The {\it HKR filtration on rigid-analytic periodic cyclic homology} of qcqs derived schemes is the functor
		$$\text{Fil}^\star_{\text{HKR}} \text{HH}(-;\mathbb{A}_f)^{t\text{S}^1} : \text{dSch}^{\text{qcqs,op}} \longrightarrow \text{FilSp}$$
		defined by the cocartesian square
		$$\begin{tikzcd}
			\text{Fil}^\star_{\text{HKR}} \text{HP}(-) \ar[d] \ar[r] & \text{Fil}^\star_{\text{HKR}} \text{HP}(-_{\Q}/\Q) \ar[d] \\
			\prod_{p \in \mathbb{P}} \text{Fil}^\star_{\text{HKR}} \text{HP}(-;\Z_p) \ar[r] & \text{Fil}^\star_{\text{HKR}} \text{HH}(-;\mathbb{A}_f)^{t\text{S}^1}.
		\end{tikzcd}$$
	\end{definition}

    Note here that the previous HKR filtrations on rigid-analytic $\text{HC}^-$, $\text{HC}$, and $\text{HP}$ could equivalently have been formulated using the Tate filtration approach taken in Section~\ref{subsectionHKRfiltrations} and \cite[Section~6]{bhatt_absolute_2022}. 
	
	\begin{remark}\label{remark14HKRfiltrationonHPsolid}
		Let $p$ be a prime number, and $X$ be a qcqs derived scheme over $\Z_{(p)}$. As in Remarks~\ref{remark12HKRfiltrationonHC-solid} and \ref{remark15''HKRfiltrationonHCsolid}, we denote the filtered complex $\text{Fil}^\star_{\text{HKR}} \text{HH}(X;\mathbb{A}_f)^{t\text{S}^1}$ by $$\text{Fil}^\star_{\text{HKR}} \text{HH}(X;\Q_p)^{t\text{S}^1} \in \mathcal{DF}(\Q).$$ In particular, the natural commutative diagram
		$$\begin{tikzcd}
			\text{Fil}^\star_{\text{HKR}} \text{HP}(X) \arrow{r} \arrow{d} & \text{Fil}^\star_{\text{HKR}} \text{HP}(X_{\Q}/\Q) \ar[d] \\
			\text{Fil}^\star_{\text{HKR}} \text{HP}(X;\Z_p) \arrow{r} & \text{Fil}^\star_{\text{HKR}}\text{HH}(X;\Q_p)^{t\text{S}^1}
		\end{tikzcd}$$
		is a cartesian square of filtered spectra.
	\end{remark}
	
	\begin{lemma}\label{lemma9reductionfromHC-HPtoHC-HPsolid}
		Let $X$ be a qcqs derived scheme. Then the natural commutative diagram
		$$\begin{tikzcd}
			\emph{Fil}^\star_{\emph{HKR}} \emph{HC}^-(X;\mathbb{A}_f) \arrow{r} \arrow{d} & \emph{Fil}^\star_{\emph{HKR}} \emph{HH}(X;\mathbb{A}_f)^{h\emph{S}^1} \ar[d] \\
			\emph{Fil}^\star_{\emph{HKR}} \emph{HP}(X;\mathbb{A}_f) \arrow{r} & \emph{Fil}^\star_{\emph{HKR}} \emph{HH}(X;\mathbb{A}_f)^{t\emph{S}^1}
		\end{tikzcd}$$
		is a cartesian square of filtered spectra.
	\end{lemma}
	
	\begin{proof}
		There is a natural commutative diagram of filtered spectra
		$$\hspace*{-.5cm}\begin{tikzcd}[sep=tiny]
			\text{Fil}^\star_{\text{HKR}} \text{HC}^-(X;\mathbb{A}_f) \arrow{r} \arrow{d} & \text{Fil}^\star_{\text{HKR}} \text{HP}(X;\mathbb{A}_f) \ar[d] \ar[r] & \text{Fil}^{\star}_{\text{HKR}} \text{HC}(X;\mathbb{A}_f)[2] \arrow{d} \\
			\text{Fil}^\star_{\text{HKR}}\text{HH}(X;\mathbb{A}_f)^{h\text{S}^1} \arrow{r} & \text{Fil}^\star_{\text{HKR}} \text{HH}(X;\mathbb{A}_f)^{t\text{S}^1} \ar[r] & \text{Fil}^{\star}_{\text{HKR}} \text{HH}(X;\mathbb{A}_f)_{h\text{S}^1}[2]
		\end{tikzcd}$$
		where, by definition of the bottom terms (Definitions~\ref{definition11HKRfiltrationonHC-solid}, \ref{definition13HKRfiltrationonHPsolid}, and \ref{definition13HKRfiltrationonHCsolid}), the horizontal lines are fibre sequences. In this diagram, the right vertical map is an equivalence (Remark~\ref{remark15'HKRfiltrationonHCsolidallprimesprestrictedproduct}), so the left square is a cartesian square.
	\end{proof}
	
	\begin{lemma}\label{lemmaHKRfiltrationHC-solidproductallprimesiscomplete}
		Let $X$ be a qcqs derived scheme. Then the filtrations $$\emph{Fil}^\star_{\emph{HKR}} \emph{HH}(X;\mathbb{A}_f)^{h\emph{S}^1}, \text{ } \emph{Fil}^\star_{\emph{HKR}} \emph{HH}(X;\mathbb{A}_f)_{h\emph{S}^1}, \text{ and } \emph{Fil}^\star_{\emph{HKR}} \emph{HH}(X;\mathbb{A}_f)^{t\emph{S}^1}$$ are complete.
	\end{lemma}
	
	\begin{proof}
		The HKR filtrations on $\text{HC}^-(X)$, $\text{HC}(X)$, and $\text{HP}(X)$ are complete by Lemma~\ref{lemmaHKRfiltrationonHC-isalwayscomplete}. The product on prime numbers $p$ of the $p$-completions of these filtrations are also complete, since $p$-completions and products commute with limits. Similarly, the HKR filtrations on $\text{HC}^-(X_{\Q}/\Q)$, $\text{HC}(X_{\Q}/\Q)$, and $\text{HP}(X_{\Q}/\Q)$ are complete by \cite[Theorem~$1.1$]{antieau_periodic_2019}. The rigid-analytic HKR filtrations on $\text{HC}^-$, $\text{HC}$, and $\text{HP}$ are thus also complete, as pushouts of three complete filtrations.
	\end{proof}

    Note here that the previous rigid-analytic HKR filtrations are not exhaustive on general qcqs derived schemes (see also Remark~\ref{remarkexhaustivityoftheHKRfiltrations}). We will however only need the fact that the rigid-analytic HKR filtration $\text{Fil}^\star_{\text{HKR}} \text{HH}(X;\mathbb{A}_f)_{h\text{S}^1}$ is $\N$-indexed, hence exhaustive, to prove that the motivic filtration on algebraic $K$-theory is exhaustive (Theorem~\ref{propositionmotivicfiltrationisexhaustive}). 	
	
	We now describe the graded pieces of these rigid-analytic HKR filtrations, by analogy with the classical HKR filtrations.
	
	\begin{definition}[Rigid-analytic Hodge-completed derived de Rham cohomology]\label{definition16rigidanalyticHodgecompletederiveddeRhamHC-}
		For every integer $i \in \Z$, the functor
		$$R\Gamma_{\text{Zar}}\big(-,\widehat{\underline{\mathbb{L}\Omega}}^{\geq i}_{-_{\mathbb{A}_f}/\mathbb{A}_f}\big) : \text{dSch}^{\text{qcqs,op}} \longrightarrow \mathcal{D}(\Q)$$
		is defined as the shifted graded piece of the HKR filtration on $\text{HH}(-;\mathbb{A}_f)^{h\text{S}^1}$:
		$$R\Gamma_{\text{Zar}}\big(-,\widehat{\underline{\mathbb{L}\Omega}}^{\geq i}_{-_{\mathbb{A}_f}/\mathbb{A}_f}\big) := \text{gr}^i_{\text{HKR}} \text{HH}(-;\mathbb{A}_f)^{h\text{S}^1}[-2i].$$
	\end{definition}
	
	\begin{remark}\label{remark18cartesiansquaredefiningsolidderiveddeRhamcohomology}
		Let $X$ be a qcqs derived scheme, and $i \in \Z$ be an integer. By Definition~\ref{definition11HKRfiltrationonHC-solid}, there is a natural cartesian square
		$$\begin{tikzcd}
			R\Gamma_{\text{Zar}}\big(X,\widehat{\mathbb{L}\Omega}^{\geq i}_{-/\Z}\big) \arrow{r} \arrow{d} & R\Gamma_{\text{Zar}}\big(X,\widehat{\mathbb{L}\Omega}^{\geq i}_{-_{\Q}/\Q}\big) \ar[d] \\
			R\Gamma_{\text{Zar}}\big(X,\prod_{p \in \mathbb{P}} \big(\widehat{\mathbb{L}\Omega}^{\geq i}_{-/\Z}\big)^\wedge_p\big) \arrow{r} & R\Gamma_{\text{Zar}}\big(X,\widehat{\underline{\mathbb{L}\Omega}}^{\geq i}_{-_{\mathbb{A}_f}/\mathbb{A}_f}\big)
		\end{tikzcd}$$
		in the derived category $\mathcal{D}(\Z)$, which can serve as an alternative definition for the bottom right term.
	\end{remark}
	
	\begin{remark}\label{remark17fibreseuquenceforsolidderiveddeRhamcohomology}
		Let $X$ be a qcqs derived scheme, and $i \in \Z$ be an integer. The complexes
		$$R\Gamma_{\text{Zar}}\big(X,\widehat{\underline{\mathbb{L}\Omega}}_{-_{\mathbb{A}_f}/\mathbb{A}_f}\big) \quad \text{and} \quad R\Gamma_{\text{Zar}}\big(X,\widehat{\underline{\mathbb{L}\Omega}}^{< i}_{-_{\mathbb{A}_f}/\mathbb{A}_f}\big)$$
		are defined as in Definition~\ref{definition16rigidanalyticHodgecompletederiveddeRhamHC-}, where $(-)^{h\text{S}^1}$ is replaced by $(-)^{t\text{S}^1}$ and $(-)_{h\text{S}^1}$ respectively. In particular, the natural fibre sequence
		$$\text{Fil}^\star_{\text{HKR}} \text{HH}(X;\mathbb{A}_f)^{h\text{S}^1} \rightarrow \text{Fil}^\star_{\text{HKR}} \text{HH}(X;\mathbb{A}_f)^{t\text{S}^1} \rightarrow \text{Fil}^{\star}_{\text{HKR}} \text{HH}(X;\mathbb{A}_f)_{h\text{S}^1}[2]$$
		induces a natural fibre sequence
		$$R\Gamma_{\text{Zar}}\big(X,\widehat{\underline{\mathbb{L}\Omega}}^{\geq i}_{-_{\mathbb{A}_f}/\mathbb{A}_f}\big) \longrightarrow R\Gamma_{\text{Zar}}\big(X,\widehat{\underline{\mathbb{L}\Omega}}_{-_{\mathbb{A}_f}/\mathbb{A}_f}\big) \longrightarrow R\Gamma_{\text{Zar}}\big(X,\widehat{\underline{\mathbb{L}\Omega}}^{< i}_{-_{\mathbb{A}_f}/\mathbb{A}_f}\big),$$
		where the right term is naturally identified with the complex $$R\Gamma_{\text{Zar}}\big(X,\big(L\Omega^{<i}_{-/\Z}\big)_{\mathbb{A}_f}\big) \in \mathcal{D}(\Q)$$ by Remark~\ref{remark15'HKRfiltrationonHCsolidallprimesprestrictedproduct}.
	\end{remark}
	
	\begin{remark}\label{remark17''solidderiveddeRhamcohomologyoneprimeatatime}
		Let $p$ be a prime number, $X$ be a qcqs derived scheme over $\Z_{(p)}$, and $i \in \Z$ be an integer. Following Remarks~\ref{remark12HKRfiltrationonHC-solid}, \ref{remark15''HKRfiltrationonHCsolid}, and~\ref{remark14HKRfiltrationonHPsolid}, we denote the complexes 
		$$R\Gamma_{\text{Zar}}\big(X,\widehat{\underline{\mathbb{L}\Omega}}^{\geq i}_{-_{\mathbb{A}_f}/\mathbb{A}_f}\big), \text{ } R\Gamma_{\text{Zar}}\big(X,\widehat{\underline{\mathbb{L}\Omega}}_{-_{\mathbb{A}_f}/\mathbb{A}_f}\big), \text{ and } R\Gamma_{\text{Zar}}\big(X,\widehat{\underline{\mathbb{L}\Omega}}^{< i}_{-_{\mathbb{A}_f}/\mathbb{A}_f}\big)$$
		by 
		$$R\Gamma_{\text{Zar}}\big(X,\widehat{\underline{\mathbb{L}\Omega}}^{\geq i}_{-_{\Q_p}/\Q_p}\big), \text{ } R\Gamma_{\text{Zar}}\big(X,\widehat{\underline{\mathbb{L}\Omega}}_{-_{\Q_p}/\Q_p}\big), \text{ and } R\Gamma_{\text{Zar}}\big(X,\underline{\mathbb{L}\Omega}^{< i}_{-_{\Q_p}/\Q_p}\big)$$
		respectively. In particular, by Remark~\ref{remark17fibreseuquenceforsolidderiveddeRhamcohomology}, there is a natural fibre sequence 
		$$R\Gamma_{\text{Zar}}\big(X,\widehat{\underline{\mathbb{L}\Omega}}^{\geq i}_{-_{\Q_p}/\Q_p}\big) \longrightarrow R\Gamma_{\text{Zar}}\big(X,\widehat{\underline{\mathbb{L}\Omega}}_{-_{\Q_p}/\Q_p}\big) \longrightarrow R\Gamma_{\text{Zar}}\big(X,\underline{\mathbb{L}\Omega}^{< i}_{-_{\Q_p}/\Q_p}\big)$$
		in the derived category $\mathcal{D}(\Q_p)$, where the right term is naturally identified with the complex $$R\Gamma_{\text{Zar}}\big(X,\big(\mathbb{L}\Omega^{<i}_{-/\Z}\big)^\wedge_p[\tfrac{1}{p}]\big) \in \mathcal{D}(\Q_p).$$
	\end{remark}
	
	In the following result, we reformulate the motivic filtration on topological cyclic homology of Definition~\ref{definitionmotivicfiltrationonTC} in terms of the rigid-analytic HKR filtration on negative cyclic homology. This can be interpreted as a filtered arithmetic fracture square for topological cyclic homology.
	
	\begin{proposition}\label{propositionmainconsequenceBFSwithfiltrations}
		Let $X$ be qcqs derived scheme. Then the natural commutative diagram
		$$\begin{tikzcd}
			\emph{Fil}^\star_{\emph{mot}} \emph{TC}(X) \arrow{r} \arrow{d} & \emph{Fil}^\star_{\emph{HKR}} \emph{HC}^-(X_{\Q}/\Q) \ar[d] \\
			\prod_{p \in \mathbb{P}} \emph{Fil}^\star_{\emph{BMS}} \emph{TC}(X;\Z_p) \arrow{r} & \emph{Fil}^\star_{\emph{HKR}} \emph{HH}(X;\mathbb{A}_f)^{h\emph{S}^1}
		\end{tikzcd}$$
		is a cartesian square of filtered spectra.
	\end{proposition}
	
	\begin{proof}
		This is a consequence of Definitions~\ref{definitionmotivicfiltrationonTC} and~\ref{definition11HKRfiltrationonHC-solid}.
	\end{proof}
	
	\begin{corollary}\label{corollaryusefulBFSafterLcdh}
		Let $X$ be a qcqs scheme. Then the natural commutative diagram
		$$\begin{tikzcd}
			\big(L_{\emph{cdh}}\, \emph{Fil}^\star_{\emph{mot}} \emph{TC}(-)\big)(X) \arrow{r} \arrow{d} & \big(L_{\emph{cdh}}\, \emph{Fil}^\star_{\emph{HKR}} \emph{HC}^-(-_{\Q}/\Q)\big)(X) \ar[d] \\
			\Big(L_{\emph{cdh}} \prod_{p \in \mathbb{P}} \emph{Fil}^\star_{\emph{BMS}} \emph{TC}(-;\Z_p)\Big)(X) \arrow{r} & \Big(L_{\emph{cdh}}\, \emph{Fil}^\star_{\emph{HKR}} \emph{HH}(-;\mathbb{A}_f)^{h\emph{S}^1}\Big)(X)
		\end{tikzcd}$$
		is a cartesian square of filtered spectra.
	\end{corollary}
	
	\begin{proof}
		This is a consequence of Proposition~\ref{propositionmainconsequenceBFSwithfiltrations}, which we restrict to qcqs schemes and then sheafify for the cdh topology.
	\end{proof}
	
	\begin{corollary}\label{corollarymainconsequenceBFSongradedpieces}
		Let $X$ be a qcqs derived scheme. Then for every integer $i \in \Z$, the natural commutative diagram
		$$\begin{tikzcd}
			\Z(i)^{\emph{TC}}(X) \arrow{r} \arrow{d} & R\Gamma_{\emph{Zar}}\big(X,\widehat{\mathbb{L}\Omega}^{\geq i}_{-_{\Q}/\Q}\big) \ar[d] \\
			\prod_{p \in \mathbb{P}} \Z_p(i)^{\emph{BMS}}(X) \arrow{r} & R\Gamma_{\emph{Zar}}\big(X,\widehat{\underline{\mathbb{L}\Omega}}^{\geq i}_{-_{\mathbb{A}_f}/\mathbb{A}_f}\big)
		\end{tikzcd}$$
		is a cartesian square in the derived category $\mathcal{D}(\Z)$.
	\end{corollary}
	
	\begin{proof}
		This is a direct consequence of Proposition~\ref{propositionmainconsequenceBFSwithfiltrations}.
	\end{proof}
	
	\medskip
	
	\subsection{A rigid-analytic Goodwillie theorem}\label{subsectionrigidanalyticGoodwillietheorem}
	
	\vspace{-\parindent}
	\hspace{\parindent}
	
In this subsection, we prove that the rigid-analytic version of periodic cyclic homology is a truncating invariant (Theorem~\ref{theoremtruncatinginvariantHPsolid}). We first recall the definition of Hochschild, cyclic, negative cyclic, and periodic cyclic homologies of a general cyclic object.

\begin{definition}[Cyclic object]\label{definitioncyclicobject}
	Let $\mathcal{C}$ be a category, or an $\infty$-category. The {\it cyclic category} $\Lambda$ is the category with objects $[n]$ indexed by non-negative integers $n$, and morphisms $[m] \rightarrow [n]$ given by homotopy classes of degree one increasing maps from $\text{S}^1$ to itself that map the subgroup $\Z/(m+1)$ to $\Z/(n+1)$. A {\it cyclic object} in $\mathcal{C}$ is then a contravariant functor from the cyclic category $\Lambda$ to $\mathcal{C}$.
\end{definition}

\begin{notation}[Hochschild homology of a cyclic object]\label{notationmapSforcyclicobjects}
	Let $\mathcal{A}$ be an abelian category with exact infinite products, and~$X_{\bullet}$ be a cyclic object in $\mathcal{A}$. Following \cite[Section~II]{goodwillie_cyclic_1985} (see also \cite[Section~$2.2$]{morrow_aws_2018}), one can define the Hochschild homology $\text{HH}(X_{\bullet})$ of $X_{\bullet}$, as an object of the derived category $\mathcal{D}(\mathcal{A})$ equipped with a natural $\text{S}^1$-action. The cyclic, negative cyclic, and periodic cyclic homologies of the cyclic object $X_{\bullet}$ are then defined by
	$$\text{HC}(X_{\bullet}) := \text{HH}(X_{\bullet})_{h\text{S}^1}, \text{ } \text{HC}^-(X_{\bullet}) := \text{HH}(X_{\bullet})^{h\text{S}^1}, \text{ and } \text{HP}(X_{\bullet}) := \text{HH}(X_{\bullet})^{t\text{S}^1}.$$
	In this context, there is moreover a natural map
	$$s : \text{HC}(X_\bullet)[-2] \longrightarrow \text{HC}(X_\bullet)$$
	in the derived category $\mathcal{D}(\mathcal{A})$, from which periodic cyclic homology $\text{HP}(X_\bullet) \in \mathcal{D}(\mathcal{A})$ can be recovered by the formula
	$$\text{HP}(X_\bullet) \simeq \lim\limits_{\longleftarrow} \Big(\cdots \xlongrightarrow{s} \text{HC}(X_\bullet)[-4] \xlongrightarrow{s} \text{HC}(X_\bullet)[-2] \xlongrightarrow{s} \text{HC}(X_\bullet)\Big).$$
\end{notation}

We now apply the previous general construction to define Hochschild homology and its variants on solid associative derived algebras over a solid commutative ring.

The notion of solid objects that we use here is the one developed by Clausen--Scholze in their recent framework of {\it condensed mathematics} \cite{clausen_condensed_2019,clausen_analytic_2020,clausen_condensed_2022}. Recall that condensed mathematics roughly aims at developing a theory of analytic geometry which is as robust as the theory of schemes in algebraic geometry, and which in particular encompasses the existing theories of complex manifolds, of rigid-analytic spaces, of adic spaces, and of Berkovich spaces. One of the major achievements of condensed mathematics, together with Efimov's work on dualisable categories, is the development of a good notion of $K$-theory for general analytics spaces \cite{efimov_K-theory_2024}.

The solid theory is the subfield of condensed mathematics concerned with {\it non-archimedean} analytic spaces. We refer to \cite{clausen_condensed_2019} for the original reference on the condensed and solid theories, and to the references therein for more recent developments.

\begin{definition}[Solid Hochschild homology]\label{definitionHHofsolid}
	Let $k$ be a solid commutative ring, and $R$ be a connective solid $k$-$\E_1$-algebra. The simplicial object
	$$\cdots \hspace{1mm} \substack{\longrightarrow\\[-0.9em] \longrightarrow\\[-0.9em] \longrightarrow\\[-0.9em] \longrightarrow} \hspace{1mm} R \otimes_{k} R \otimes_{k} R \hspace{1mm} \substack{\longrightarrow\\[-0.9em] \longrightarrow\\[-0.9em] \longrightarrow} \hspace{1mm} R\otimes_{k} R \hspace{1mm} \substack{\longrightarrow\\[-0.9em] \longrightarrow} \hspace{1mm} R$$
	has a natural structure of cyclic object in the derived $\infty$-category of solid $k$-modules (Definition~\ref{definitioncyclicobject}), induced by permutation of the tensor factors. We write $\underline{\text{HH}}(R/k)$, $\underline{\text{HC}}(R/k)$, $\underline{\text{HC}^-}(R/k)$ and $\underline{\text{HP}}(R/k)$ for the Hochschild, cyclic, negative cyclic and periodic cyclic homologies of this cyclic object (Notation~\ref{notationmapSforcyclicobjects}). If $k=\Z$, we simply denote these by $\underline{\text{HH}}(R)$, $\underline{\text{HC}}(R)$, $\underline{\text{HC}^-}(R)$ and $\underline{\text{HP}}(R)$.
\end{definition}

For $f : R \rightarrow R'$ a map of $\Z$-$\E_1$-algebras and $F : \E_1\text{-Alg}_{\Z} \rightarrow \mathcal{D}(\Z)$ a functor, we denote by $F(f) \in \mathcal{D}(\Z)$ the cofibre of the map $F(R) \rightarrow F(R')$. More generally, for $f : R \rightarrow R'$ a map of solid $\Z$-$\mathbb{E}_1$-algebras and $F$ a $\mathcal{D}(\text{Solid})$-valued functor on solid $\Z$-$\mathbb{E}_1$-algebras, denote by $F(f) \in \mathcal{D}(\text{Solid})$ the cofibre of the map $F(R) \rightarrow F(R')$. 

The following result is \cite[Theorem~IV.$2.6$]{goodwillie_cyclic_1985}. More precisely, this result of Goodwillie is for maps of simplicial $\Z$-algebras, and the underlying $\infty$\nobreakdash-category of simplicial $\Z$-algebras is naturally identified with the $\infty$\nobreakdash-category of connective $\Z$-$\mathbb{E}_1$-algebras, by the monoidal Dold--Kan correspondence and \cite[Proposition~$7.1.4.6$]{lurie_higher_2017}.

\begin{theorem}[\cite{goodwillie_cyclic_1985}]\label{theoremGoodwillieIV26}
	Let $f : R \rightarrow R'$ be a $1$-connected map of connective $\Z$-$\E_1$-algebras. For every integer $n \geq 0$, the natural map
	$$n!s^n : \emph{HC}_{\ast+2n}(f) \longrightarrow \emph{HC}_\ast(f)$$
	is the zero map for $\ast \leq n-1$.
\end{theorem}

Goodwillie's proof of Theorem~\ref{theoremGoodwillieIV26}, although stated with respect to the abelian category of $\Z$\nobreakdash-modules (or $k$-modules, for $k$ an arbitrary discrete commutative ring), is valid for any abelian symmetric monoidal category with exact infinite products. More precisely, the argument of Goodwillie is a rather long but elementary reduction to a more concrete computation, which corresponds to the proof of \cite[Theorem~II.5.1]{goodwillie_cyclic_1985}. The crucial properties used by Goodwillie about Hochschild homology and its variants are those derived from standard properties of the tensor product and from the notion of Hochschild homology of a cyclic object in a symmetric monoidal category (Notation~\ref{notationmapSforcyclicobjects}). These arguments then adapt readily to solid Hochschild homology (Definition~\ref{definitionHHofsolid}), and one can prove the following generalisation of Goodwillie's result.

\begin{theorem}\label{theoremGoodwillieIV26solid}
	Let $f : R \rightarrow R'$ be a $1$-connected map of connective solid $\Z$-$\E_1$-algebras. Then for every integer $n \geq 0$, the natural map
	$$n!s^n : \underline{\emph{HC}}(f)[-2n] \longrightarrow \underline{\emph{HC}}(f),$$
	in the derived category $\mathcal{D}(\emph{Solid})$,
	is the zero map on cohomology groups\footnote{By this, we mean that it is the zero map as a map of solid abelian groups, {\it i.e.}, that it factors through the zero object of the abelian category of solid abelian groups. Note that the underlying abelian group of a nonzero solid abelian group can be zero ({\it e.g.}, the quotient of $\Z_p$ with the $p$-adic topology by $\Z_p$ with the discrete topology). In particular, being zero for a map of solid abelian groups cannot be detected on the underlying map of abelian groups.} in degrees at least $-n+1$.
\end{theorem}

\begin{proof}
	To prove the result for simplicial solid $\Z$-algebras, it suffices to prove that the abelian category of solid abelian groups is symmetric monoidal, and has exact infinite products. The first claim is \cite[Theorem~$6.2$\,(i)]{clausen_condensed_2019}. The second claim is a consequence of the fact that the abelian category of condensed abelian groups has exact arbitrary products (\cite[Theorem~$1.10$]{clausen_condensed_2019}), and that the category of solid abelian groups, as a subcategory of the abelian category of condensed abelian groups, is stable under all limits (\cite[Theorem~$5.8\,(i)$]{clausen_condensed_2019}). We omit the proof of the analogue of \cite[Proposition~$7.1.4.6$]{lurie_higher_2017} to pass from simplicial solid $\Z$-algebras to connective solid $\Z$-$\mathbb{E}_1$-algebras.
\end{proof}

For the rest of this section, and following the convention of condensed mathematics, a condensed object is called {\it discrete} if its condensed structure is trivial. Given a classical object $X$, we denote by $\underline{X}$ the associated discrete condensed object.

\begin{lemma}\label{lemmasolidHHisHHpcompleted}
	Let $R$ be a connective $\Z$-$\mathbb{E}_1$-algebra, and $p$ be a prime number. Then the natural map
	$$\underline{\emph{HH}(R)} \longrightarrow \underline{\emph{HH}}(\underline{R}^\wedge_p),$$
	in the derived category $\mathcal{D}(\emph{Solid})$, exhibits the target as the $p$-completion of the source. In particular, there is a natural equivalence
	$$\underline{\emph{HH}(R)}^\wedge_p[\tfrac{1}{p}] \xlongrightarrow{\sim} \underline{\emph{HH}}(\underline{R}^\wedge_p[\tfrac{1}{p}]/\Q_p)$$
	in the derived category $\mathcal{D}(\emph{Solid})$.
\end{lemma}

\begin{proof}
	The solid tensor product of $p$-complete solid connective $\Z$-$\mathbb{E}_1$-algebras is $p$-complete (\cite[Scholze's comment, Point 2]{clausen_witt_2021}), and so is the totalisation of a complex of $p$-complete solid connective $\Z$-modules (because totalisation and $p$-completion commute). In particular, the complex $\underline{\text{HH}}(\underline{R}^\wedge_p)$, as a totalisation of tensor powers of the $p$-complete solid connective $\Z$-$\mathbb{E}_1$-algebra~$\underline{R}^\wedge_p$, is $p$-complete. By the derived Nakayama lemma (\cite[091N]{stacks_project_authors_stacks_2019}), it thus suffices to prove the first statement modulo $p$. By base change for Hochschild homology (resp. solid Hochschild homology), this is equivalent to proving that the natural map
	$$\underline{\text{HH}((R/p)/\F_p)} \longrightarrow \underline{\text{HH}}((\underline{R}/p)/\underline{\F_p})$$
	is an equivalence in the derived category $\mathcal{D}(\text{Solid})$. The desired equivalence is then a consequence of the fact that reduction modulo $p$ and tensor products commute with the functor~$\underline{(-)}$ from $\Z$-$\mathbb{E}_1$-algebras to solid $\Z$-$\mathbb{E}_1$-algebras. The second statement follows from the fact that rationalisation commutes with the functor $$\underline{(-)} : \mathcal{D}(\Z) \longrightarrow \mathcal{D}(\text{Solid}),$$ and base change for solid Hochschild homology.
\end{proof}

\begin{proposition}\label{propositionsolidHCisHCpcompleted}
	Let $R$ be a connective $\Z$-$\mathbb{E}_1$-algebra, and $p$ be a prime number. Then the natural map
	$$\underline{\emph{HC}(R)} \longrightarrow \underline{\emph{HC}}(\underline{R}^\wedge_p),$$
	in the derived category $\mathcal{D}(\emph{Solid})$, exhibits the target as the $p$-completion of the source. In particular, there is a natural equivalence
	$$\underline{\emph{HC}(R)}^\wedge_p[\tfrac{1}{p}] \xlongrightarrow{\sim} \underline{\emph{HC}}(\underline{R}^\wedge_p[\tfrac{1}{p}]/\Q_p)$$
	in the derived category $\mathcal{D}(\emph{Solid})$.
\end{proposition}

\begin{proof}
	The first statement is a consequence of Lemma~\ref{lemmasolidHHisHHpcompleted}, and the description \cite[Definition~$2.19$]{morrow_aws_2018} of (solid) cyclic homology in terms of the cyclic object $\text{HH}(R)$ (resp. $\underline{\text{HH}}(\underline{R}^\wedge_p)$) in the stable $\infty$-category $\mathcal{D}(\Z)$ (resp. $\mathcal{D}(\text{Solid})$). The second statement follows from the fact that rationalisation commutes direct sums (or equivalently, with the functor $(-)_{h\text{S}^1}$) and with the functor $\underline{(-)}$ from the derived category $\mathcal{D}(\Z)$ to the derived category $\mathcal{D}(\text{Solid})$.
\end{proof}

\begin{corollary}\label{corollaryadaptationoftheoremGoodwillieIV26}
	Let $f : R \rightarrow R'$ a $1$-connected map of connective $\Z$-$\E_1$-algebras. Then for every integer $n \geq 0$, the map
	$$n! s^n : \emph{HC}(f;\mathbb{A}_f)[-2n] \longrightarrow \emph{HC}(f;\mathbb{A}_f),$$
	in the derived category $\mathcal{D}(\Q)$, is the zero map on cohomology groups in degrees at least $-n+1$.
\end{corollary}

\begin{proof}
	For every prime number $p$, the natural map
	$$n! s^n : \underline{\text{HC}}(\underline{f}^\wedge_p)[-2n] \longrightarrow \underline{\text{HC}}(\underline{f}^\wedge_p),$$
	in the derived category $\mathcal{D}(\text{Solid})$, is zero in cohomological degrees at least $-n+1$ (Theorem~\ref{theoremGoodwillieIV26solid}). By Proposition~\ref{propositionsolidHCisHCpcompleted}, this is equivalent to the fact that the natural map
    $$n!s^n : \underline{\text{HC}(f)}^\wedge_p[-2n] \longrightarrow \underline{\text{HC}(f)}^\wedge_p,$$
    in the derived category $\mathcal{D}(\text{Solid})$, is zero in cohomological degrees at least $-n+1$. Forgetting the condensed structure, this implies that the natural map
    $$n!s^n : \text{HC}(f;\Z_p)[-2n] \longrightarrow \text{HC}(f;\Z_p),$$
    in the derived category $\mathcal{D}(\Z)$, is zero in cohomological degrees at least $-n+1$. Taking the product over all primes $p$ and the rationalisation then implies the desired result.
\end{proof}

\begin{theorem}[Rigid-analytic HP is truncating]\label{theoremtruncatinginvariantHPsolid}
	The construction 
	$$R \longmapsto \emph{HH}(R;\mathbb{A}_f)^{t\emph{S}^1} := \Big(\Big(\prod_{p \in \mathbb{P}} \emph{HH}(R;\Z_p)\Big)_{\Q}\Big)^{t\emph{S}^1},$$
	from connective $\Z$-$\E_1$-algebras $R$ to the derived category $\mathcal{D}(\Q)$, is truncating. More precisely, there exists a truncating invariant \hbox{$E : \emph{Cat}^{\emph{perf}}_{\infty} \rightarrow \mathcal{D}(\Q)$} such that the previous construction is the composite $R \mapsto \emph{Perf}(R) \mapsto E(\emph{Perf}(R))$.
\end{theorem}

\begin{proof}
	Let $f : R \rightarrow R'$ be a $1$-connected map of connective $\Z$-$\E_1$-algebras. We want to prove that the natural map
	$$\text{HH}(R;\mathbb{A}_f)^{t\text{S}^1} \longrightarrow \text{HH}(R';\mathbb{A}_f)^{t\text{S}^1}$$
	is an equivalence, or equivalently that its homotopy cofibre $\text{HH}(f;\mathbb{A}_f)^{t\text{S}^1}$ vanishes in the derived category $\mathcal{D}(\Q)$. 
	For every integer $n \geq 0$, the natural map
	$$n!s^n : \text{HC}(f;\mathbb{A}_f)[-2n] \longrightarrow \text{HC}(f;\mathbb{A}_f)$$
	is the zero map in cohomological degrees at least $-n+1$ (Corollary~\ref{corollaryadaptationoftheoremGoodwillieIV26}). This map is $\Q$-linear, so the map
	$$s^n : \text{HC}(f;\mathbb{A}_f)[-2n] \longrightarrow \text{HC}(f;\mathbb{A}_f)$$
	is also the zero map in cohomological degrees at least $-n+1$. Taking the inverse limit over integers $n \geq 0$ and using the equivalence at the end of Notation~\ref{notationmapSforcyclicobjects} then implies that the complex $\text{HH}(f;\mathbb{A}_f)^{t\text{S}^1}$ is zero in the derived category $\mathcal{D}(\Q)$.
\end{proof}

\medskip

	\subsection{Rigid-analytic derived de Rham cohomology is a cdh sheaf}
	
	\vspace{-\parindent}
	\hspace{\parindent}
	
	The aim of this section is to prove the cdh descent result Corollary~\ref{corollaryHPsolidallprimespisacdhsheafwithfiltration}, which is a rigid-analytic analogue of the following result, used in \cite{elmanto_motivic_2023} to prove Theorem~\ref{theoremrationalstructuremain} in characteristic zero.
	
	\begin{proposition}[\cite{cortinas_cyclic_2008,bals_periodic_2024}]\label{propositionfilteredHPisacdhsheafinchar0}
		For every integer $i \in \Z$, the presheaf
		$$\emph{Fil}^i_{\emph{HKR}} \emph{HP}(-_{\Q}/\Q) : \emph{Sch}^{\emph{qcqs,op}} \longrightarrow \mathcal{D}(\Q)$$
		is a cdh sheaf.
	\end{proposition}
	
	\begin{proof}
		This is a consequence of \cite[Corollary~A.$6$]{land_k-theory_2019} and \cite[Theorem~$1.3$]{bals_periodic_2024}.
	\end{proof}
	
	We first extract the following cdh descent result from Theorem~\ref{theoremtruncatinginvariantHPsolid}. Note that this argument uses the theory of truncating invariants, as developed by Land--Tamme \cite{land_k-theory_2019}, in a crucial way.
	
	\begin{corollary}\label{corollaryHPsolidallprimespisacdhsheafwithoutfiltration}
		The presheaf
		$$\emph{HH}(-;\mathbb{A}_f)^{t\emph{S}^1} : \emph{Sch}^{\emph{qcqs,op}} \longrightarrow \mathcal{D}(\Q)$$
		is a cdh sheaf.
	\end{corollary}
	
	\begin{proof}
		This is a consequence of Theorem~\ref{theoremtruncatinginvariantHPsolid} and \cite[Theorem~E]{land_k-theory_2019}.
	\end{proof}
	
	We then adapt the main splitting result of \cite{bals_periodic_2024} on periodic cyclic homology over $\Q$ to our rigid-analytic setting.
	
	\begin{proposition}\label{propositionsplittingofHPsolid}
		Let $X$ be a qcqs derived scheme. Then there exists a natural equivalence
		$$\emph{HH}(X;\mathbb{A}_f)^{t\emph{S}^1} \simeq \prod_{i \in \Z} R\Gamma_{\emph{Zar}}\big(X,\widehat{\underline{\mathbb{L}\Omega}}_{-_{\mathbb{A}_f}/\mathbb{A}_f}\big)[2i]$$
		in the derived category $\mathcal{D}(\Q)$.
	\end{proposition}
	
	\begin{proof}
		We adapt the proof of \cite[Theorem~$3.4$]{bals_periodic_2024}, which proves a similar decomposition for periodic cyclic homology over $\Q$. The crucial point to adapt this proof is to note that there is a natural equivalence
		$$\text{HH}(X;\mathbb{A}_f) \simeq \bigoplus_{i \in \N} R\Gamma_{\text{Zar}}\big(X,(\mathbb{L}^i_{-/\Z})_{\mathbb{A}_f}\big)[i]$$
		in the derived category $\mathcal{D}(\Q)$ (see \cite[Remark~$2.8$]{bals_periodic_2024}). This is the rigid-analytic analogue of \cite[Proposition~$2.7$]{bals_periodic_2024}, and it is for instance a consequence of Lemma~\ref{lemma29howtouseAdamsoperations} applied to the $\N$-indexed filtration $\text{Fil}^\star_{\text{HKR}} \text{HH}(-;\mathbb{A}_f)$. By \cite[Theorem~$3.2$]{bals_periodic_2024}, this implies that there is a natural equivalence
		$$\Big(\text{Fil}^\star_{\text{H}} R\Gamma_{\text{Zar}}\big(X,\widehat{\underline{\mathbb{L}\Omega}}_{-_{\mathbb{A}_f}/\mathbb{A}_f}\big) \otimes \text{Fil}^\star_{\text{T}} (\mathbb{A}_f)^{t\text{S}^1}\Big)^\wedge \xlongrightarrow{\sim} \text{Fil}^\star_{\text{T}} \text{HH}(X;\mathbb{A}_f)^{t\text{S}^1}$$
		in the filtered derived category $\mathcal{DF}(\Q)$, where the tensor product $\otimes$ is the Day convolution tensor product, $\text{Fil}^\star_{\text{H}}$ is the Hodge filtration on derived de Rham cohomology, and $(-)^\wedge$ is the completion with respect to the associated filtration. By a degree argument explained in \cite[proof of Theorem~$3.4$]{bals_periodic_2024}, the filtered object $\text{Fil}^\star_{\text{T}} (\mathbb{A}_f)^{t\text{S}^1}$ carries a canonical splitting, which induces an equivalence
		$$\prod_{i \in \Z} \Big(\text{Fil}^{\star+i}_{\text{H}} R\Gamma_{\text{Zar}}\big(X,\widehat{\underline{\mathbb{L}\Omega}}_{-_{\mathbb{A}_f}/\mathbb{A}_f}\big)\Big)[2i] \xlongrightarrow{\sim} \text{Fil}^\star_{\text{T}} \text{HH}(X;\mathbb{A}_f)^{t\text{S}^1}$$
		in the filtered derived category $\mathcal{DF}(\Q)$. 
		
		It then suffices to prove that the desired result is indeed obtained by taking the colimit over $\star \rightarrow -\infty$ of the previous equivalence. Following \cite[proof of Theorem~$3.4$]{bals_periodic_2024}, the result for the source is a formal consequence of the connectivity estimate for the functor
		$$R\Gamma_{\text{Zar}}\big(-,\widehat{\underline{\mathbb{L}\Omega}}^{< i}_{-_{\mathbb{A}_f}/\mathbb{A}_f}\big)$$
		on animated commutative ring, which is itself a consequence of Remark~\ref{remark15'HKRfiltrationonHCsolidallprimesprestrictedproduct}, and of the classical connectivity estimate for the HKR filtration on cyclic homology (Proposition~\ref{propositionHKRfiltrationonHCrationalisfinitary}). The result for the target is a consequence of the fact that the Tate filtration is always exhaustive (\cite[Proposition~B.$6$]{bals_periodic_2024}).
	\end{proof}
	
	\begin{corollary}\label{corollaryHPsolidallprimespisacdhsheafongradedpieces}
		The presheaf
		$$R\Gamma_{\emph{Zar}}\big(-,\widehat{\underline{\mathbb{L}\Omega}}_{-_{\mathbb{A}_f}/\mathbb{A}_f}\big) : \emph{Sch}^{\emph{qcqs,op}} \longrightarrow \mathcal{D}(\Q)$$
		is a cdh sheaf.\footnote{Note here that this cdh sheaf is however not finitary. Similarly, the cdh sheaf $R\Gamma_{\text{Zar}}(-,\widehat{\mathbb{L}\Omega}_{-_{\Q}/\Q}) : \text{Sch}^{\text{qcqs,op}} \rightarrow \mathcal{D}(\Q)$ is not finitary, because the completion for the Hodge filtration does not commute with filtered colimits. Following \cite[Construction~6.3.1]{bhatt_absolute_2022}, it would be interesting to find a suitable category of coefficients making the Hodge-completed rigid-analytic derived de Rham cohomology into a finitary sheaf.}
	\end{corollary}
	
	\begin{proof}
		This is a consequence of the natural splitting Proposition~\ref{propositionsplittingofHPsolid}, and of the cdh descent result Corollary~\ref{corollaryHPsolidallprimespisacdhsheafwithoutfiltration}.
	\end{proof}
	
	\begin{corollary}\label{corollaryHPsolidallprimespisacdhsheafwithfiltration}
		For every integer $i \in \Z$, the presheaf
		$$\emph{Fil}^i_{\emph{HKR}} \emph{HH}(-;\mathbb{A}_f)^{t\emph{S}^1} : \emph{Sch}^{\emph{qcqs,op}} \longrightarrow \mathcal{D}(\Q)$$
		is a cdh sheaf.
	\end{corollary}
	
	\begin{proof}
		The HKR filtration on $\text{HH}(-;\mathbb{A}_f)^{t\text{S}^1}$ is complete by Lemma~\ref{lemmaHKRfiltrationHC-solidproductallprimesiscomplete}, so it suffices to prove the result on graded pieces. The result is then Corollary~\ref{corollaryHPsolidallprimespisacdhsheafongradedpieces}.
	\end{proof}

	\medskip

    \subsection{The motivic filtration}\label{subsectionAHSS}
	
	\vspace{-\parindent}
	\hspace{\parindent}
	
	In this subsection, we use the results of the previous sections to prove Theorem~\ref{theoremintroAHSSandAdams}.
	
	\begin{theorem}[The motivic filtration is $\N$-indexed]\label{propositionmotivicfiltrationisexhaustive}
		Let $X$ be a qcqs derived scheme. Then for every integer $i \leq 0$, the natural map
		$$\emph{Fil}^i_{\emph{mot}} \emph{K}(X) \longrightarrow \emph{K}(X)$$
		is an equivalence of spectra. In particular, the motivic filtration $\emph{Fil}^\star_{\emph{mot}} \emph{K}(X)$ on $\emph{K}(X)$ is exhaustive.
	\end{theorem}
	
	\begin{proof}
		First assume that $X$ is a qcqs classical scheme. The filtration $\text{Fil}^\star_{\text{cdh}} \text{KH}(X)$ is $\N$-indexed by Theorem~\ref{theoremBEM}\,(1), so it suffices to prove that the filtration
		$$\text{cofib}\Big(\text{Fil}^\star_{\text{mot}} \text{TC}(X) \longrightarrow \big(L_{\text{cdh}} \text{Fil}^\star_{\text{mot}} \text{TC}\big)(X)\Big)$$
		is $\N$-indexed (Definition~\ref{definitionmotivicfiltrationonKtheoryofschemes}). To prove this, we use Proposition~\ref{propositionmainconsequenceBFSwithfiltrations} and Corollary~\ref{corollaryusefulBFSafterLcdh}. For every prime number $p$, the filtration $\text{Fil}^\star_{\text{BMS}} \text{TC}(-;\Z_p)$ is $\N$-indexed on qcqs schemes (\cite[Theorem~$7.2\,(1)$]{bhatt_topological_2019} and Theorem~\ref{theoremBMSfiltrationonTCpcompletedsatisfiesquasisyntomicdescentandisLKEfrompolynomialalgebras}\,$(2)$). In particular, the filtration
		$$\text{cofib}\Big(\prod_{p \in \mathbb{P}} \text{Fil}^\star_{\text{BMS}} \text{TC}(X;\Z_p) \longrightarrow \big(L_{\text{cdh}} \prod_{p \in \mathbb{P}} \text{Fil}^\star_{\text{BMS}} \text{TC}(-;\Z_p)\big)(X)\Big)$$
		is $\N$-indexed. The presheaf $\text{Fil}^\star_{\text{HKR}} \text{HP}(-_{\Q}/\Q)$ is a cdh sheaf on qcqs schemes (Proposition~\ref{propositionfilteredHPisacdhsheafinchar0}), and the filtration $\text{Fil}^\star_{\text{HKR}} \text{HC}(X_{\Q}/\Q)$ is $\N$-indexed by construction. In particular, the filtration
		$$\text{cofib}\Big(\text{Fil}^\star_{\text{HKR}} \text{HC}^-(X_{\Q}/\Q) \longrightarrow \big(L_{\text{cdh}} \text{Fil}^\star_{\text{HKR}} \text{HC}^-(-_{\Q}/\Q)\big)(X)\Big)$$
		is naturally identified with the filtration
		$$\text{cofib}\Big(\text{Fil}^{\star}_{\text{HKR}} \text{HC}(X_{\Q}/\Q) \longrightarrow \big(L_{\text{cdh}} \text{Fil}^{\star}_{\text{HKR}} \text{HC}(-_{\Q}/\Q)\big)(X)\Big)[1],$$ which is $\N$-indexed.
		Similarly, the presheaf $\text{Fil}^\star_{\text{HKR}} \text{HH}(-;\mathbb{A}_f)^{t\text{S}^1}$ is a cdh sheaf on qcqs schemes (Corollary~\ref{corollaryHPsolidallprimespisacdhsheafwithfiltration}), and the filtration $\text{Fil}^\star_{\text{HKR}} \text{HH}(-;\mathbb{A}_f)_{h\text{S}^1}$ is $\N$-indexed (Remark~\ref{remark15'HKRfiltrationonHCsolidallprimesprestrictedproduct}). In particular, the filtration
		$$\text{cofib}\Big(\text{Fil}^\star_{\text{HKR}} \text{HH}(X;\mathbb{A}_f)^{h\text{S}^1} \longrightarrow \big(L_{\text{cdh}} \text{Fil}^\star_{\text{HKR}} \text{HH}(-;\mathbb{A}_f)^{h\text{S}^1}\big)(X)\Big)$$
		is naturally identified with the filtration
		$$\text{cofib}\Big(\text{Fil}^{\star}_{\text{HKR}} \text{HH}(X;\mathbb{A}_f)_{h\text{S}^1} \longrightarrow \big(L_{\text{cdh}} \text{Fil}^{\star}_{\text{HKR}} \text{HH}(-;\mathbb{A}_f)_{h\text{S}^1}\big)(X)\Big)[1],$$ which is $\N$-indexed.
		
		Assume now that $X$ is a general qcqs derived scheme. By Definition~\ref{definitionmotivicfiltrationonKtheoryofderivedschemes} and the previous paragraph, it suffices to prove that the filtration
		$$\text{cofib}\Big(\text{Fil}^\star_{\text{mot}} \text{TC}(X) \longrightarrow \text{Fil}^\star_{\text{mot}} \text{TC}(X^{\text{cl}})\Big)$$
		is $\N$-indexed. To prove this, we use Proposition~\ref{propositionmainconsequenceBFSwithfiltrations}. For every prime number~$p$, the filtration $\text{Fil}^\star_{\text{BMS}} \text{TC}(X;\Z_p)$ is $\N$-indexed on polynomial $\Z$-algebras by the previous paragraph, and hence on general qcqs derived schemes by Zariski descent and Theorem~\ref{theoremBMSfiltrationonTCpcompletedsatisfiesquasisyntomicdescentandisLKEfrompolynomialalgebras}\,$(2)$. In particular, the filtration
		$$\text{cofib}\Big(\prod_{p \in \mathbb{P}} \text{Fil}^\star_{\text{BMS}} \text{TC}(X;\Z_p) \longrightarrow \prod_{p \in \mathbb{P}} \text{Fil}^\star_{\text{BMS}} \text{TC}(X^{\text{cl}};\Z_p)\Big)$$
		is $\N$-indexed. The filtration
		$$\text{cofib}\Big(\text{Fil}^\star_{\text{HKR}} \text{HC}^-(X_{\Q}/\Q) \longrightarrow \text{Fil}^\star_{\text{HKR}} \text{HC}^-((X^{\text{cl}})_{\Q}/\Q)\Big)$$
		is $\N$-indexed by \cite[Theorem~$4.48$]{elmanto_motivic_2023}. Similarly, using Theorem~\ref{theoremtruncatinginvariantHPsolid} and Proposition~\ref{propositionsplittingofHPsolid} as in the proof of Corollary~\ref{corollaryHPsolidallprimespisacdhsheafwithfiltration}, the natural map
		$$\text{Fil}^\star_{\text{HKR}} \text{HH}(X;\mathbb{A}_f)^{t\text{S}^1} \longrightarrow \text{Fil}^\star_{\text{HKR}} \text{HH}(X^{\text{cl}};\mathbb{A}_f)^{t\text{S}^1}$$
		is an equivalence of filtered spectra. In particular, the filtration
		$$\text{cofib}\Big(\text{Fil}^\star_{\text{HKR}} \text{HH}(X;\mathbb{A}_f)^{h\text{S}^1} \longrightarrow \text{Fil}^\star_{\text{HKR}} \text{HH}(X^{\text{cl}};\mathbb{A}_f)^{h\text{S}^1}\Big)$$
		is naturally identified with the filtration
		$$\text{cofib}\Big(\text{Fil}^\star_{\text{HKR}} \text{HH}(X;\mathbb{A}_f)_{h\text{S}^1} \longrightarrow \text{Fil}^\star_{\text{HKR}} \text{HH}(X^{\text{cl}};\mathbb{A}_f)_{h\text{S}^1}\Big)[1],$$
		which is $\N$-indexed by Remark~\ref{remark15'HKRfiltrationonHCsolidallprimesprestrictedproduct}.
	\end{proof}
	
	\begin{corollary}\label{corollarymotiviccomplexiszerofornegativeweight}
		Let $X$ be a qcqs scheme. Then for every integer $i<0$, the motivic complex $\Z(i)^{\emph{mot}}(X) \in \mathcal{D}(X)$ is zero.
	\end{corollary}
	
	\begin{proof}
		This is a direct consequence of Theorem~\ref{propositionmotivicfiltrationisexhaustive}.
	\end{proof}

	\begin{proposition}\label{propositionfibrefilteredTCandLcdhTCiscomplete}
		Let $d \geq 0$ be an integer, and $X$ be a qcqs scheme of valuative dimension at most~$d$. Then for every integer $i \in \Z$, the fibre of the natural map of spectra
		$$\emph{Fil}^i_{\emph{mot}} \emph{TC}(X) \longrightarrow \emph{Fil}^i_{\emph{mot}} L_{\emph{cdh}} \emph{TC}(X)$$
		is in cohomological degrees at most $-i+d+2$. In particular, the filtration formed by these spectra for all integers $i \in \Z$, and the rationalisation of this filtration, are complete.
	\end{proposition}
	
	\begin{proof}
		The last statement is a consequence of the first, as the connectivity bound for a given filtration induces the same connectivity bound for its rationalisation. For the connectivity bound, we use Proposition~\ref{propositionmainconsequenceBFSwithfiltrations} and Corollary~\ref{corollaryusefulBFSafterLcdh} to compare the spectra $\text{Fil}^i_{\text{mot}} \text{TC}(X)$ and $\text{Fil}^i_{\text{mot}} L_{\text{cdh}} \text{TC}(X)$.
		
		The presheaf $\prod_{p \in \mathbb{P}} \text{Fil}^i_{\text{BMS}} \text{TC}(X;\Z_p)$ takes values in cohomological degrees at most $-i+1$ on affine schemes (Lemma~\ref{lemmaBMSfiltrationproductallprimesiscomplete}). In particular, the fibre of the natural map of spectra
		$$\prod_{p \in \mathbb{P}} \text{Fil}^i_{\text{BMS}} \text{TC}(X;\Z_p) \longrightarrow \big(L_{\text{cdh}} \prod_{p \in \mathbb{P}} \text{Fil}^i_{\text{BMS}} \text{TC}(-;\Z_p)\big)(X)$$
		is in cohomological degrees at most $-i+d+2$, as each term is in cohomological degrees at most $-i+d+1$ (\cite[Theorem~$3.12$]{clausen_hyperdescent_2021} and \cite[Theorem~$2.4.15$]{elmanto_cdh_2021}). 
		
		By Proposition~\ref{propositionfilteredHPisacdhsheafinchar0}, the fibre of the natural map of spectra
		$$\text{Fil}^i_{\text{HKR}} \text{HC}^-(X_{\Q}/\Q) \longrightarrow \big(L_{\text{cdh}} \text{Fil}^i_{\text{HKR}} \text{HC}^-(-_{\Q}/\Q)\big)(X)$$
		is naturally identified with the fibre of the natural map of spectra
		$$\text{Fil}^{i}_{\text{HKR}} \text{HC}(X_{\Q}/\Q)[1] \longrightarrow \big(L_{\text{cdh}} \text{Fil}^{i}_{\text{HKR}} \text{HC}(-_{\Q}/\Q)[1]\big)(X).$$
		The presheaf $\text{Fil}^{i}_{\text{HKR}} \text{HC}(-_{\Q}/\Q)$ takes values in cohomological degrees at most $-i-1$ (Proposition~\ref{propositionHKRfiltrationonHCrationalisfinitary}), so the previous fibre is in cohomological degrees at most $-i+d-1$ (\cite[Theorem~$3.12$]{clausen_hyperdescent_2021} and \cite[Theorem~$2.4.15$]{elmanto_cdh_2021}).
		
		Similarly, by Corollary~\ref{corollaryHPsolidallprimespisacdhsheafwithfiltration}, the fibre of the natural map of spectra
		$$\text{Fil}^i_{\text{HKR}}\text{HH}(X;\mathbb{A}_f)^{h\text{S}^1} \longrightarrow \Big(L_{\text{cdh}} \text{Fil}^i_{\text{HKR}} \text{HH}(-;\mathbb{A}_f)^{h\text{S}^1}\Big)(X)$$
		is naturally identified with the fibre of the natural map of spectra
		$$\text{Fil}^{i}_{\text{HKR}}\text{HH}(X;\mathbb{A}_f)_{h\text{S}^1}[1] \longrightarrow \Big(L_{\text{cdh}} \text{Fil}^{i}_{\text{HKR}} \text{HH}(-;\mathbb{A}_f)_{h\text{S}^1}[1]\Big)(X)$$
		The presheaf $\text{Fil}^{i}_{\text{HKR}}\text{HH}(-;\mathbb{A}_f)_{h\text{S}^1}$ takes values in cohomological degrees at most $-i-1$ (Remark~\ref{remark15'HKRfiltrationonHCsolidallprimesprestrictedproduct}), so the previous fibre is in cohomological degrees at most $-i+d-1$ (\cite[Theorem~$3.12$]{clausen_hyperdescent_2021} and \cite[Theorem~$2.4.15$]{elmanto_cdh_2021}).
		
		The previous three connectivity results imply the desired result.
	\end{proof}
	
	\begin{proposition}[Completeness of the motivic filtration]\label{propositionmotivicfiltrationiscompleteonqcqsschemesoffinitevaluativedimension}
		Let $d \geq 0$ be an integer, and $X$~be a qcqs scheme of valuative dimension at most $d$. Then for every integer $i \in \Z$, the spectrum $\emph{Fil}^i_{\emph{mot}} \emph{K}(X)$ is in cohomological degrees at most $-i+d+2$. In particular, the motivic filtration $\emph{Fil}^\star_{\emph{mot}} \emph{K}(X)$, and its rationalisation $\emph{Fil}^\star_{\emph{mot}} \emph{K}(X;\Q)$, are complete.
	\end{proposition}
	
	\begin{proof}
		As in the proof of Proposition~\ref{propositionfibrefilteredTCandLcdhTCiscomplete}, the last statement is a consequence of the first. By Definition~\ref{definitionmotivicfiltrationonKtheoryofschemes}, there is a natural fibre sequence of spectra
		$$\text{fib}\Big(\text{Fil}^i_{\text{mot}}\text{TC}(X) \longrightarrow \text{Fil}^i_{\text{mot}} L_{\text{cdh}} \text{TC}(X)\Big) \longrightarrow \text{Fil}^i_{\text{mot}} \text{K}(X) \longrightarrow \text{Fil}^\star_{\text{cdh}} \text{KH}(X).$$
		The left term of this fibre sequence is in cohomological degrees at most $-i+d+2$ by Proposition~\ref{propositionfibrefilteredTCandLcdhTCiscomplete}, and the right term is in cohomological degrees at most $-i+d$ by Theorem~\ref{theoremBEM}\,(1), hence the desired result.
	\end{proof}
	
	\begin{remark}[Non-noetherian Weibel vanishing]
		Let $X$ be a qcqs scheme of valuative dimension at most $d$. The same argument as in Proposition~\ref{propositionmotivicfiltrationiscompleteonqcqsschemesoffinitevaluativedimension} implies that the negative $K$-groups $\text{K}_{-n}(X)$ vanish for integers $n > d+2$. This is a weak form of Weibel vanishing in the non-noetherian case. Note that the stronger (and optimal) result that these negative $K$-groups $\text{K}_{-n}(X)$ vanish for integers $n > d$ was recently proved by Kelly--Saito--Tamme (\cite[Theorem~B\,(1)]{kelly_procdh_2026}), thus extending the result in the case of noetherian schemes by Kerz--Saito--Tamme (\cite[Theorem~B\,(i)]{kerz_algebraic_2018}). 
	\end{remark}
	
	\begin{remark}[Motivic Weibel vanishing]\label{remarkweakconnectivityformotiviccomplexes}
		Let $X$ be a qcqs scheme of dimension at most $d$. Proposition~\ref{propositionmotivicfiltrationiscompleteonqcqsschemesoffinitevaluativedimension} implies that for every integer $i \in \Z$, the motivic complex $\Z(i)^{\text{mot}}(X) \in \mathcal{D}(\Z)$ is in degrees at most~\hbox{$i+d+2$}. When $X$ is noetherian (in which case the valuative and Krull dimensions coincide), it can even be proved to be in degrees at most $i+d$ (\cite[Theorem~D]{bouis_weibel_2025}).
	\end{remark}
	
	\begin{remark}\label{remark28statementsongradedpiecesimplystatementonfiltration}
		A map of finitary presheaves of filtered spectra on qcqs schemes, which are filtration-complete on finite-dimensional noetherian schemes, is an equivalence if and only if it is an equivalence on graded pieces. In light of Propositions~\ref{propositionCmotivicfiltrationfinitaryNisnevichsheaf} and~\ref{propositionmotivicfiltrationiscompleteonqcqsschemesoffinitevaluativedimension}, we will formulate most of our results at the level of motivic cohomology, although they can often be promoted to results on the associated filtered spectra.
	\end{remark}
	
	\begin{corollary}[Completeness of $\text{Fil}^\star_{\text{mot}} L_{\text{cdh}} \text{TC}$]
		Let $X$ be a qcqs scheme of finite valuative dimension. Then the filtrations $\emph{Fil}^\star_{\emph{mot}} L_{\emph{cdh}} \emph{TC}(X)$ and $\emph{Fil}^\star_{\emph{mot}} L_{\emph{cdh}} \emph{TC}(X;\Q)$ are complete.
	\end{corollary}
	
	\begin{proof}
		This is a consequence of Propositions~\ref{propositionfilteredTCandrationalfilteredTCarecomplete} and \ref{propositionfibrefilteredTCandLcdhTCiscomplete}.
	\end{proof}
	
	\begin{remark}
		The filtrations $\text{Fil}^\star_{\text{mot}} \text{TC}(X)$ and $\text{Fil}^\star_{\text{mot}} L_{\text{cdh}} \text{TC}(X)$ do not satisfy separately a connectivity bound similar to that of Proposition~\ref{propositionfibrefilteredTCandLcdhTCiscomplete}.
	\end{remark}
	
	\begin{corollary}[Atiyah--Hirzebruch spectral sequence]\label{corollaryAHSS}
		Let $X$ be a qcqs derived scheme. The motivic filtration $\emph{Fil}^\star_{\emph{mot}} \emph{K}(X)$ on non-connective algebraic $K$-theory $\emph{K}(X)$ induces a natural Atiyah--Hirzebruch spectral sequence 
		$$E^{i,j}_2 = \emph{H}^{i-j}_{\emph{mot}}(X,\Z(-j)) \Longrightarrow \emph{K}_{-i-j}(X).$$
		If $X$ is a qcqs classical scheme of finite valuative dimension, then this Atiyah--Hirzebruch spectral sequence is convergent. With rational coefficients, the Atiyah--Hirzebruch spectral sequence degenerates for any qcqs derived scheme.
	\end{corollary}
	
	\begin{proof}
		The first statement is a consequence of the fact the motivic filtration $\text{Fil}^\star_{\text{mot}} \text{K}(X)$ is $\N$-indexed (Theorem~\ref{propositionmotivicfiltrationisexhaustive}). The second statement is a consequence of the connectivity bound for this motivic filtration (Proposition~\ref{propositionmotivicfiltrationiscompleteonqcqsschemesoffinitevaluativedimension}). The third statement is a consequence of the rational splitting of the motivic filtration (Theorem~\ref{theorem21'rationalmotivicfiltrationsplits}).
	\end{proof}
	
	The main consequence of Theorem~\ref{propositionmotivicfiltrationisexhaustive} and Proposition~\ref{propositionmotivicfiltrationiscompleteonqcqsschemesoffinitevaluativedimension} that we will use is the following result.
	
	\begin{corollary}\label{corollaryKtheorysplitsrationally}
		Let $X$ be a qcqs derived scheme. Then for every integer $i \geq 0$, there exists a natural equivalence of spectra
		$$\emph{K}(X;\Q) \simeq \big(\bigoplus_{0 \leq j < i} \Q(i)^{\emph{mot}}(X)[2i]\big) \oplus \emph{Fil}^i_{\emph{mot}} \emph{K}(X;\Q).$$
	\end{corollary}
	
	\begin{proof}
		The motivic filtration $\text{Fil}^\star_{\text{mot}} \text{K}(X;\Q)$ is $\N$-indexed by Theorem~\ref{propositionmotivicfiltrationisexhaustive}. The result on qcqs classical schemes is then a consequence of Lemma~\ref{lemma29howtouseAdamsoperations} and Remark~\ref{remarklemmahowtouseAdamsonclassicalschemes}, where the necessary hypotheses are satisfied by Proposition~\ref{propositionmotivicfiltrationiscompleteonqcqsschemesoffinitevaluativedimension} and Corollary~\ref{corollary22''AdamsoperationsongradedpiecesoffilteredKtheory}.
		
		Assume now that $X$ is a general qcqs derived scheme. Again by Lemma~\ref{lemma29howtouseAdamsoperations}, where the necessary hypotheses are satisfied for the filtration $\text{Fil}^\star_{\text{mot}} \text{TC}(-)$ by Proposition~\ref{propositionfilteredTCandrationalfilteredTCarecomplete} and Corollary~\ref{corollary31''AdamsoperationsongradedpiecesoffilteredTC}, there is a natural equivalence of spectra
		$$\text{Fil}^0_{\text{mot}} \text{TC}(X;\Q) \simeq \big(\bigoplus_{0 \leq j < i} \Q(j)^{\text{TC}}(X)[2j]\big) \oplus \text{Fil}^i_{\text{mot}} \text{TC}(X;\Q).$$
		The result is then a consequence of Definition~\ref{definitionmotivicfiltrationonKtheoryofderivedschemes} and the previous case, where the equivalences
		$$\text{Fil}^0_{\text{mot}} \text{K}(X^{\text{cl}};\Q) \simeq \big(\bigoplus_{0 \leq j < i} \Q(j)^{\text{mot}}(X^{\text{cl}})[2j]\big) \oplus \text{Fil}^i_{\text{mot}} \text{K}(X^{\text{cl}};\Q)$$
		and 
		$$\text{Fil}^0_{\text{mot}} \text{TC}(X^{\text{cl}};\Q) \simeq \big(\bigoplus_{0 \leq j < i} \Q(j)^{\text{TC}}(X^{\text{cl}})[2j]\big) \oplus \text{Fil}^i_{\text{mot}} \text{TC}(X^{\text{cl}};\Q)$$
		are compatible, by construction, with the natural map
		$$\text{Fil}^\star_{\text{mot}} \text{K}(X^{\text{cl}};\Q) \longrightarrow \text{Fil}^\star_{\text{mot}} \text{TC}(X^{\text{cl}};\Q)$$
		of Definition~\ref{definitionmotivicfiltrationonKtheoryofschemes}.
	\end{proof}

	\begin{lemma}\label{lemmacdhsheafificationcommuteswithfilteredcolimits}
		Let $\tau$ be the Zariski, Nisnevich, or cdh topology, and $(F_i)_{i \in I}$ be a direct system of presheaves. Then the natural map of presheaves
		$$\lim\limits_{\longrightarrow i} L_{\tau}\, F_i \longrightarrow L_{\tau} \lim\limits_{\longrightarrow i} F_i$$
		is an equivalence. In particular, the sheafification functor $L_{\tau}$ sends finitary presheaves to finitary $\tau$ sheaves.
	\end{lemma}
	
	\begin{proof}
		As a left adjoint, the sheafification functor $L_{\tau}$, from presheaves to $\tau$ sheaves, commutes with all colimits. Being a sheaf for the topology $\tau$ is detected using only finite limits, so the inclusion functor from $\tau$ sheaves to presheaves commutes with filtered colimits. Composing the previous two functors then implies that the functor $L_{\tau}$, from presheaves to presheaves, commutes with filtered colimits. 
		
		To prove the second statement, let $F$ be a finitary presheaf, and $(R_i)_{i \in I}$ be a direct system of commutative rings. The fact that the natural map $$\lim\limits_{\longrightarrow i} L_{\tau}\, F(R_i) \longrightarrow L_{\tau} F(\lim\limits_{\longrightarrow i} R_i)$$
		is an equivalence is a consequence of the finitariness of $F$, and of the first statement applied to the direct system of presheaves $\big(F(-_{R_i})\big)_{i \in I}$.
	\end{proof}
	
	\begin{proposition}[The motivic filtration is finitary]\label{propositionCmotivicfiltrationfinitaryNisnevichsheaf}
		Let $i \in \Z$ be an integer. Then the presheaf
		$$\emph{Fil}^i_{\emph{mot}} \emph{K}(-) : \emph{dSch}^{\emph{qcqs,op}} \longrightarrow \emph{Sp}$$
		is a finitary Nisnevich sheaf, \emph{i.e.}, it satisfies descent for the Nisnevich topology and commutes with filtered colimits of rings. 
	\end{proposition}
	
	\begin{proof}
		It suffices to prove the result modulo $p$ for every prime number $p$, and rationally. Algebraic $K$-theory is a finitary Nisnevich sheaf on qcqs derived schemes (\cite[Proposition~A.$15$]{clausen_descent_2020}). The presheaf $\text{Fil}^i_{\text{mot}} \text{K}(-;\Q)$ is a natural direct summand of rationalised algebraic $K$-theory (Corollary~\ref{corollaryKtheorysplitsrationally}), so it also is a finitary Nisnevich sheaf. To prove the result modulo a prime number $p$, note that the presheaf $\text{Fil}^i_{\text{cdh}} \text{KH}(-)$ is a finitary cdh sheaf (Definition~\ref{definitionmotivicfiltrationonKHtheoryofschemes} and Lemma~\ref{lemmacdhsheafificationcommuteswithfilteredcolimits}), and in particular a finitary Nisnevich sheaf. By Theorem~\ref{theoremBMSfiltrationonTCpcompletedsatisfiesquasisyntomicdescentandisLKEfrompolynomialalgebras}, the presheaf $\text{Fil}^i_{\text{BMS}} \text{TC}(-;\F_p)$ is a finitary Nisnevich sheaf. By Lemma~\ref{lemmacdhsheafificationcommuteswithfilteredcolimits}, this implies that the presheaf $L_{\text{cdh}} \text{Fil}^i_{\text{BMS}} \text{TC}(-;\F_p)$ is a finitary Nisnevich sheaf, and the result modulo $p$ is then a consequence of Proposition~\ref{propositionpadicstructuremain}.
	\end{proof}
	
	\begin{corollary}\label{corollarymotiviccomplexesarefinitary}
		Let $i \in \Z$ be an integer. Then the presheaf
		$$\Z(i)^{\emph{mot}}(-) : \emph{dSch}^{\emph{qcqs,op}} \longrightarrow \mathcal{D}(\Z)$$
		is a finitary Nisnevich sheaf.
	\end{corollary}
	
	\begin{proof}
		This is a direct consequence of Proposition~\ref{propositionCmotivicfiltrationfinitaryNisnevichsheaf}.
	\end{proof}
	
	\begin{proof}[Proof of Theorem~\ref{theorem21'rationalmotivicfiltrationsplits}]
		The motivic filtration $\text{Fil}^\star_{\text{mot}} \text{K}(-)$ is finitary on qcqs schemes (Proposition~\ref{propositionCmotivicfiltrationfinitaryNisnevichsheaf}), and, for every noetherian scheme $X$ of finite dimension $d$ and every integer $i \in \Z$, the spectrum $\text{Fil}^i_{\text{mot}} \text{K}(X;\Q)$ is in cohomological degrees at most $-i+d+2$ (Proposition~\ref{propositionmotivicfiltrationiscompleteonqcqsschemesoffinitevaluativedimension}). Proposition~\ref{propositionhowtouseAdamsoperations} then implies that there exists a natural multiplicative equivalence of filtered spectra
		$$\text{Fil}^\star_{\text{mot}} \text{K}(X;\Q) \simeq \bigoplus_{j \geq \star} \Q(j)^{\text{mot}}(X)[2j].$$
		The same argument as in Corollary~\ref{corollaryKtheorysplitsrationally} then implies the result for general qcqs derived schemes.
	\end{proof}

	\begin{corollary}\label{corollaryKtheorysplitsrationally2}
		Let $X$ be a qcqs derived scheme. Then there exists a natural equivalence of spectra
		$$\emph{K}(X;\Q) \simeq \bigoplus_{i \geq 0} \Q(i)^{\emph{mot}}(X)[2i].$$
	\end{corollary}
	
	\begin{proof}
		The motivic filtration $\text{Fil}^\star_{\text{mot}} \text{K}(X;\Q)$ is $\N$-indexed by Theorem~\ref{propositionmotivicfiltrationisexhaustive}. The result is then a consequence of the rational splitting Theorem~\ref{theorem21'rationalmotivicfiltrationsplits}.
	\end{proof}

    \medskip

	\subsection{Rational structure of motivic cohomology}\label{subsectionrationalstructure}
	
	\vspace{-\parindent}
	\hspace{\parindent}
	
	In this subsection, we finish the proof of Theorem~\ref{theoremrationalstructuremain}. We first use an argument of Weibel to prove the following result at the level of $K$-theory. We then use the rational splitting Corollary~\ref{corollaryKtheorysplitsrationally2} to prove a filtered version of this result, which reduces the proof of Theorem~\ref{theoremrationalstructuremain} to the case of characteristic zero.
	
	\begin{lemma}\label{lemma2WeibelargumentrationalKtheory}
		Let $X$ be a qcqs scheme. Then the natural commutative diagram
		$$\begin{tikzcd}
			\emph{K}(X;\Q) \arrow{r} \arrow{d} & \emph{K}(X_{\Q};\Q) \ar[d] \\
			\emph{KH}(X;\Q) \arrow{r} & \emph{KH}(X_{\Q};\Q)
		\end{tikzcd}$$
		is a cartesian square of spectra.
	\end{lemma}
	
	\begin{proof}
		By Zariski descent, it suffices to prove the result for affine schemes $X=\text{Spec}(R)$. For every integer $n \in \Z$, let $$\text{NK}_n(R) := \text{coker}\big(\text{K}_n(R) \longrightarrow \text{K}_n(R[T])\big)$$
		be the $n^{\text{th}}$ $NK$-group of $R$. By \cite[Corollary~$6.4$]{weibel_meyer_1981}, there is a natural isomorphism of abelian groups
        $$\text{NK}_n(R)\otimes_{\Z} S^{-1}\Z \xlongrightarrow{\cong} \text{NK}_n(R\otimes_{\Z} S^{-1}\Z)$$
        for every integer $n \in \Z$, every multiplicative set $S \subseteq \Z$, and every commutative ring $R$.\footnote{More precisely, Weibel assumes in \cite[Corollary~6.4]{weibel_meyer_1981} that $S$ is a multiplicative set of nonzero divisors in the commutative ring $R$, but this hypothesis can be removed by \cite[Exercise~9.12]{thomason_higher_1990}.} Taking $S$ to be the set of positive integers, this implies that there is a natural isomorphism of abelian groups
		$$\text{NK}_n(R)\otimes_{\Z} \Q \xlongrightarrow{\cong} \text{NK}_n(R\otimes_{\Z} \Q)$$
		for every integer $n \in \Z$ and every commutative ring $R$. For every commutative ring $R$, the homotopy groups of the fibre $\text{K}^{\text{W}}(R)$ of the map of spectra $\text{K}(R) \rightarrow \text{KH}(R)$ have a natural exhaustive complete filtration with graded pieces given by the iterated $NK$-groups of $R$. In particular, for every integer $n \in \Z$ and every commutative ring $R$, the natural map 
		$$\text{K}^{\text{W}}(R;\Q) \longrightarrow \text{K}^{\text{W}}(R\otimes_{\Z} \Q;\Q)$$
		is an equivalence of spectra, which implies the desired result.
	\end{proof}

    The following result, although easy to prove and a seemingly significant structural result on rational $K$-theory, does not seem to have been previously recorded in the literature.
	\begin{corollary}\label{corollaryTCrationalwithcdhandHC-withoutfiltration}
		Let $X$ be a qcqs scheme. Then the natural commutative diagram
		$$\begin{tikzcd}
			\emph{K}(X;\Q) \arrow{r} \arrow{d} & \emph{HC}^-(X_{\Q}/\Q) \ar[d] \\
			\emph{KH}(X;\Q) \arrow[r] & L_{\emph{cdh}} \emph{HC}^-(-_{\Q}/\Q)(X)
		\end{tikzcd}$$
		is a cartesian square of spectra.
	\end{corollary}
	
	\begin{proof}
		The result for qcqs $\Q$-schemes $X$ is due to Corti\~{n}as--Haesemeyer--Schlichting--Weibel \cite{cortinas_cyclic_2008,cortinas_K-regularity_2008} (see also Theorem~\ref{theoremKST+LT}). The general result is then a consequence of Lemma~\ref{lemma2WeibelargumentrationalKtheory}.
	\end{proof}
	
	We record for completeness the following result, which is well-known in characteristic zero (\cite{cortinas_cyclic_2008,cortinas_K-regularity_2008}).
	
	\begin{corollary}\label{corollaryrationalKtheoryintermsofHC}
		Let $X$ be a qcqs scheme. Then there is a natural fibre sequence of spectra
		$$\emph{K}(X;\Q) \longrightarrow \emph{KH}(X;\Q) \longrightarrow \emph{cofib}\big(\emph{HC}(X_{\Q}/\Q) \longrightarrow L_{\emph{cdh}} \emph{HC}(-_{\Q}/\Q)(X)\big)[1].$$
	\end{corollary}
	
	\begin{proof}
		By Theorem~\ref{theoremKST+LT} and Lemma~\ref{lemma2WeibelargumentrationalKtheory}, the natural commutative diagram
		$$\begin{tikzcd}
			\text{K}(X;\Q) \arrow{r} \arrow{d} & \text{HC}^-(X_{\Q}/\Q) \ar[d] \\
			\text{KH}(X;\Q) \arrow{r} & L_{\text{cdh}} \text{HC}^-(-_{\Q}/\Q)(X)
		\end{tikzcd}$$
		is a cartesian square of spectra. By construction, there is moreover a natural commutative diagram of spectra
		$$\begin{tikzcd}
			\text{HC}^-(X_{\Q}/\Q) \arrow{r} \arrow{d} & \text{HP}(X_{\Q}/\Q) \ar[d] \ar[r] & \text{HC}(X_{\Q}/\Q)[2] \ar[d] \\
			L_{\text{cdh}} \text{HC}^-(-_{\Q}/\Q)(X) \arrow{r} & L_{\text{cdh}} \text{HP}(-_{\Q}/\Q)(X) \arrow{r} & L_{\text{cdh}} \text{HC}(-_{\Q}/\Q)(X)[2],
		\end{tikzcd}$$
		where the horizontal lines are fibre sequences. The middle vertical line of this diagram is an equivalence (\cite[Corollary~$3.13$]{cortinas_cyclic_2008}, see also \cite[Corollary~A.$6$]{land_k-theory_2019}), so the cofibre of the left vertical map is naturally identified with the spectrum
		$$\text{cofib}\big(\text{HC}(X_{\Q}/\Q) \longrightarrow L_{\text{cdh}} \text{HC}(-_{\Q}/\Q)(X)\big)[1].$$
	\end{proof}
	
	The following result is a filtered refinement of Lemma~\ref{lemma2WeibelargumentrationalKtheory}.
	
	\begin{corollary}\label{corollaryfilteredLemma2onKandKHwithrational}
		Let $X$ be a qcqs scheme. Then the natural commutative diagram
		$$\begin{tikzcd}
			\emph{Fil}^\star_{\emph{mot}} \emph{K}(X;\Q) \arrow{r} \arrow{d} & \emph{Fil}^\star_{\emph{mot}} \emph{K}(X_{\Q};\Q) \ar[d] \\
			\emph{Fil}^\star_{\emph{cdh}} \emph{KH}(X;\Q) \arrow{r} & \emph{Fil}^\star_{\emph{cdh}} \emph{KH}(X_{\Q};\Q)
		\end{tikzcd}$$
		is a cartesian square of filtered spectra.
	\end{corollary}
	
	\begin{proof}
		By Corollary~\ref{corollaryKtheorysplitsrationally}, for every integer $i \geq 0$, the spectrum $\text{Fil}^i_{\text{mot}} \text{K}(X;\Q)$ is naturally a direct summand of the spectrum $\text{K}(X;\Q)$. The same applies to the other three filtrations of the cartesian square, by noting that the motivic filtration $\text{Fil}^\star_{\text{cdh}} \text{KH}(-)$ also is $\N$-indexed (Theorem~\ref{theoremBEM}\,(1)). The compatibility between the several filtrations is automatic from the construction of the splittings. So the result is a consequence of Lemma~\ref{lemma2WeibelargumentrationalKtheory}.
	\end{proof}

	The following result is a filtered refinement of Corollary~\ref{corollaryTCrationalwithcdhandHC-withoutfiltration}.
	
	\begin{theorem}\label{theoremTCrationalwithcdhandHC-filtered}
		Let $X$ be a qcqs scheme. Then the natural commutative diagram
		$$\begin{tikzcd}
			\emph{Fil}^\star_{\emph{mot}} \emph{K}(X;\Q) \arrow{r} \arrow{d} & \emph{Fil}^\star_{\emph{HKR}} \emph{HC}^-(X_{\Q}/\Q) \ar[d] \\
			\emph{Fil}^\star_{\emph{cdh}} \emph{KH}(X;\Q) \arrow[r] & L_{\emph{cdh}} \emph{Fil}^\star_{\emph{HKR}} \emph{HC}^-(-_{\Q}/\Q)(X)
		\end{tikzcd}$$
		is a cartesian square of filtered spectra.
	\end{theorem}
	
	\begin{proof}
		This is a consequence of Corollary~\ref{corollaryfilteredLemma2onKandKHwithrational}, where the filtration $\text{Fil}^\star_{\text{mot}} \text{TC}(-)$ of qcqs $\Q$-schemes is naturally identified with the filtration $\text{Fil}^\star_{\text{HKR}} \text{HC}^-(-/\Q)$ (Remark~\ref{remarkcomparisontoEMfiltrationonTCoverafield}).
	\end{proof}

	The following result is the rational part of Theorem~\ref{theoremintromaincartesiansquares}.

	\begin{corollary}\label{corollarycartesiansquarerational}
		Let $X$ be a qcqs scheme. Then for every integer $i \in \Z$, the natural commutative diagram
		$$\begin{tikzcd}
			\Q(i)^{\emph{mot}}(X) \ar[r] \ar[d] & R\Gamma_{\emph{Zar}}\big(X,\widehat{\mathbb{L}\Omega}^{\geq i}_{-_{\Q}/\Q}\big) \ar[d] \\
			\Q(i)^{\emph{cdh}}(X) \ar[r] & R\Gamma_{\emph{cdh}}\big(X,\widehat{\mathbb{L}\Omega}^{\geq i}_{-_{\Q}/\Q}\big)
		\end{tikzcd}$$
		is a cartesian square in the derived category $\mathcal{D}(\Q)$.
	\end{corollary}

	\begin{proof}
		This is a direct consequence of Theorem~\ref{theoremTCrationalwithcdhandHC-filtered}.
	\end{proof}
	
	\begin{proof}[Proof of Theorem~\ref{theoremrationalstructuremain}]
		By Theorem~\ref{theoremTCrationalwithcdhandHC-filtered}, the natural commutative diagram
		$$\begin{tikzcd}
			\text{Fil}^\star_{\text{mot}} \text{K}(X;\Q) \arrow{r} \arrow{d} & \text{Fil}^\star_{\text{HKR}} \text{HC}^-(X_{\Q}/\Q) \ar[d] \\
			\text{Fil}^\star_{\text{cdh}} \text{KH}(X;\Q) \arrow{r} & L_{\text{cdh}} \text{Fil}^\star_{\text{HKR}} \text{HC}^-(-_{\Q}/\Q)(X)
		\end{tikzcd}$$
		is a cartesian square of filtered spectra. By Definition~\ref{definitionHKRfiltrationonHC} and Remark~\ref{remarkHHandvariantsrelativetoQ}, there is a natural commutative diagram of filtered spectra
		$$\hspace*{-.25cm}\begin{tikzcd}[sep=small]
			\text{Fil}^\star_{\text{HKR}} \text{HC}^-(X_{\Q}/\Q) \arrow{r} \arrow{d} & \text{Fil}^\star_{\text{HKR}} \text{HP}(X_{\Q}/\Q) \ar[d] \ar[r] & \text{Fil}^{\star}_{\text{HKR}} \text{HC}(X_{\Q}/\Q)[2] \ar[d] \\
			L_{\text{cdh}} \text{Fil}^\star_{\text{HKR}} \text{HC}^-(-_{\Q}/\Q)(X) \arrow{r} & L_{\text{cdh}} \text{Fil}^\star_{\text{HKR}} \text{HP}(-_{\Q}/\Q)(X) \arrow{r} & L_{\text{cdh}} \text{Fil}^{\star}_{\text{HKR}} \text{HC}(-_{\Q}/\Q)(X)[2]
		\end{tikzcd}$$
		where the horizontal lines are fibre sequences. The middle vertical map of this diagram is an equivalence (Proposition~\ref{propositionfilteredHPisacdhsheafinchar0}), so the cofibre of the left vertical map is naturally identified with the filtered spectrum
		$$\text{cofib}\big(\text{Fil}^{\star}_{\text{HKR}} \text{HC}(X_{\Q}/\Q) \longrightarrow L_{\text{cdh}} \text{Fil}^{\star}_{\text{HKR}} \text{HC}(-_{\Q}/\Q)(X)\big) [1].$$
	\end{proof}

	The following result was proved by Elmanto--Morrow \cite{elmanto_motivic_2023} for qcqs schemes $X$ over $\Q$.
	
	\begin{corollary}[Rational motivic cohomology]\label{corollaryrationalmainresultongradedpieces}
		Let $X$ be qcqs scheme. Then for every integer $i \in \Z$, there is a natural fibre sequence
		$$\Q(i)^{\emph{mot}}(X) \longrightarrow \Q(i)^{\emph{cdh}}(X) \longrightarrow \emph{cofib}\Big(R\Gamma_{\emph{Zar}}\big(X,\mathbb{L}\Omega^{<i}_{-_{\Q}/\Q}\big) \longrightarrow R\Gamma_{\emph{cdh}}\big(X,\Omega^{<i}_{-_{\Q}/\Q}\big)\Big)[-1]$$
		in the derived category $\mathcal{D}(\Q)$.
	\end{corollary}
	
	\begin{proof}
		For every valuation ring extension $V$ of $\Q$, the natural map
		$$\mathbb{L}\Omega^{<i}_{V/\Q} \longrightarrow \Omega^{<i}_{V/\Q}$$
		is an equivalence in the derived category $\mathcal{D}(\Q)$ by results of Gabber--Ramero (\cite[Theorem~$6.5.8\,(ii)$ and Corollary~$6.5.21$]{gabber_almost_2003}). The presheaves $R\Gamma_{\text{cdh}}(-, \mathbb{L}\Omega^{<i}_{-_{\Q}/\Q})$ and $R\Gamma_{\text{cdh}}(-,\Omega^{<i}_{-_{\Q}/\Q})$ are finitary cdh sheaves on qcqs schemes, so the natural map
		$$R\Gamma_{\text{cdh}}\big(-,\mathbb{L}\Omega^{<i}_{-_{\Q}/\Q}\big) \longrightarrow R\Gamma_{\text{cdh}}\big(-, \Omega^{<i}_{-_{\Q}/\Q}\big)$$
		is an equivalence (\cite[Corollary~$2.4.19$]{elmanto_cdh_2021}). The result then follows from Theorem~\ref{theoremrationalstructuremain}.
	\end{proof}
	
	\begin{example}[Weight zero motivic cohomology]\label{exampleweightzeromotiviccohomology}
		For every qcqs scheme $X$, the natural map
		$$\Z(0)^{\text{mot}}(X) \longrightarrow \Z(0)^{\text{cdh}}(X) \simeq R\Gamma_{\text{cdh}}(X,\Z),$$
		where the last idenfication is a consequence of Example~\ref{exampleweightzerolissemotiviccohomology}, is an equivalence in the derived category $\mathcal{D}(\Z)$. Indeed, it suffices to prove the result rationally, and modulo $p$ for every prime number $p$. The result rationally is a consequence of Corollary~\ref{corollaryrationalmainresultongradedpieces}. For every prime number $p$, the presheaf $\F_p(0)^{\text{BMS}}$ is naturally identified with the presheaf $R\Gamma_{\text{ét}}(-,\F_p)$ (\cite[Proposition~$7.16$]{bhatt_topological_2019}) which is a cdh sheaf on qcqs schemes (\cite[Theorem~$5.4$]{bhatt_arc-topology_2021}), so the result modulo $p$ is a consequence of Corollary~\ref{corollarymainpadicstructureongradeds}.
	\end{example}

    \begin{remark}
        Following Sections~\ref{subsectionmotivicfiltrationonTC} and~\ref{sectionrigidanalyticdR}, one can define rigid-analytic variants of $\text{THH}$, $\text{TC}^-$, $\text{TP}$, and $\text{TC}$ on qcqs derived schemes, equipped with natural motivic filtrations. Proposition~\ref{propositionmainconsequenceBFSwithfiltrations} (together with the $p$-adic continuity result \cite[Corollary~7.4.11]{bhatt_absolute_2022}) then implies that (the motivic filtration on) rigid-analytic TC is naturally identified with (the HKR filtration on) rigid-analytic $\text{HC}^-$. In particular, using Corollary~\ref{corollaryHPsolidallprimespisacdhsheafwithfiltration}, there is a natural fibre sequence of presheaves of filtered spectra
        $$\prod_{p \in \mathbb{P}} \text{Fil}^\star_{\text{mot}} \text{TC}(-;\Z_p) \rightarrow \prod_{p \in \mathbb{P}} \text{Fil}^\star_{\text{mot}} L_{\text{cdh}} \text{TC}(-;\Z_p) \rightarrow \text{cofib}\big(\text{Fil}^\star_{\text{HKR}} \text{HC}(-;\mathbb{A}_f) \rightarrow \text{Fil}^\star_{\text{HKR}} L_{\text{cdh}} \text{HC}(-;\mathbb{A}_f)\big),$$
        which can be seen as a rigid-analytic version of Theorem~\ref{theoremrationalstructuremain}.
    \end{remark}

	\medskip
	
	\subsection{A global Beilinson fibre square}
	
	\vspace{-\parindent}
	\hspace{\parindent}
	
	In this subsection we prove Theorem~\ref{theoremBFSglobalwithsolid}, which is a global refinement of the Beilinson fibre square of Antieau--Mathew--Morrow--Nikolaus \cite[Theorem~$6.17$]{antieau_beilinson_2020} in terms of the rigid-analytic derived de Rham cohomology of Section~\ref{sectionrigidanalyticdR}. Note that our statements are formulated in the generality of derived schemes, and that the functor $-_{\F_p}$ is then the derived base change from~$\Z$ to $\F_p$. The results in \cite{antieau_beilinson_2020} are stated in the generality of $p$-torsionfree commutative rings, on which derived and classical reduction modulo $p$ coincide; see also the recent work of Flach--Krause--Morin \cite{flach_derham_2026} for an alternative treatment of their result, where these hypotheses are removed.

    \begin{construction}[The maps $\chi$ and $\widetilde{\chi}$]\label{constructionthemapchiandtildechi}
        Let $i \in \Z$ be an integer. Using the techniques of \cite{antieau_beilinson_2020,bhatt_absolute_2022}, and following the approach taken in \cite{flach_derham_2026}, we construct for every qcqs derived scheme $X$ a natural map
        $$\chi : \prod_{p \in \mathbb{P}} \Z_p(i)^{\text{BMS}}(X_{\F_p}) \longrightarrow \prod_{p \in \mathbb{P}} R\Gamma_{\text{Zar}}\big(X,(\mathbb{L}\Omega_{-/\Z})^\wedge_p\big)$$
        in the derived category $\mathcal{D}(\Z)$.

        For every prime number $p$ and every animated commutative ring $R$, let $\widetilde{\chi}_p$ be the composite map
        $$\widetilde{\chi}_p : \Z_p(i)^{\text{BMS}}(R/p) \simeq \text{fib}\big(\mathcal{N}^{\geq i} \Prism_{R/p}\{i\} \rightarrow \Prism_{R/p}\{i\}\big) \longrightarrow \mathcal{N}^{\geq i} \Prism_{R/p}\{i\},$$
        where Nygaard completions are not necessary in the first equivalence by \cite[Corollary~5.31]{antieau_beilinson_2020} (see also \cite[Proposition~7.4.6]{bhatt_absolute_2022}), and let $\chi_p$\footnote{Note that this map $\chi_p$ corresponds to the map $\chi_i$ of \cite[Theorem~6.17]{antieau_beilinson_2020}, constructed there for prime numbers $p \geq i+2$.} be the composite map
        $$\chi_p : \Z_p(i)^{\text{BMS}}(R/p) \xlongrightarrow{\widetilde{\chi}_p} \mathcal{N}^{\geq i} \Prism_{R/p}\{i\} \xlongrightarrow{\text{can}} \Prism_{R/p}\{i\} \simeq (\mathbb{L}\Omega_{R/\Z})^\wedge_p,$$
        where the last equivalence is \cite[Theorem~5.4.2]{bhatt_absolute_2022}. In particular, note that the maps $\chi_p$ and $\widetilde{\chi}_p$ coincide rationally, because the Nygaard filtration on the prismatic cohomology $\Prism_{R/p}\{i\}$ becomes trivial after inverting $p$.

        The map
        $$\widetilde{\chi} : \prod_{p \in \mathbb{P}} \Z_p(i)^{\text{BMS}}(X_{\F_p}) \longrightarrow \prod_{p \in \mathbb{P}} \mathcal{N}^{\geq i} \Prism_{X_{\F_p}}\{i\}$$
        is then defined as the product over prime numbers $p$ of the Zariski sheafification of the maps $\widetilde{\chi}_p$, and similarly for the map $\chi$. Note that, by construction (and \cite[Theorem~5.4.2]{bhatt_absolute_2022}), taking the direct sum over integers $i \in \Z$, the previous maps $\chi$, $\widetilde{\chi}$, $\chi_p$, and $\widetilde{\chi}_p$ are multiplicative.
    \end{construction}

    The viewpoint of the following two Beilinson fibre squares, where one replaces de Rham cohomology by the prismatic cohomology of the special fibre, was first taken in \cite{flach_derham_2026}. 
    
        \begin{proposition}[Beilinson fibre square, after \cite{antieau_beilinson_2020,flach_derham_2026}]\label{propositionchitildepcartesian}
        Let $X$ be a qcqs derived scheme, and $i \in \Z$ be an integer. Then for every prime number $p \geq i+1$, there is a natural map
        $$R\Gamma_{\emph{Zar}}\big(X,(\mathbb{L}\Omega^{\geq i}_{-/\Z})^\wedge_p\big) \longrightarrow \mathcal{N}^{\geq i} \Prism_{X_{\F_p}}\{i\}$$
        in the derived category $\mathcal{D}(\Z)$, and the resulting diagram
        $$\begin{tikzcd}
			\Z_p(i)^{\emph{BMS}}(X) \ar[r] \ar[d] & R\Gamma_{\emph{Zar}}\big(X,(\mathbb{L}\Omega^{\geq i}_{-/\Z})^\wedge_p\big) \ar[d] \\
			\Z_p(i)^{\emph{BMS}}(X_{\F_p}) \arrow[r,"\widetilde{\chi}_p"] & \mathcal{N}^{\geq i} \Prism_{X_{\F_p}}\{i\}
		\end{tikzcd}$$
        is cartesian if $p \ge i+2$. With rational coefficients, the previous map is well-defined for every prime number $p$, and the resulting previous diagram is a cartesian square in $\mathcal{D}(\Q)$.
    \end{proposition}

    \begin{proof}
        The construction of the first map (either rationally, or integrally if $p \geq i+2$) follows from \cite[Theorem~5.4.2 and Proposition~5.5.12]{bhatt_absolute_2022}. Via the identification between the syntomic cohomologies of Fontaine--Messing and of Bhatt--Morrow--Scholze for $p \geq i+2$ (\cite[Theorem~F]{antieau_beilinson_2020}), the fact that the resulting square is cartesian is a consequence of \cite[Theorem~5.4]{bloch_p-adic_2014} and \cite[Proposition~5.5.12]{bhatt_absolute_2022}.
    \end{proof}

    \begin{theorem}[Global Beilinson fibre square, after \cite{antieau_beilinson_2020,flach_derham_2026}]\label{theoremBFSasinAMMN}
        Let $X$ be a qcqs derived scheme, and $i \in \Z$ be an integer.
        \begin{enumerate}
            \item The natural diagram
            $$\begin{tikzcd}
			\Big(\prod_{p \in \mathbb{P}} \Z_p(i)^{\emph{BMS}}(X)\Big)_{\Q} \ar[r] \ar[d] & \Big(\prod_{p \in \mathbb{P}} R\Gamma_{\emph{Zar}}\big(X,(\mathbb{L}\Omega^{\geq i}_{-/\Z})^\wedge_p\big)\Big)_{\Q} \ar[d] \\
			\Big(\prod_{p \in \mathbb{P}} \Z_p(i)^{\emph{BMS}}(X_{\F_p})\Big)_{\Q} \arrow[r,"\widetilde{\chi}"] & \Big(\prod_{p \in \mathbb{P}} \mathcal{N}^{\geq i} \Prism_{X_{\F_p}}\{i\}\Big)_{\Q}
		\end{tikzcd}$$
        is cartesian in the derived category $\mathcal{D}(\Q)$. Moreover, taking the direct sum over integers $i \in \Z$, this cartesian square is a cartesian square of graded $\mathbb{E}_\infty$-$\Q$-algebras.
        \item The natural diagram
        $$\begin{tikzcd}
			\Big(\prod_{p \in \mathbb{P}} \Z_p(i)^{\emph{BMS}}(X)\Big)_{\Q} \ar[r] \ar[d] & \Big(\prod_{p \in \mathbb{P}} R\Gamma_{\emph{Zar}}\big(X,(\mathbb{L}\Omega^{\geq i}_{-/\Z})^\wedge_p\big)\Big)_{\Q} \ar[d] \\
			\Big(\prod_{p \in \mathbb{P}} \Z_p(i)^{\emph{BMS}}(X_{\F_p})\Big)_{\Q} \arrow[r,"\chi"] & \Big(\prod_{p \in \mathbb{P}} R\Gamma_{\emph{Zar}}\big(X,(\mathbb{L}\Omega_{-/\Z})^\wedge_p\big)\Big)_{\Q}
		\end{tikzcd}$$
        in the derived category $\mathcal{D}(\Q)$ is commutative, with total cofibre naturally identified with the complex 
        $$\Big(\prod_{p \in \mathbb{P}} R\Gamma_{\emph{Zar}}\big(X,\mathbb{L}\Omega^{<i}_{-_{\F_p}/\Z}\big)\Big)_{\Q} \in \mathcal{D}(\Q).$$
        \end{enumerate}
    \end{theorem}
	
	\begin{proof}
        (1) is a consequence of Proposition~\ref{propositionchitildepcartesian}. 
        (2) is a consequence of (1) and \cite[Proposition~5.5.12]{bhatt_absolute_2022}.
	\end{proof}

	In the following result, the map $\underline{\widehat{\chi}}$ is defined as the composite of the map $\chi$ of Construction~\ref{constructionthemapchiandtildechi} with the natural maps
	$$R\Gamma_{\text{Zar}}\big(X,(\mathbb{L}\Omega_{-/\Z})_{\mathbb{A}_f}\big) \rightarrow R\Gamma_{\text{Zar}}\big(X,(\widehat{\mathbb{L}\Omega}_{-/\Z})_{\mathbb{A}_f}\big) \rightarrow R\Gamma_{\text{Zar}}\big(X,\widehat{\underline{\mathbb{L}\Omega}}_{-_{\mathbb{A}_f}/\mathbb{A}_f}\big)$$
	induced by Hodge-completion and Remark~\ref{remark18cartesiansquaredefiningsolidderiveddeRhamcohomology}.

	\begin{theorem}\label{theoremBFSglobalwithsolid}
		Let $X$ be a qcqs derived scheme. Then for every integer $i \in \Z$, there exists a natural commutative square
		$$\begin{tikzcd}
			\Q(i)^{\emph{TC}}(X) \ar[r] \ar[d] & R\Gamma_{\emph{Zar}}\big(X,\widehat{\mathbb{L}\Omega}^{\geq i}_{-_{\Q}/\Q}\big) \ar[d] \\
			\big(\prod_{p \in \mathbb{P}} \Z_p(i)^{\emph{BMS}}(X_{\F_p})\big)_{\Q} \arrow[r,"\underline{\widehat{\chi}}"] & R\Gamma_{\emph{Zar}}\big(X,\widehat{\underline{\mathbb{L}\Omega}}_{-_{\mathbb{A}_f}/\mathbb{A}_f}\big)
		\end{tikzcd}$$
		in the derived category $\mathcal{D}(\Q)$, whose total cofibre is naturally identified with the complex
		$$\Big(\prod_{p \in \mathbb{P}} R\Gamma_{\emph{Zar}}\big(X,\mathbb{L}\Omega^{<i}_{-_{\F_p}/\Z}\big)\Big)_{\Q} \in \mathcal{D}(\Q).$$
	\end{theorem}

	\begin{proof}
		By construction, there exists a natural commutative diagram
		$$\hspace*{-.2cm}\begin{tikzcd}[sep=small]
			\Q(i)^{\text{TC}}(X) \ar[r] \ar[d] & R\Gamma_{\text{Zar}}\big(X,(\widehat{\mathbb{L}\Omega}^{\geq i}_{-/\Z})_{\Q}\big) \ar[r] \ar[d] & R\Gamma_{\text{Zar}}\big(X,\widehat{\mathbb{L}\Omega}^{\geq i}_{-_{\Q}/\Q}\big) \ar[d] \\
			\Big(\prod_{p \in \mathbb{P}} \Z_p(i)^{\text{BMS}}(X)\Big)_{\Q} \ar[r] \ar[d] & R\Gamma_{\text{Zar}}\big(X,(\widehat{\mathbb{L}\Omega}^{\geq i}_{-/\Z})_{\mathbb{A}_f}\big) \ar[r] \ar[d] & R\Gamma_{\text{Zar}}\big(X,\widehat{\underline{\mathbb{L}\Omega}}^{\geq i}_{-_{\mathbb{A}_f}/\mathbb{A}_f}\big) \ar[d] \\
			\Big(\prod_{p \in \mathbb{P}} \Z_p(i)^{\text{BMS}}(X_{\F_p})\Big)_{\Q} \arrow[r] & R\Gamma_{\text{Zar}}\big(X,(\widehat{\mathbb{L}\Omega}_{-/\Z})_{\mathbb{A}_f}\big) \ar[r] & R\Gamma_{\text{Zar}}\big(X,\widehat{\underline{\mathbb{L}\Omega}}_{-_{\mathbb{A}_f}/\mathbb{A}_f}\big)
		\end{tikzcd}$$
		in the derived category $\mathcal{D}(\Q)$, where all the inner squares are cartesian expect the left bottom one, and where the commutativity for the left bottom square is part of Theorem~\ref{theoremBFSasinAMMN}. The desired total cofibre is then naturally identified with the total cofibre of the left bottom square, and the result is then a consequence of Theorem~\ref{theoremBFSasinAMMN}.
	\end{proof}

	Theorem~\ref{theoremBFSglobalwithsolid} means in particular that the complex $\Q(i)^{\text{TC}}(X)$ can be expressed purely in terms of characteristic zero, characteristic $p$, and rigid-analytic information.

	\begin{corollary}\label{corollaryBFSonlyonep}
		Let $p$ be a prime number, and $X$ be a qcqs derived $\Z_{(p)}$-scheme. Then for every integer $i \in \Z$, there exists a natural cartesian square
		$$\begin{tikzcd}
			\Q(i)^{\emph{TC}}(X) \ar[r] \ar[d] & R\Gamma_{\emph{Zar}}\big(X,\widehat{\mathbb{L}\Omega}^{\geq i}_{-_{\Q}/\Q}\big) \ar[d] \\
			\Q_p(i)^{\emph{BMS}}(X_{\F_p}) \arrow[r,"\underline{\widehat{\chi}}"] & R\Gamma_{\emph{Zar}}\big(X,\underline{\widehat{\mathbb{L}\Omega}}_{-_{\Q_p}/\Q_p}\big)
		\end{tikzcd}$$
		in the derived category $\mathcal{D}(\Q)$.
	\end{corollary}

	\begin{proof}
		The base change $X_{\F_{\ell}}$ is zero for every prime number $\ell$ different from~$p$. The complex~$\mathbb{L}\Omega^{<i}_{X_{\F_p}/\Z}$ is $\F_p$-linear, so its rationalisation vanishes. The result is then a consequence of Theorem~\ref{theoremBFSglobalwithsolid} and Remark~\ref{remark14HKRfiltrationonHPsolid}.
	\end{proof}

	The following result is an analogue of Theorem~\ref{theoremBFSglobalwithsolid} at the level of filtered objects.

	\begin{theorem}\label{theoremmainconsequenceBFSwithfiltrations}
		Let $X$ be qcqs derived scheme. Then there exists a natural commutative square of filtered spectra
		$$\begin{tikzcd}
			\emph{Fil}^\star_{\emph{mot}} \emph{TC}(X;\Q) \arrow{r} \arrow{d} & \emph{Fil}^\star_{\emph{HKR}} \emph{HC}^-(X_{\Q}/\Q) \ar[d] \\
			\Big(\prod_{p \in \mathbb{P}} \emph{Fil}^\star_{\emph{BMS}} \emph{TC}(X_{\F_p})\Big)_{\Q} \arrow[r,"\underline{\widehat{\chi}}"] & \emph{Fil}^\star_{\emph{HKR}} \emph{HH}(X;\mathbb{A}_f)^{t\emph{S}^1},
		\end{tikzcd}$$
		whose total cofibre is naturally identified with the filtered spectrum
		$$\big(\prod_{p \in \mathbb{P}} \emph{Fil}^{\star}_{\emph{HKR}} \emph{HC}(X_{\F_p})\big)_{\Q}[2].$$
	\end{theorem}

	\begin{proof}
		The construction of the map $\chi$ in the proof of \cite[Theorem~$6.17$]{antieau_beilinson_2020} adapts readily to define a map at the filtered level instead of graded pieces. The proof is then the same as in Theorem~\ref{theoremBFSglobalwithsolid}.
	\end{proof}

	\begin{question}
		Given the results of the previous sections, and in particular Corollary~\ref{corollaryHPsolidallprimespisacdhsheafwithfiltration} and Theorem~\ref{theoremmainconsequenceBFSwithfiltrations}, it is a natural question --to which we do not know the answer-- to ask whether the presheaf
		$$\Big(\prod_{p \in \mathbb{P}} \text{Fil}^\star_{\text{BMS}} \text{TC}(-_{\F_p})\Big)_{\Q} : \text{Sch}^{\text{qcqs,op}} \longrightarrow \text{FilSp}$$
		is a cdh sheaf, where $-_{\F_p}$ again means derived base change from $\Z$ to $\F_p$.
	\end{question}

	\section{\texorpdfstring{$p$}{TEXT}-adic structure of motivic cohomology}
	
	\vspace{-\parindent}
	\hspace{\parindent}
	
	In this section, we give a description of motivic cohomology with finite coefficients in terms of Bhatt--Lurie's syntomic cohomology (Theorem~\ref{theorempadicmotiviccohomologyintermsofsyntomicohomology}).

    \medskip
	
	\subsection{Comparison to étale cohomology}
	
	\vspace{-\parindent}
	\hspace{\parindent}
	
	In this subsection, we construct a natural comparison map, called the Beilinson--Lichtenbaum comparison map, from the motivic cohomology of a scheme to the étale cohomology of its generic fibre (Definition~\ref{definitionBeilinsonLichtenbaumcomparisonmap}). We then use this comparison map to establish a complete description of $\ell$-adic motivic cohomology in terms of étale cohomology (Theorem~\ref{theoremladicmotiviccohomology}). 
	
	We use the following important result of Deligne to construct the Beilinson--Lichtenbaum comparison map.
	
	\begin{theorem}[Deligne, \cite{bhatt_arc-topology_2021}]\label{theoremBMcdhdescentforétalecohomology}
		Let $p$ be a prime number, and $k \geq 1$ be an integer. Then for every integer~\hbox{$i \geq 0$}, the presheaf
		$$R\Gamma_{\emph{ét}}(-[\tfrac{1}{p}],\mu_{p^k}^{\otimes i}) : \emph{dSch}^{\emph{qcqs,op}} \longrightarrow \mathcal{D}(\Z/p^k)$$
		is a cdh sheaf.
	\end{theorem}
	
	\begin{proof}
		By \cite[Theorem~$5.4$]{bhatt_arc-topology_2021}, the presheaf $R\Gamma_{\text{ét}}(-[\tfrac{1}{p}],\mu_{p^k}^{\otimes i})$ is an arc sheaf on qcqs schemes, and the arc topology is finer than the cdh topology.\footnote{More precisely, the arc topology is finer than the $v$ topology, which is finer than the h topology, which is finer than the cdh topology (see \cite{bhatt_arc-topology_2021,elmanto_cdh_2021}), and the cdh topology on derived schemes depends only on the classical locus of these derived schemes.}
	\end{proof}
	
	\begin{definition}[Cdh-local Beilinson--Lichtenbaum comparison map]\label{definitioncdhlocalBeilinsonLichtenbaumcomparisonmap}
		Let $p$ be a prime number, and $k \geq 1$ be an integer. For any integer $i \geq 0$, the {\it cdh-local Beilinson--Lichtenbaum comparison map} is the map
		$$\Z/p^k(i)^{\text{cdh}}(-) \longrightarrow R\Gamma_{\text{ét}}(-[\tfrac{1}{p}], \mu_{p^k}^{\otimes i})$$
		of functors from (the opposite category of) qcqs derived schemes to the derived category $\mathcal{D}(\Z/p^k)$ defined as the composite
		$$\Z/p^k(i)^{\text{cdh}}(-) \longrightarrow \Z/p^k(i)^{\text{cdh}}(-[\tfrac{1}{p}]) \simeq \big(L_{\text{cdh}} \tau^{\leq i} R\Gamma_{\text{ét}}(-,\mu_{p^k}^{\otimes i})\big)(-[\tfrac{1}{p}]) \longrightarrow R\Gamma_{\text{ét}}(-[\tfrac{1}{p}],\mu_{p^k}^{\otimes i})$$
		where the first map is induced by base change from $\Z$ to $\Z[\tfrac{1}{p}]$, the equivalence is Theorem~\ref{theoremBEM}\,(3), and the last map is induced by the natural transformation $\tau^{\leq i} \rightarrow \text{id}$ and Theorem~\ref{theoremBMcdhdescentforétalecohomology}.
	\end{definition}
	
	\begin{definition}[Beilinson--Lichtenbaum comparison map]\label{definitionBeilinsonLichtenbaumcomparisonmap}
		Let $p$ be a prime number, and $k \geq 1$ be an integer. For any integer $i \geq 0$, the {\it Beilinson--Lichtenbaum comparison map} (or {\it motivic-étale comparison map}) is the map
		$$\Z/p^k(i)^{\text{mot}}(-) \longrightarrow R\Gamma_{\text{ét}}(-[\tfrac{1}{p}], \mu_{p^k}^{\otimes i})$$
		of functors from (the opposite category of) qcqs derived schemes to the category $\mathcal{D}(\Z/p^k)$ defined as the composite
		$$\Z/p^k(i)^{\text{mot}}(-) \longrightarrow \Z/p^k(i)^{\text{cdh}}(-) \longrightarrow R\Gamma_{\text{ét}}(-[\tfrac{1}{p}],\mu_{p^k}^{\otimes i})$$
		where the first map is cdh sheafification and the second map is the cdh-local Beilinson--Lichtenbaum comparison map of Definition~\ref{definitioncdhlocalBeilinsonLichtenbaumcomparisonmap}.
	\end{definition}
	
	\begin{remark}\label{remarkcommutativediagramtodefinemotivictosyntomiccomparisonmap}
		Let $p$ be a prime number, $k \geq 1$ be an integer, $R$ be an animated commutative ring, and $R^h_p$ be the $p$-henselisation of $R$. Then for every integer $i \geq 0$, the natural diagram
		$$\begin{tikzcd}
			\Z/p^k(i)^{\text{mot}}(\text{Spec}(R)) \arrow{r} \arrow{d} & \Z/p^k(i)^{\text{cdh}}(\text{Spec}(R)) \ar[d] \ar[r] & R\Gamma_{\text{ét}}\big(\text{Spec}(R[\tfrac{1}{p}]),\mu_{p^k}^{\otimes i}\big) \ar[d] \\
			\Z/p^k(i)^{\text{mot}}(\text{Spec}(R^h_p)) \ar[d] \arrow{r} & \Z/p^k(i)^{\text{cdh}}(\text{Spec}(R^h_p)) \ar[r] \ar[d] & R\Gamma_{\text{ét}}\big(\text{Spec}(R^h_p[\tfrac{1}{p}]),\mu_{p^k}^{\otimes i}\big) \arrow[d,"\text{id}"] \\
			\Z/p^k(i)^{\text{BMS}}(\text{Spec}(R^h_p)) \ar[r] & \big(L_{\text{cdh}} \Z/p^k(i)^{\text{BMS}}\big)(\text{Spec}(R^h_p)) \ar[r] & R\Gamma_{\text{ét}}\big(\text{Spec}(R^h_p[\tfrac{1}{p}]),\mu_{p^k}^{\otimes i}\big),
		\end{tikzcd}$$
		is a commutative diagram in the derived category $\mathcal{D}(\Z/p^k)$, where the top horizontal right map and the middle horizontal right map are given by Definition~\ref{definitioncdhlocalBeilinsonLichtenbaumcomparisonmap}, and the bottom horizontal right map is induced by \cite[Theorem~$8.3.1$]{bhatt_absolute_2022}. This statement is a consequence of the naturality of the constructions, except for the commutativity of the bottom right square, which is explained in \cite[Remark~$5.41$]{bachmann_A^1-invariant_2025}.
	\end{remark}
	
	\begin{theorem}[$\ell$-adic motivic cohomology]\label{theoremladicmotiviccohomology}
		Let $p$ be a prime number, $X$ be a qcqs derived scheme over $\Z[\tfrac{1}{p}]$, and $k \geq 1$ be an integer. Then for every integer \hbox{$i \geq 0$}, the Beilinson--Lichtenbaum comparison for classical motivic cohomology induces a natural equivalence
		$$\Z/p^k(i)^{\emph{mot}}(X) \simeq \big(L_{\emph{cdh}} \tau^{\leq i} R\Gamma_{\emph{ét}}(-,\mu_{p^k}^{\otimes i})\big)(X)$$
		in the derived category $\mathcal{D}(\Z/p^k)$.\footnote{Assuming some regularity assumption on $X$, the cdh sheafification in this equivalence can be replaced with a Nisnevich sheafification: this corresponds to the classical optimal form of the Beilinson--Lichtenbaum comparison proved by Voevodsky \cite{voevodsky_motivic_2011} (see also \cite{geisser_bloch-kato_2001} for the relation with the Bloch--Kato conjecture).}
	\end{theorem}
	
	\begin{proof}
		The syntomic complex $\Z/p^k(i)^{\text{BMS}}$ and its cdh sheafification $L_{\text{cdh}} \Z/p^k(i)^{\text{BMS}}$ vanish on qcqs derived $\Z[\tfrac{1}{p}]$-schemes. In particular, the natural map $$\Z/p^k(i)^{\text{mot}}(X) \longrightarrow \Z/p^k(i)^{\text{cdh}}(X)$$
		is an equivalence in the derived category $\mathcal{D}(\Z/p^k)$ (Proposition~\ref{propositionpadicstructuremain}). The Beilinson--Lichtenbaum comparison for classical motivic cohomology induces a natural equivalence
		$$\Z/p^k(i)^{\text{cdh}}(X) \xlongleftarrow{\sim} \big(L_{\text{cdh}} \tau^{\leq i} R\Gamma_{\text{ét}}(-,\mu_{p^k}^{\otimes i})\big)(X)$$
		in the derived category $\mathcal{D}(\Z/p^k)$ (Theorem~\ref{theoremBEM}\,(3)). The desired equivalence is the composite of the previous two equivalences.
	\end{proof}
	\begin{corollary}\label{corollaryladicmotivcohomologylowdegreesisétalecohomology}
		Let $p$ be a prime number, $X$ be a qcqs derived $\Z[\tfrac{1}{p}]$-scheme, and \hbox{$k \geq 1$} be an integer. Then for every integer $i \geq 0$, there is a natural equivalence
		$$\tau^{\leq i} \Z/p^k(i)^{\emph{mot}}(X) \simeq \tau^{\leq i} R\Gamma_{\emph{ét}}(X,\mu_{p^k}^{\otimes i})$$
		in the derived category $\mathcal{D}(\Z/p^k)$. In particular, for every integer $j \le i$, there is a natural isomorphism of abelian groups $$\emph{H}^j_{\emph{mot}}(X,\Z/p^k(i)) \cong \emph{H}^j_{\emph{ét}}(X,\Z/p^k(i)).$$
	\end{corollary}
	
	\begin{proof}
		The natural map
		$$\tau^{\leq i} \big(L_{\text{cdh}} \tau^{\leq i} R\Gamma_{\text{ét}}(-,\mu_{p^k}^{\otimes i})\big)(X) \longrightarrow \tau^{\leq i} \big(L_{\text{cdh}} R\Gamma_{\text{ét}}(-,\mu_{p^k}^{\otimes i})\big)(X)$$
		is an equivalence in the derived category $\mathcal{D}(\Z/p^k)$. The result is then a consequence of Theorems~\ref{theoremBMcdhdescentforétalecohomology} and \ref{theoremladicmotiviccohomology}.
	\end{proof}
	
	\medskip
	
	\subsection{Comparison to syntomic cohomology}
	
	\vspace{-\parindent}
	\hspace{\parindent}
	
	In this subsection, we study motivic cohomology with finite coefficients. Our main result is a computation of $p$-adic motivic cohomology in terms of syntomic cohomology (Theorem~\ref{theorempadicmotiviccohomologyintermsofsyntomicohomology}).
	
	\begin{notation}[Syntomic cohomology of derived scheme, after Bhatt--Lurie \cite{bhatt_absolute_2022}]\label{notationsyntomiccohomology}
		Let $X$ be a qcqs derived scheme, $p$ be a prime number, and $i \in \Z$ be an integer. We denote by 
		$$\Z_p(i)^{\text{syn}}(X) \in \mathcal{D}(\Z_p)$$
		the {\it syntomic cohomology} of $X$, as defined in \cite[Section~$8.4$]{bhatt_absolute_2022}. For every integer $k \geq 1$, we also denote by $\Z/p^k(i)^{\text{syn}}(X)$ the derived reduction modulo $p^k$ of the previous complex. In particular, the presheaf $\Z/p^k(i)^{\text{syn}}$ is a Zariski sheaf, whose restriction to animated commutative rings $R$ is left Kan extended from smooth $\Z$-algebras, and such that there is, by definition, a natural cartesian square
		$$\begin{tikzcd}
			\Z/p^k(i)^{\text{syn}}(\text{Spec}(R)) \arrow{r} \arrow{d} & R\Gamma_{\text{ét}}(\text{Spec}(R[\frac{1}{p}]),\mu_{p^k}^{\otimes i}) \ar[d] \\
			\Z/p^k(i)^{\text{BMS}}(\text{Spec}(R^h_p)) \arrow{r} & R\Gamma_{\text{ét}}(\text{Spec}(R^h_p[\tfrac{1}{p}]),\mu_{p^k}^{\otimes i})
		\end{tikzcd}$$
		in the derived category $\mathcal{D}(\Z/p^k)$, where $R^h_p$ is the $p$-henselisation of the animated commutative ring~$R$, and the bottom map is the map of \cite[Theorem~$8.3.1$]{bhatt_absolute_2022}.
	\end{notation}
	
	\begin{construction}[Motivic-syntomic comparison map]\label{constructionmotivicsyntomiccomparisonmap}
		Let $p$ be a prime number, and $k \geq 1$ be an integer. For any integer $i \geq 0$, the {\it motivic-syntomic comparison map} is the map
		$$\Z/p^k(i)^{\text{mot}}(-) \longrightarrow \Z/p^k(i)^{\text{syn}}(-)$$
		of functors from (the opposite category of) qcqs derived schemes to the derived category $\mathcal{D}(\Z/p^k)$ defined as the Zariski sheafification of the map on animated commutative rings $R$ induced by the natural commutative diagram
		$$\begin{tikzcd}
			\Z/p^k(i)^{\text{mot}}(\text{Spec}(R)) \ar[d] \ar[r] & R\Gamma_{\text{ét}}\big(\text{Spec}(R[\tfrac{1}{p}]),\mu_{p^k}^{\otimes i}\big) \ar[d] \\
			\Z/p^k(i)^{\text{BMS}}(\text{Spec}(R^h_p)) \ar[r] &  R\Gamma_{\text{ét}}\big(\text{Spec}(R^h_p[\tfrac{1}{p}]),\mu_{p^k}^{\otimes i}\big)
		\end{tikzcd}$$
		of Remark~\ref{remarkcommutativediagramtodefinemotivictosyntomiccomparisonmap}.
	\end{construction}
	
	The following cartesian square, where the bottom horizontal map was described independently in \cite{bachmann_A^1-invariant_2025}, can be seen as an alternative definition of $p$-adic motivic cohomology of qcqs derived schemes (see Corollary~\ref{corollarymainpadicstructureongradeds}).
	
	\begin{proposition}\label{propositioncartesiansquaremotivicsyntomic}
		Let $X$ be a qcqs derived scheme, $p$ be a prime number, and $k \geq 1$ be an integer. Then for every integer $i \geq 0$, the commutative diagram
		$$\begin{tikzcd}
			\Z/p^k(i)^{\emph{mot}}(X) \ar[r] \ar[d] & \Z/p^k(i)^{\emph{syn}}(X) \ar[d]\\
			\Z/p^k(i)^{\emph{cdh}}(X) \ar[r] & \big(L_{\emph{cdh}} \Z/p^k(i)^{\emph{syn}}\big)(X)
		\end{tikzcd}$$
		where the top horizontal map is the motivic-syntomic comparison map of Construction~\ref{constructionmotivicsyntomiccomparisonmap} and the vertical maps are cdh sheafification, is a cartesian square in the derived category $\mathcal{D}(\Z/p^k)$.
	\end{proposition}
	
	\begin{proof}
		By \cite[Remark~$8.4.4$]{bhatt_absolute_2022}, there is a natural fibre sequence
		$$R\Gamma_{\text{ét}}(X,j_!\mu_{p^k}^{\otimes i}) \longrightarrow \Z/p^k(i)^{\text{syn}}(X) \longrightarrow \Z/p^k(i)^{\text{BMS}}(X)$$
		in the derived category $\mathcal{D}(\Z/p^k)$, where $j_! : (X[\tfrac{1}{p}])_{\text{ét}} \rightarrow X_{\text{ét}}$ is the extension by zero functor associated to the open immersion $j : X[\tfrac{1}{p}] \rightarrow X$. The first term of this fibre sequence satisfies arc descent by \cite[Theorem~$5.4$]{bhatt_arc-topology_2021}, hence in particular cdh descent. The result is then a consequence of Corollary~\ref{corollarymainpadicstructureongradeds}.
	\end{proof}

	The following result is a mixed characteristic generalisation of Elmanto--Morrow's fundamental fibre sequence for motivic cohomology of characteristic $p$ schemes (\cite[Corollary~$4.40$]{elmanto_motivic_2023}).
	
	\begin{theorem}[$p$-adic motivic cohomology]\label{theorempadicmotiviccohomologyintermsofsyntomicohomology}
		Let $X$ be a qcqs derived scheme, $p$ be a prime number, and $k \geq 1$ be an integer. Then for every integer $i \geq 0$, there is a natural fibre sequence
		$$\Z/p^k(i)^{\emph{mot}}(X) \longrightarrow \Z/p^k(i)^{\emph{syn}}(X) \longrightarrow \big(L_{\emph{cdh}} \tau^{>i} \Z/p^k(i)^{\emph{syn}}\big)(X)$$
		in the derived category $\mathcal{D}(\Z/p^k)$.\footnote{Assuming some regularity assumption on $X$, the cdh sheafification in this fibre sequence can be replaced with a Nisnevich sheafification, thus recovering the classical form of the Beilinson--Lichtenbaum comparison. We refer to \cite{bouis_beilinson-lichtenbaum_2025} for the most general results in this direction.} In particular, the fibre of the motivic-syntomic comparison map is in degrees at least $i+2$.
	\end{theorem}
	
	\begin{proof}
		By Theorem~\ref{theoremBEM}\,(3), there is a natural equivalence
		$$\Z/p^k(i)^{\text{cdh}}(X) \simeq \big(L_{\text{cdh}} \tau^{\leq i} \Z/p^k(i)^{\text{syn}}\big)(X)$$
		in the derived category $\mathcal{D}(\Z/p^k)$. The result is then a consequence of Proposition~\ref{propositioncartesiansquaremotivicsyntomic}.
	\end{proof}
	
	\begin{corollary}\label{corollarypadiccomparisoninsmalldegreessyntomiccoho}
		Let $X$ be a qcqs derived scheme, $p$ be a prime number, and $k \geq 1$ be an integer. Then for every integer $i \geq 0$, the motivic-syntomic comparison map induces a natural equivalence
		$$\tau^{\leq i} \Z/p^k(i)^{\emph{mot}}(X) \xlongrightarrow{\sim} \tau^{\leq i} \Z/p^k(i)^{\emph{syn}}(X)$$
		in the derived category $\mathcal{D}(\Z/p^k)$. In particular, for every integer $j \le i$, there is a natural isomorphism of abelian groups $$\emph{H}^j_{\emph{mot}}(X,\Z/p^k(i)) \cong \emph{H}^j_{\emph{syn}}(X,\Z/p^k(i)).$$
	\end{corollary}
	
	\begin{proof}
		This is a consequence of Theorem~\ref{theorempadicmotiviccohomologyintermsofsyntomicohomology}.
	\end{proof}

	Algebraic $K$-theory being $\mathbb{A}^1$-invariant on regular schemes, one could expect the classical $\mathbb{A}^1$\nobreakdash-inva\-riant motivic cohomology to be a good theory for general regular schemes, not only schemes that are smooth over a Dedekind domain. However, most of the results on classical motivic cohomology in mixed characteristic are proved only in the smooth case, as consequences of the Gersten conjecture proved by Geisser \cite{geisser_motivic_2004}. This is the case for the Beilinson--Lichtenbaum conjecture, which compares motivic cohomology with finite coefficients to étale cohomology. Combined with Theorem~\ref{theoremintrocohomology}\,$(6)$, the following result extends the analogous result for classical motivic cohomology to the regular case.
	
	\begin{corollary}[Beilinson--Lichtenbaum conjecture for $F$-smooth schemes]\label{corollaryFsmoothnessBeilinsonLichtenbaumcomparison}
		Let $p$ be a prime number, $X$ be a $p$-torsionfree $F$-smooth scheme ({\it e.g.}, a regular scheme flat over $\Z$), and $k \geq 1$ be an integer. Then for every integer $i \geq 0$, the fibre of the Beilinson--Lichtenbaum comparison map
		$$\Z/p^k(i)^{\emph{mot}}(X) \longrightarrow R\Gamma_{\emph{ét}}\big(X[\tfrac{1}{p}],\mu_{p^k}^{\otimes i}\big)$$
		is in degrees at least $i+1$.
	\end{corollary}
	
	\begin{proof}
		By \cite[Theorem~$1.8$]{bhatt_syntomic_2023}, the fibre of the natural map
		$$\Z/p^k(i)^{\text{syn}}(X) \longrightarrow R\Gamma_{\text{ét}}\big(X[\tfrac{1}{p}],\mu_{p^k}^{\otimes i}\big)$$
		is in degrees at least $i+1$. The result is then a consequence of Theorem~\ref{theorempadicmotiviccohomologyintermsofsyntomicohomology}.
	\end{proof}
	
	\begin{corollary}\label{corollarycomparisonpadicmotiviccohomologyofFsmooththingswithmotiviccohoofgenericfibre}
		Let $p$ be a prime number, $X$ be a $p$-torsionfree $F$-smooth scheme, and $k \geq 1$ be an integer. Then the fibre of the natural map
		$$\Z/p^k(i)^{\emph{mot}}(X) \longrightarrow \Z/p^k(i)^{\emph{mot}}\big(X[\tfrac{1}{p}]\big)$$
		is in degrees at least $i+1$.
	\end{corollary}
	
	\begin{proof}
		By construction, the Beilinson--Lichtenbaum comparison map
		$$\Z/p^k(i)^{\text{mot}}(X) \longrightarrow R\Gamma_{\text{ét}}\big(X[\tfrac{1}{p}],\mu_{p^k}^{\otimes i}\big)$$
		naturally factors as the composite
		$$\Z/p^k(i)^{\text{mot}}(X) \longrightarrow \Z/p^k(i)^{\text{mot}}\big(X[\tfrac{1}{p}]\big) \longrightarrow R\Gamma_{\text{ét}}\big(X[\tfrac{1}{p}],\mu_{p^k}^{\otimes i}\big).$$
		The fibre of the second map is in degrees at least $i+1$ by Corollary~\ref{corollaryladicmotivcohomologylowdegreesisétalecohomology}, and the fibre of the composite is in degrees at least $i+1$ by Corollary~\ref{corollaryFsmoothnessBeilinsonLichtenbaumcomparison}, so the fibre of the first map is in degrees at least $i+1$.
	\end{proof}

	\section{Comparison to \texorpdfstring{$\mathbb{A}^1$}{TEXT}-invariant motivic cohomology}\label{sectionA1invariantmotiviccohomology}
	
	\vspace{-\parindent}
	\hspace{\parindent}
	
	The theory of classical motivic cohomology of smooth schemes over a mixed characteristic Dedekind domain \cite{bloch_algebraic_1986,levine_techniques_2001,geisser_motivic_2004}, as a theory of $\mathbb{A}^1$-invariant motivic cohomology, admits a natural generalisation to general qcqs schemes. More precisely, Spitzweck constructed in \cite{spitzweck_commutative_2018}, for every qcqs scheme $X$, an $\mathbb{A}^1$-motivic spectrum $\text{H}\!\Z^{\text{Spi}} \in \text{SH}(\Z)$, which represents Bloch's cycle complexes on smooth $\Z$-schemes, and whose pullback to $\text{SH}(B)$, for $B$ a field or a mixed characteristic Dedekind domain, still represents Bloch's cycle complexes on smooth $B$-schemes. Bachmann then proved in \cite{bachmann_very_2022} that Spitzweck's construction coincides with the zeroth slice of the homotopy $K$-theory motivic spectrum $\text{KGL} \in \text{SH}(\Z)$. Finally, Bachmann--Elmanto--Morrow recently proved in \cite{bachmann_A^1-invariant_2025} that the slice filtration is compatible with arbitrary pullbacks, thus defining a well-behaved $\mathbb{A}^1$-motivic spectrum $\text{H}\!\Z_X \in \text{SH}(X)$ for arbitrary qcqs schemes $X$. The associated $\mathbb{A}^1$-invariant motivic complexes
	$$\Z(i)^{\mathbb{A}^1}(X) \in \mathcal{D}(\Z)$$
	are related to the homotopy $K$-theory $\text{KH}(X)$ by an $\mathbb{A}^1$-invariant Atiyah--Hirzebruch spectral sequence. For our purposes, we will only use that there is a natural map
	$$\Z(i)^{\text{cdh}}(X) \longrightarrow \Z(i)^{\mathbb{A}^1}(X)$$
	which exhibits the target as the $\mathbb{A}^1$-localisation of the source (\cite[Theorem~$9.1$]{bachmann_A^1-invariant_2025}).
	
	In this section, we prove that the $\mathbb{A}^1$-localisation of the motivic complexes $\Z(i)^{\text{mot}}$ recover Bach\-mann--Elmanto--Morrow's $\mathbb{A}^1$-invariant motivic complexes $\Z(i)^{\mathbb{A}^1}$ (Theorem~\ref{theoremA1localmotiviccohomologymain}). This is a motivic refinement of \cite[Theorem~$1.0.1$]{elmanto_thh_2021}, and implies that the motivic complexes $\Z(i)^{\text{mot}}$ recover the classical motivic complexes $\Z(i)^{\text{cla}}$ on smooth schemes over a mixed characteristic Dedekind domain after $\mathbb{A}^1$-localisation (Corollary~\ref{corollary26A1localmotiviccohomologyisclassicalmotiviccohomology}).
	
	\begin{lemma}\label{lemmaA1cdh=A1cdhA1}
		Let $\mathcal{C}$ be a presentable stable $\infty$-category, and $F : \emph{Sch}^{\emph{qcqs,op}} \rightarrow \mathcal{C}$ be a presheaf. Then the natural map
		$$L_{\mathbb{A}^1} L_{\emph{cdh}} F \longrightarrow L_{\mathbb{A}^1} L_{\emph{cdh}} L_{\mathbb{A}^1} F$$
		induced by $\mathbb{A}^1$-localisation is an equivalence of $\mathcal{C}$-valued presheaves.
	\end{lemma}
	
	\begin{proof}
		A filtered colimit of cdh sheaves is a cdh sheaf (Lemma~\ref{lemmacdhsheafificationcommuteswithfilteredcolimits}). In particular, the $\mathbb{A}^1$-localisation of a cdh sheaf is a cdh sheaf, so the natural composites
		$$L_{\mathbb{A}^1} L_{\text{cdh}} F \longrightarrow L_{\mathbb{A}^1} L_{\text{cdh}} L_{\mathbb{A}^1} F \longrightarrow L_{\mathbb{A}^1} L_{\text{cdh}} L_{\mathbb{A}^1} L_{\text{cdh}} F$$
		and
		$$L_{\mathbb{A}^1} L_{\text{cdh}} L_{\mathbb{A}^1} F \longrightarrow L_{\mathbb{A}^1} L_{\text{cdh}} L_{\mathbb{A}^1} L_{\text{cdh}} F \longrightarrow L_{\mathbb{A}^1} L_{\text{cdh}} L_{\mathbb{A}^1} L_{\text{cdh}} L_{\mathbb{A}^1} F$$
		are equivalences in the $\infty$-category $\mathcal{C}$. This implies the desired result.
	\end{proof}
	
	\begin{lemma}\label{lemmaA1finitedeRhamiszerochar0}
		For every integer $i \geq 0$, the $\mathbb{A}^1$-localisation of the presheaf
		$$R\Gamma_{\emph{Zar}}\big(-,\mathbb{L}\Omega^{<i}_{-_{\Q}/\Q}\big) : \emph{dSch}^{\emph{qcqs,op}} \longrightarrow \mathcal{D}(\Q)$$
		is zero.
	\end{lemma}
	
	\begin{proof}
		By Zariski descent, it suffices to prove the result on animated commutative rings. The functor $\mathbb{L}\Omega^{<i}_{-_{\Q}/\Q}$, from animated commutative rings to the derived category $\mathcal{D}(\Q)$, is left Kan extended from polynomial $\Z$-algebras. Equivalently, it commutes with sifted colimits. This property being preserved by $\mathbb{A}^1$-localisation, the functor $L_{\mathbb{A}^1} \mathbb{L}\Omega^{<i}_{-_{\Q}/\Q}$ is also left Kan extended from polynomial $\Z$-algebras. As it is also constant on polynomial $\Z$-algebras, and zero on the zero ring, it is the zero functor.
	\end{proof}
	
	\begin{corollary}
		Let $X$ be a qcqs derived scheme. Then for every integer $i \geq 0$, the natural map
		$$\big(L_{\mathbb{A}^1} \widehat{\mathbb{L}\Omega}^{\geq i}_{-_{\Q}/\Q}\big)(X) \longrightarrow \big(L_{\mathbb{A}^1} \widehat{\mathbb{L}\Omega}_{-_{\Q}/\Q}\big)(X)$$
		is an equivalence in the derived category $\mathcal{D}(\Q)$.
	\end{corollary}
	
	\begin{proof}
		There is a natural fibre sequence
		$$\big(L_{\mathbb{A}^1} \widehat{\mathbb{L}\Omega}^{\geq i}_{-_{\Q}/\Q}\big)(X) \longrightarrow \big(L_{\mathbb{A}^1} \widehat{\mathbb{L}\Omega}_{-_{\Q}/\Q}\big)(X) \longrightarrow \big(L_{\mathbb{A}^1} \mathbb{L}\Omega^{<i}_{-_{\Q}/\Q}\big)$$
		in the derived category $\mathcal{D}(\Q)$. The result is then a consequence of Lemma~\ref{lemmaA1finitedeRhamiszerochar0}.
	\end{proof}
	
	\begin{lemma}\label{lemmaA1syntomiccohoiszero}
		Let $p$ be a prime number. Then for every integer $i \geq 0$, the $\mathbb{A}^1$-localisation of the presheaf
		$$\F_p(i)^{\emph{BMS}}(-) : \emph{dSch}^{\emph{qcqs,op}} \longrightarrow \mathcal{D}(\F_p)$$
		is zero.
	\end{lemma}
	
	\begin{proof}
		The presheaf $\F_p(i)^{\text{BMS}}$ is a Zariski sheaf, and its restriction to animated commutative rings is left Kan extended from polynomial $\Z$-algebras (Corollary~\ref{corollaryBMSsyntomiccohomologyhasquasisyntomicdescentandLKEfrompolynomialZalgebras}). The result then follows by the same argument as in Lemma~\ref{lemmaA1finitedeRhamiszerochar0}.
	\end{proof}
	
	\begin{theorem}\label{theoremA1localmotiviccohomologymain}
		Let $X$ be a qcqs scheme. Then for every integer $i \geq 0$, the natural map
		$$\big(L_{\mathbb{A}^1} \Z(i)^{\emph{mot}}\big)(X) \longrightarrow \big(L_{\mathbb{A}^1} \Z(i)^{\emph{cdh}}\big)(X) \simeq \Z(i)^{\mathbb{A}^1}(X)$$
		is an equivalence in the derived category $\mathcal{D}(\Z)$.
	\end{theorem}
	
	\begin{proof}
		It suffices to prove the result rationally, and modulo $p$ for every prime number $p$. By Lemma~\ref{lemmaA1finitedeRhamiszerochar0}, the object
		$$\big(L_{\mathbb{A}^1} R\Gamma_{\text{Zar}}\big(X,\mathbb{L}\Omega^{<i}_{-_{\Q}/\Q}\big)\big)(X)$$
		is zero in the derived category $\mathcal{D}(\Q)$. Lemma~\ref{lemmaA1cdh=A1cdhA1} then implies that the object
		$$\big(L_{\mathbb{A}^1} R\Gamma_{\text{cdh}}(-,\mathbb{L}\Omega^{<i}_{-_{\Q}/\Q}\big)\big)(X)$$
		is zero in the derived category $\mathcal{D}(\Q)$. In particular, the natural map
		$$\big(L_{\mathbb{A}^1} R\Gamma_{\text{Zar}}\big(-, \mathbb{L}\Omega^{<i}_{-_{\Q}/\Q}\big)\big)(X) \longrightarrow \big(L_{\mathbb{A}^1} R\Gamma_{\text{cdh}}(-,\mathbb{L}\Omega^{<i}_{-_{\Q}/\Q}\big)\big)(X)$$
		is an equivalence in the derived category $\mathcal{D}(\Q)$, which implies the desired result rationally by Corollary~\ref{corollaryrationalmainresultongradedpieces}. Similarly, for every prime number $p$, the syntomic complex $\F_p(i)^{\text{BMS}}$ vanishes after $\mathbb{A}^1$\nobreakdash-locali\-sation by Lemma~\ref{lemmaA1syntomiccohoiszero}, which implies the desired result modulo $p$ by Lemma~\ref{lemmaA1cdh=A1cdhA1} and Corollary~\ref{corollarymainpadicstructureongradeds}.
	\end{proof}
	
	\begin{remark}
		Let $X$ be a qcqs derived scheme. One can prove, using similar arguments and Corollary~\ref{corollarymainconsequenceBFSongradedpieces}, that there is a natural fibre sequence
		$$\big(L_{\mathbb{A}^1} \Z(i)^{\text{TC}}\big)(X) \longrightarrow R\Gamma_{\text{Zar}}\big(X,\widehat{\mathbb{L}\Omega}_{-_{\Q}/\Q}\big) \longrightarrow R\Gamma_{\text{Zar}}\big(X,\widehat{\underline{\mathbb{L}\Omega}}_{-_{\mathbb{A}_f}/\mathbb{A}_f}\big)$$
		in the derived category $\mathcal{D}(\Z)$. Note that this implies Theorem~\ref{theoremA1localmotiviccohomologymain} on qcqs classical schemes, via the cdh descent results \cite[Lemma~$4.5$]{elmanto_motivic_2023} and Corollary~\ref{corollaryHPsolidallprimespisacdhsheafongradedpieces}.
	\end{remark}
	
	\begin{theorem}\label{theorem25A1localmotivicfiltrationisclassicalfiltration}
		Let $X$ be a smooth scheme over a mixed characteristic Dedekind domain. Then the natural map
		$$\emph{Fil}^\star_{\emph{cla}} \emph{K}(X) \longrightarrow \big(L_{\mathbb{A}^1} \emph{Fil}^\star_{\emph{mot}} \emph{K}\big)(X)$$
		is an equivalence of filtered spectra.
	\end{theorem}
	
	\begin{proof}
        By the rational splitting of algebraic $K$-theory induced by Adams operations, this natural map is an equivalence rationally, so it suffices to prove that it is an equivalence modulo every prime number $p$. Let $p$ be a prime number. By \cite[Theorem~$9.1$]{bachmann_A^1-invariant_2025}, the natural map
		$$\text{Fil}^\star_{\text{cla}} \text{K}(X) \longrightarrow \big(L_{\mathbb{A}^1} \text{Fil}^\star_{\text{cdh}} \text{KH}\big)(X)$$
		is an equivalence of filtered spectra, so it suffices to prove that the natural map
		$$\big(L_{\mathbb{A}^1} \text{Fil}^\star_{\text{mot}} \text{K}\big)(X)/p \longrightarrow \big(L_{\mathbb{A}^1} \text{Fil}^\star_{\text{cdh}} \text{KH}\big)(X)/p$$
		is an equivalence of filtered spectra. By Proposition~\ref{propositionpadicstructuremain}, this is equivalent to the fact that the natural map
		$$\big(L_{\mathbb{A}^1} \text{Fil}^\star_{\text{BMS}} \text{TC}(-;\F_p)\big)(X) \longrightarrow \big(L_{\text{A}^1} L_{\text{cdh}} \text{Fil}^\star_{\text{BMS}} \text{TC}(-;\F_p)\big)(X)$$
		is an equivalence of filtered spectra. The filtered spectrum $\big(L_{\mathbb{A}^1} \text{Fil}^\star_{\text{BMS}} \text{TC}(-;\F_p)\big)(X)$ is complete by the connectivity bound of Lemma~\ref{lemmaBMSfiltrationproductallprimesiscomplete}, and its graded pieces are zero by Lemma~\ref{lemmaA1syntomiccohoiszero}, so it is zero. By Lemma~\ref{lemmaA1cdh=A1cdhA1}, the target $\big(L_{\text{A}^1} L_{\text{cdh}} \text{Fil}^\star_{\text{BMS}} \text{TC}(-;\F_p)\big)(X)$ of the previous map is then also zero, and this map is in particular an equivalence.
	\end{proof}

	\begin{corollary}\label{corollary26A1localmotiviccohomologyisclassicalmotiviccohomology}
		Let $X$ be a smooth scheme over a mixed characteristic Dedekind domain. Then for every integer $i \geq 0$, there is a natural equivalence
        $$\Z(i)^{\emph{cla}}(X) \simeq \big(L_{\mathbb{A}^1} \Z(i)^{\emph{mot}}\big)(X)$$
		in the derived category $\mathcal{D}(\Z)$.
	\end{corollary}
	
	\begin{proof}
		This is a consequence of Theorem~\ref{theorem25A1localmotivicfiltrationisclassicalfiltration}.
	\end{proof}

    \begin{remark}\label{remark27A1localisationshouldnotbenecessary}
        While writing these results, we expected that the $\mathbb{A}^1$-localisa\-tion in Theorem~\ref{theorem25A1localmotivicfiltrationisclassicalfiltration} and Corollary~\ref{corollary26A1localmotiviccohomologyisclassicalmotiviccohomology} was not necessary. This is now \cite[Theorem~A]{bouis_beilinson-lichtenbaum_2025}.
	\end{remark}

	\bibliographystyle{alpha}
	
	{\footnotesize
    \bibliography{biblio.bib}
    }

\end{document}